\documentclass[10pt,a4paper]{article}
\usepackage{amsmath,amssymb,amsthm,amscd,mathrsfs}
\newcommand{\N}{\mathbb{N}}
\newcommand{\Z}{\mathbb{Z}}

\newcommand{\R}{\mathbb{R}}
\newcommand{\C}{\mathbb{C}} 
\newcommand{\diag}{\mathrm{diag}}
\newcommand{\g}{\mathfrak{g}}
\newcommand{\ah}{\mathfrak{a}}
\newcommand{\n}{\mathfrak{n}}
\newcommand{\ka}{\mathfrak{k}}

\newcommand{\Ad}{\mathrm{Ad}}

\DeclareMathOperator{\Hom}{Hom}

\DeclareMathOperator{\Indd}{-Ind}

\DeclareMathOperator{\im}{Im}

\newtheorem{thm}{Theorem}[section]
\newtheorem{df}[thm]{Definition}
\newtheorem{prop}[thm]{Proposition}
\newtheorem{lem}[thm]{Lemma}
\newtheorem{cor}[thm]{Corollary}
\newtheorem{rem}[thm]{Remark}

\numberwithin{equation}{section}

\title{Generalized Whittaker functions for degenerate principal series of $GL(4,\R)$}
\date{}
\author{Kazuki Hiroe\thanks{E-mail:\texttt{kazuki@ms.u-tokyo.ac.jp}}}

\begin{document}
\maketitle
\begin{center}
Graduate School of Mathematical Sciences, The University of Tokyo, 3-8-1 Komaba, Meguro-ku, Tokyo, 153-8914, Japan.
\end{center}

\begin{abstract}
We give a characterization of a generalized Whittaker model of a degenerate principal series representation of $GL(n,\R)$ as the kernel of some differential operators. By this characterization, we investigate some examples on $GL(4,\R)$. We obtain the dimensions of the generalized Whittaker models and give their basis in terms of hypergeometric functions of one and two variables. We show the multiplicity one of the generalized Whittaker models by using the theory of hypergeometric functions.
\end{abstract}
\section{Introduction}
Our interest in this paper is generalized Whittaker models of degenerate principal representations. There are many studies about them for admissible (non-degenerate) characters of unipotent radicals of parabolic subgroups (for example \cite{Has},\cite{M},\cite{W-1},\cite{W-2},\cite{Y-3}). In the case of degenerate principal series representations, Yamashita gives existence theorem and multiplicity formula for wide classes of generalized Whittaker models, i.e., generalized Whittaker models for generalized Gelfand-Graev representations in \cite{Y-3}. However, their techniques strongly depend on the admissibility of the characters of the unipotent subgroups. On the other hand, if we regard the Whittaker models as an analogue of Fourier coefficients of an automoprhic form at a cusp, we often meet the necessity to consider non-admissible characters. For example, Terras gives an expansion of the Epstein zeta function in terms of modified Bessel functions \cite{T-1}. Non-admissible characters play important roles there. The Epstein zeta function corresponds to the degenerate principal series representation of $GL(n,\R)$ induced from the character of the maximal parabolic subgroup $P_{1,n}$ which fixes the unit vector $e_{n}=(0,\ldots,0,1)$ (cf. \cite{O-H}). Hence the Fourier coefficients given by Terras can be seen as the generalized Whittaker functions for this representation. It seems, however, widely open about the problem of the existence and the multiplicity formula of the generalized Whittaker models for non-admissible characters of unipotent subgroups, though recently the solutions of such problems for degenerate characters of maximal unipotent subgroup are  obtained by Abe and Oshima independently \cite{A},\cite{O-4} for degenerate principal series representations with generic parameters.  In this paper, we will give some examples about this problems in the case of $GL(4,\R)$.  

The other purpose of this paper is to give an expression of the generalized Whittaker function as a hypergeometric function of several variables. According to the recent work of Oshima and Shimeno \cite{O-S}, Whittaker functions can be seen as the confluent hypergeometric functions obtained from Heckman-Opdam hypergeometric functions. The similarities of Whittaker functions with spherical functions were already pointed out by Hashizume in \cite{Has-sem}. Also there are various explicit pictures of Whittaker functions as hypergeometric functions of several variables given by Oda and his collaborators (see \cite{H-I-O} for the reference). We will show that the generalized Whittaker functions of the degenerate principal series representations of $GL(4,\R)$ are written by modified Bessel functions and Horn's hypergeometric function $\mathrm{H}_{10}$ in this paper. There is a similar work on $SL(3,\R)$ in \cite{I-O}.

Let us explain the contents of this paper. In Section 2 and Section 3, we will give a characterization of a Whittaker model of a degenerate principal series representation of $GL(n,\R)$ as the kernel of a family of differential operators. More precisely, let $G=GL(n,\R)$ and consider an Iwasawa decomposition $G=KAN$ where $K=O(n)$, $A$ is the group of diagonal matrices with positive real entries, and $N$ is the group of lower triangular matrices with 1s on diagonal entries. We take an increasing sequence of positive integers ending at $n$, i.e., $\Theta=\{n_{1},\ldots,n_{L}\}$ with $0<n_{1}<n_{2}<\cdots <n_{L}=n$. Then let $P_{\Theta}$ be the parabolic subgroup corresponding to the sequence $\Theta$ and take the Langlands decomposition $P_{\Theta}=M_{\Theta}A_{\Theta}N_{\Theta}$. For a linear mapping  $\lambda\in \Hom_{\R}(\mathrm{Lie}(A_{\Theta}),\C)$, we can consider an induced representation $C^{\infty}\Indd_{P_{\Theta}}^{G}(1_{M_{\Theta}}\otimes e^{\lambda}\otimes 1_{N_{\Theta}})$. We call this representation a degenerate principal series representation. The underlying representation space of this is
\begin{multline*}
C^{\infty}(G/P_{\Theta},\lambda)=\\
\{f\in C^{\infty}(G)\mid f(gp)=(1_{M_{\Theta}}\otimes e^{\lambda}\otimes 1_{N_{\Theta}})(p^{-1})f(g),\ g\in G,p\in P_{\Theta}\}
\end{multline*}
and the action of $G$ is defined by the left translation. Then we consider an ideal of $U(\g)$ the universal enveloping algebra of $\g_{\C}=\mathfrak{gl}(n,\C)$ such that $I_{\Theta}(\lambda)=\{X\in U(\g)\mid R_{X}f=0, \ f\in C^{\infty}(G/P_{\Theta},\lambda) \}$. Here $R_{X}$ is the right derivation by $X\in U(\g)$. We consider $\lambda$ as an element of $\Hom_{\R}(\mathrm{Lie}(A),\C)$ and we assume it is regular and dominant. Under this assumption, the generators of the ideal $I_{\Theta}(\lambda)$ is known by Oshima (cf. Theorem \ref{generator1}). 	Let $U$ be a closed subgroup $N$ and $(\eta,V_{\eta})$ an irreducible unitary representation of $U$. We consider the space $C^{\infty}_{\eta}(U\backslash G)=\{f\colon G\rightarrow V_{\eta}^{\infty}\,\text{smooth}\mid f(ug)=\eta(u)f(g),u\in U,g\in G\}$ where $V_{\eta}^{\infty}$ is the space of smooth vectors in $V_{\eta}$. Let $X_{\Theta,\lambda}$ be the Harish-Chandra module of $C^{\infty}(G/P_{\Theta},\lambda)$ and $X_{\Theta,\lambda}^{*}$ its dual Harish-Chandra module, i.e., the space of $K$-finite vectors of $\Hom_{\C}(X_{\Theta,\lambda},\C)$. The generalized Whittaker model is the image of $X_{\Theta,\lambda}$ by the element of $\Hom_{\g_{\C},K}(X_{\Theta,\lambda},C^{\infty}_{\eta}(U\backslash G))$. Then we can show the following characterization theorem of the generalized Whittaker model.
\begin{thm}[see Theorem \ref{Yamashita}]
Assume that $X_{\Theta,\lambda}^{*}$ is irreducible. We take a nonzero $K$-fixed vector $f_{0}$ in $X_{\Theta,\lambda}^{*}$. Then the following mapping  
\[
\begin{array}{ccc}
\tilde{\Phi}\colon\Hom_{\g_{\C},K}(X_{\Theta,\lambda}^{*},C^{\infty}_{\eta}(U\backslash G))&\xrightarrow[]{\sim}&C^{\infty}_{\eta}(U\backslash G/K, I_{\Theta}(\lambda))\\
W&\longmapsto&W(f_{0})(g)
\end{array}
\] 
is a linear isomorphism.
Here 

\begin{multline*}
C^{\infty}_{\eta}(U\backslash G/K, I_{\Theta}(\lambda))\\
=\{f\colon G\rightarrow V_{\eta}^{\infty}\ \text{smooth}\,|\, f(ngk)=\eta(n)f(g), g\in G, n\in U,\ k\in K\\
\text{and}\ R_{X}f(g)=0,\ X\in I_{\Theta}(\lambda)\}.
\end{multline*}

\end{thm}
This theorem is an analogue of the theorem for the generalized Whittaker models of unitary highest weight modules obtained by Yamashita \cite{Y-1}, \cite{Y-2}.

From Section \ref{GL4}, we consider examples on degenerate principal series representations of $GL(4,\R)$ induced from characters of maximal parabolic subgroups $P_{1,4}$ and $P_{2,4}$ by using above theorem. We will determine the dimension of $\Hom_{\g_{\C},K}(X_{\Theta,\lambda}^{*},C^{\infty}_{\eta}(U\backslash G))$ and the basis of $C^{\infty}_{\eta}(U\backslash G/K, I_{\Theta}(\lambda))$. Let us explain more detailed settings. As the space $C^{\infty}_{\eta}(U\backslash G)$, we consider the space defined as follows. 
\begin{enumerate}
\item the group $U$ is a closed subgroup of $N$ and $\eta$ is its unitary character,
\item the unitary induced representation $L^{2}\Indd _{U}^{N}\eta$ is an irreducible unitary representation of $N$.
\end{enumerate} 

We will classify the $G$-equivalent classes of these $C^{\infty}_{\eta}(U\backslash G)$ in Section \ref{classify-G} (see Proposition \ref{classify}). 

There is a linear isomorphism from the space $C^{\infty}_{\eta}(U\backslash G/K)$ onto $C^{\infty}(U\backslash N\times A)$ (cf. Lemma \ref{cross}). In Section \ref{diff}, we will see how the the action of the Lie algebra $\g$ is written as differential operators on $C^{\infty}(U\backslash N\times A)$.

Our main results are in Section \ref{explicit}. In this section, we will give the dimensions of $C^{\infty}_{\eta}(U\backslash G/K, I_{\Theta}(\lambda))$ and the basis of them as the functions on $C^{\infty}(U\backslash N\times A)$. These basis can be written in terms of modified Bessel functions and Horn's hypergeometric functions $\mathrm{H}_{10}$ (see Theorem \ref{thmI-1}, Theorem \ref{thmI-2}, Theorem \ref{thmII-1}, Theorem \ref{thmII-2}, Theorem \ref{thmII-3}, Theorem \ref{thmIII}). According to these theorems, we can conclude the following. For the degenerate principal series representation induced from a character of $P_{1,4}$, the multiplicity one theorem is true for the generalized Whittaker models for characters of the closed proper subgroups of $N$. On the other hand, for the degenerate principal series representation induced from a character of $P_{2,4}$, the multiplicity one theorem is no longer true. This fact seems to correspond to the result of Terras in \cite{T-2}. In that paper, she could determine only the nonsingular terms in the Fourier expansion of Eisenstein series corresponding to this degenerate principal series representation (Theorem 1 in \cite{T-2}). And she could not say anything about degenerate terms in Fourier expansion. The multiplicities of the generalized Whittaker models corresponding to these Fourier coefficients seems to be one of the cause of this phenomenon.

Finally, we give some facts about Horn's hypergeometric functions in Appendix. 
\subsection*{Acknowledgement}
The author expresses his hearty thanks to Professor Takayuki Oda for his constant encouragement and his many advice. Also he is most grateful to Professor Toshio Oshima who showed his unpublished note about Whittaker functions and gave him many useful suggestions. Finally he would like to thank Noriyuki Abe and the discussions with him are very helpful to understand the representation theoretic aspect of this problem.  
\section{Spherical degenerate principal series representations of $GL(n,\R)$}\label{sec1}

In this section, we study degenerate principal series representations of $GL(n,\R)$ and their annihilators in the enveloping algebra $U(\mathfrak{gl}(n,\C))$. T.Oshima shows that the image of a degenerate principal series representation by the Poisson transform is characterized by the kernel of the annihilator of the degenerate principal series representation \cite{O-1}. He also give the explicit generators for its annihilator \cite{O-2}, \cite{O-3}. We will give a brief review of these results here.
\subsection{Spherical degenerate principal series representations of $GL(n,\R)$.}\label{defofprinc}
Let $G=GL(n,\R)$. We denote its Lie algebra by $\g=\mathfrak{gl}(n,\R)$. We take the Iwasawa decomposition of $G$ as $G=KAN$, where $K=O(n)$ and  $A$ is the group of $n\times n$ diagonal matrices with positive real entries and $N$ is the group of lower triangular matrices with 1s on the diagonal entries. Let $E_{ij}$ be the matrix with 1 in the $(i,j)$-entry and $0$ elsewhere.
We introduce a non-degenerate bilinear form on $\g_{\C}=\mathfrak{gl}(n,\C)=M(n,\C)$ by
\[
\langle X,Y \rangle=\mathrm{tr}(XY)\ \text{for}\ X,Y\in \g_{\C}.
\]
By this bilinear form, we identify $\g_{\C}$ with its dual space $\g_{\C}^{*}$. The dual basis $\{E_{ij}^{*}\}$ of $\{E_{ij}\}$ is given by $E_{ij}^{*}=E_{ji}$. For simplicity, we write $e_{i}=E_{ii}^{*}$.

We consider the Lie algebra  
\[\mathfrak{a}=\{\sum_{i=1}^{n}a_iE_{ii}\ |\ a_i\in\R,\ i=1,\ldots,n\},
\]  
of $A$. Then the root system of  $\mathfrak{a}$ in $\g$ is 
\[
\bigtriangleup(\g,\ah)=\{e_i-e_j\ |\ 1\le i\neq j\le n\}.
\]
We put $\alpha_i=e_{i+1}-e_{i}$ for $i=1,\ldots,n-1$ and fix a simple system of  $\bigtriangleup(\g,\ah)$ as 
\[\Pi(\g,\ah)=\{\alpha_1,\ldots,\alpha_{n-1}\}.
\]
Then the positive system of $\bigtriangleup(\g,\ah)$ associated to $\Pi(\g,\ah)$ is $\bigtriangleup^{+}(\g,\ah)=\{e_i-e_j\ |\ 1\le j<i \le n\}$. Then the Lie algebra $\mathfrak{n}$ of $N$ is written by 
\begin{align*}\mathfrak{n}&=\sum_{\alpha\in \bigtriangleup^{+}(\g,\ah)}\g_{\alpha}\\
&=\sum_{i>j}\R E_{ij}
\end{align*}
 where $\g_{\alpha}=\{X\in \g\ |\ \textrm{ad}(H)X=\alpha (H)X\ \text{for}\ H\in\ah\}.$
On the other hand, let $\overline{N}$ be the group of upper triangular matrices with 1s on the diagonal entries.  Then the Lie algebra $\overline{\mathfrak{n}}$ of $\overline{N}$ is also written by
\begin{align*}
\overline{\mathfrak{n}}=&\sum_{\alpha\in \bigtriangleup^{+}(\g,\ah)}\g_{-\alpha}\\
&=\sum_{i<j}\R E_{ij}.
\end{align*}
Let $\Theta=\{n_{1},\ldots,n_{L}\}$ be a sequence of strictly increasing positive integers ending at $n$, i.e., $(0=n_0<)n_1<n_2<\cdots <n_{L}(=n)$. For this $\Theta$, the associated standard parabolic subgroup $P_{\Theta}$ can be defined as follows. Let 
\[
\ah_{\Theta}=\{\sum_{k=1}^{L}a_{k}\sum_{i=n_{k-1}+1}^{n_k}E_{ii}\ |\ a_{k}\in\R, k=1,\ldots,L\}.
\]
Let $L_{\Theta}$ be the centralizer of $\ah_{\Theta}$ in $G$, i.e.,
\[
L_{\Theta}=\left\{l=
\begin{pmatrix}
l_1&&&\\
&l_2&&\\
&&\ddots&\\
&&&l_{L}
\end{pmatrix}
\ |\ 
l_{i}\in GL(n_i-n_{i-1},\R)
\right\}
\]
and $\mathfrak{l}_{\Theta}$ its Lie algebra which is the centralizer of $\ah_{\Theta}$ in $\g$.
We put 
\[
\mathfrak{n}_{\Theta}=\sum_{\iota_{\Theta}(i)>\iota_{\Theta}(j)}\!\R E_{ij}
\]
where 
\begin{equation}\label{iota}
\iota_{\Theta}(\nu)=i\ \text{if}\ n_{i-1}<\nu\le n_i\ \text{for}\ i=1,\ldots,L.
\end{equation}
The corresponding analytic subgroup of $G$ is $N_{\Theta}=\exp\mathfrak{n}_{\Theta}$, i.e.,
\begin{multline*}
N_{\Theta}=\\
\left\{
n=\begin{pmatrix}
I_{n'_1}&&&&\\
N_{21}&I_{n'_2}&&&\\
N_{31}&N_{32}&I_{n'_3}&&\\
\vdots&\vdots&\vdots&\ddots&\\
N_{L1}&N_{L2}&N_{L3}&\cdots&I_{n'_L}
\end{pmatrix}
\ |\ N_{ij}\in M(n'_i,n'_j;\R), n'_i=n_{i}-n_{i-1}
\right\}.
\end{multline*}
Here $I_m$ denotes the identity matrix of size $m$ and $M(k,l;\R)$ denotes the space of matrices of size $k\times l$ with components in $\R$. We also define $\overline{\mathfrak{n}}_{\Theta}=\sum_{\iota_{\Theta}(i)<\iota_{\Theta}(j)}\R E_{ij}$ and $\overline{N}_{\Theta}=\exp \overline{\mathfrak{n}}_{\Theta}$ as well.

Then we define the parabolic subgroup $P_{\Theta}=L_{\Theta}N_{\Theta}$, i.e.,
\[
P_{\Theta}=\left\{
p=\begin{pmatrix}
g_{1}&&&\\
*&g_2&&\\
\vdots&\vdots&\ddots&\\
*&*&\cdots&g_{L}
\end{pmatrix}
\in GL(n,\R)
\ |\ g_{i}\in GL(n_{i}-n_{i-1},\R)
\right\}.
\]
Its Lie algebra is written as $\mathfrak{p}_{\Theta}=\mathfrak{l}_{\Theta}\oplus \mathfrak{n}_{\Theta}$.

For $(\lambda_1,\lambda_2,\cdots,\lambda_{L})\in \C^{L}$, we define a $1$-dimensional representation of $P_{\Theta}$, $\lambda\colon P_{\Theta}\rightarrow \C^{\times}$ as follows,
\[
\lambda(p)=|\det(g_1)|^{\lambda_1}|\det(g_2)|^{\lambda_2}\cdots |\det(g_L)|^{\lambda_L},\text{for}\ p\in P_{\Theta}.
\]
 We define a spherical degenerate principal series representation of $G$, denote by $\pi_{\Theta,\lambda}=C^{\infty}\!\textrm{-ind}_{P_{\Theta}}^{G}(\lambda)$. The underlying representation space is 
\[
C^{\infty}(G/P_{\Theta};\lambda)=\{
\phi\in C^{\infty}(G)\ |\ \phi(gp)= \lambda(p)\phi(g),\ g\in G, p\in P_{\Theta}\}
\]
where $C^{\infty}(G)$ is the space of $C^{\infty}$-functions on $G$. The action of $G$ on this space is defined by the left translation, $\pi_{\Theta,\lambda}(g)\phi(x)=\phi(g^{-1}x)$ for $g\in G$ and $\phi\in C^{\infty}(G/P_{\Theta};\lambda)$. 

We consider the annihilator of  $C^{\infty}(G/P_{\Theta};\lambda)$ in the universal enveloping algebra. Let $U(\g)$ be the universal enveloping algebra of $\g_{\C}$. We can see $U(\g)$ as the ring of left $G$-invariant differential operators on $C^{\infty}(G)$ by the natural extent ion of the differentiation of the right translation,
\[ 
 R_{X}(f)(g)=\frac{d}{dt}f(g\exp(tX))|_{t=0}.
\]
for $X\in\g, f\in C^{\infty}(G)$.
The representation of $U(\g)$ on $C^{\infty}(G/P_{\Theta};\lambda)$ is defined by the differentiation of $\pi_{\Theta,\lambda}$, i.e., for $X\in \g,\ \phi \in C^{\infty}(G/P_{\Theta};\lambda)$, $\pi_{\Theta,\lambda}(X)\phi(x)=\frac{d}{dt}\phi(\exp{(-tX)}x)|_{t=0}$. 

Let $L_{g}$ and $R_{g}$ be the left and right translations by $g\in G$ respectively, i.e., $L_{g}f(x)=f(g^{-1}x)$ and $R_{g}f(x)=f(xg)$ for $f\in C^{\infty}(G)$.

\begin{df}
We define the annihilator of $C^{\infty}(G/P_{\Theta};\lambda)$ in $U(\g)$ by
\[
\mathrm{Ann}_{U(\g)}(\pi_{\Theta,\lambda})=\{X\in U(\g)\ ;\ \pi_{\Theta,\lambda}(X)\phi(x)=0,\ \text{for all}\ \phi \in C^{\infty}(G/P_{\Theta};\lambda) 
\}. 
\]
This is a two-sided ideal of $U(\g)$.

\end{df}
We consider an antiautomorphism $\iota$ of $U(\g)$ defined by $\iota(XY)=(-Y)(-X)$ for $X,Y\in \g_{\C}$. We denote the differentiation of $\lambda$ by $d\lambda\colon \mathfrak{p}_{\Theta}\rightarrow \C$. Although the proposition below is a well-known fact, we give a proof for the completeness of the paper.
\begin{prop}\label{ann}
The annihilator of $\pi_{\Theta,\lambda}$ is written as follows,
\[
\iota(\mathrm{Ann}_{U(\g)}(\pi_{\Theta,\lambda})) = \displaystyle\bigcap _{g\in G} \mathrm{Ad}(g)J_{\Theta}(d\lambda).
\]
Here
\[
J_{\Theta}(d\lambda)=\sum_{X\in\mathfrak{p}_{\Theta}}U(\g)(X-d\lambda(X))
\]
is a left ideal of $U(\g)$.
\end{prop}
\begin{proof}
For $X\in \mathfrak{p}_{\Theta}$ and $f\in C^{\infty}(G/P_{\Theta};\lambda)$, we have
\begin{equation}\label{migisayo}
\begin{aligned}
R_{X}f(g)&=\frac{d}{dt}f(g\cdot \exp{tX})| _{t=0}\\
&=\frac{d}{dt}\lambda(\exp{tX})|_{t=0} f(g).
\end{aligned}
\end{equation}
This implies $R_{X}f=0$ for $X\in J_{\Theta}(d\lambda)$. We recall the equation $L_{X}f(g)=R_{\text{Ad}(g^{-1})\iota(X)}f(g)$, $X\in U(\g)$. Since $X\in \bigcap_{g\in G}\text{Ad}(g)J_{\Theta}(d\lambda)$ implies $\Ad(g)X\in J_{\Theta}(d\lambda)$, we have
\[
L_{\iota(X)}f(g)=R_{\text{Ad}(g^{-1})X}f(g)=0,
\]  
for $X\in \bigcap_{g\in G}\text{Ad}(g)J_{\Theta}(d\lambda)$. Hence we have the inclusion $\bigcap_{g\in G}\mathrm{Ad}(g)J_{\Theta}(d\lambda)\subset \iota(\mathrm{Ann}_{U(\g)}(\pi_{\Theta,\lambda}))$.

On the other hand, we take $X\in \mathrm{Ann}_{U(\g)}(\pi_{\Theta,\lambda})$, and put $X_{g_{0}}=\text{Ad}(g_{0}^{-1})\iota(X)$ for $g_{0}\in G$. Then we have 
\begin{equation}\label{gizero}
\begin{aligned}
R_{X_{g_{0}}}f(g)&=L_{g_{0}g^{-1}}(R_{X_{g_{0}}}f)(g_{0})=R_{X_{g_{0}}}(L_{g_{0}g^{-1}}f)(g_{0})\\
&=L_{X}(L_{g_{0}g^{-1}}f)(g_{0})=L_{X}(\pi_{\Theta,\lambda}(g_{0}g^{-1})f)(g_{0})=0
\end{aligned}
\end{equation}
for $f\in C^{\infty}(G/P_{\Theta};\lambda)$. By the decomposition $\g=\overline{\mathfrak{n}}_{\Theta}\oplus \mathfrak{p}_{\Theta}$ and the Poincar\'e-Birkoff-Witt theorem, we have \[
U(\g)=U(\overline{\mathfrak{n}}_{\Theta})\oplus J_{\Theta}(d\lambda)
\] where $U(\overline{\mathfrak{n}}_{\Theta})$ is the universal enveloping algebra of $\overline{\mathfrak{n}}_{\Theta}\otimes_{\R}\C$. Hence there exist $Y\in U(\overline{\mathfrak{n}}_{\Theta})$ and $Z\in J_{\Theta}(d\lambda)$ such that $X_{g_{0}}=Y+Z$. By the equation (\ref{migisayo}), we have $R_{Z}f(g)=0$ for $g\in G$ and $f\in C^{\infty}(G/P_{\Theta};\lambda)$. Therefore the equation (\ref{gizero}) lead us that $0=R_{X_{g_{0}}}f(g)=R_{Y}f(g)$. We will show $Y=0$. Then this means $X_{g_{0}}\in J_{\Theta}(d\lambda)$. Therefore we can show the inclusion $\bigcap_{g\in G}\mathrm{Ad}(g)J_{\Theta}(d\lambda)\subset \iota(\mathrm{Ann}_{U(\g)}(\pi_{\Theta,\lambda}))$.

We consider the space of compactly supported $C^{\infty}$-functions on $\overline{N}_{\Theta}$, and denote it by $C^{\infty}_{o}(\overline{N}_{\Theta})$. For $g\in \overline{N}_{\Theta}P_{\Theta}$, we take $\bar{n}(g)\in \overline{N}_{\Theta}$ and $p(g)\in P_{\Theta}$ such that $g=\bar{n}(g)p(g)$. Then we have an injection
\[
\begin{array}{ccc}
C^{\infty}_{o}(\overline{N}_{\Theta})&\longrightarrow &C^{\infty}(G/P_{\Theta};\lambda)\\
f&\longmapsto&\begin{cases}\lambda(p(g))f(\bar{n}(g))&\text{if}\ g\in \overline{N}_{\Theta}P_{\Theta}\\
0&\text{otherwise}.
\end{cases}
\end{array}
 \]
By this injection, we can consider $C_{o}^{\infty}(\overline{N}_{\Theta})\subset C^{\infty}(G/P_{\Theta};\lambda)$. Therefore if we recall that $R_{Y}f(g)=0$ for $g\in G, f\in C^{\infty}(G/P_{\Theta};\lambda)$, we have
\[
R_{Y}f(\bar{n})=0\ \text{for}\ \bar{n}\in \overline{N}_{\Theta},f\in C_{o}^{\infty}(\overline{N}_{\Theta}).
\]
For any $\psi \in C^{\infty}(\overline{N}_{\Theta})$ and $\bar{n}\in \overline{N}_{\Theta}$, there exists $f\in C_{o}^{\infty}(\overline{N}_{\Theta})$ such that $\psi=f$ on some neighbourhood of $\bar{n}$ in $\overline{N}_{\Theta}$. Hence this implies 
\[
R_{Y}\psi(\bar{n})=0\ \text{for}\ \bar{n}\in \overline{N}_{\Theta},\psi\in C^{\infty}(\overline{N}_{\Theta}).
\]
Therefore $Y\in U(\overline{\mathfrak{n}}_{\Theta})$ must be 0, because of the fact that $U(\overline{\mathfrak{n}}_{\Theta})$ is identified with the ring of all left invariant differential operators in $\overline{N}_{\Theta}$. Hence $X_{g_{0}}\in J_{\Theta}(\lambda)$ for any $g_{0}\in G$. This complete the proof.
\end{proof}

\subsection{The Poisson transform for the degenerate principal series representation.}\label{Poisson sec}
For simplicity we write $I_{\Theta}(\lambda)=\bigcap _{g\in G} \mathrm{Ad}(g)J_{\Theta}(d\lambda)$. Then we will see that this ideal $I_{\Theta}(\lambda)$ characterizes the image of the Poisson transform from the degenerate principal series. To explain this fact, we should extend the representation space of the degenerate principal series to the space of hyperfunctions on $G$.

The space $\mathcal{B}(G)$ of hyperfunctions on $G$ is a left $G$-module by the left translation $G\times \mathcal{B}(G)\ni (g,f(x))\mapsto f(g^{-1}x)$. We take a parabolic subgroup $P_{\Theta}$ of $G$. Also we take a character $\lambda\colon P_{\Theta}\rightarrow \C^{\times}$ for $(\lambda_1,\cdots,\lambda_{L})\in \C^{L}$. Then we can define a $G$-submodule 
\[
\mathcal{B}(G/P_{\Theta};\lambda)=\{f\in \mathcal{B}(G)\ |\ f(xp)=\lambda(p)f(x)\ \text{for}\ p\in P_{\Theta}\},
\]
as in Section \ref{defofprinc}.
Let $M=\{k\in K\mid kak^{-1}=a,\, a\in A\}$, then we can define the minimal parabolic subgroup $P_{o}=P_{\{1,2,\cdots,n\}}=MAN$. We define a character of $P_{o}$ by
\[
\begin{array}{cccc}
\lambda_{\Theta}\colon &P_{o}&\longrightarrow &\C^{\times}\\
&man&\longmapsto &\displaystyle\prod_{i=1}^{L}\prod_{j=n_i+1}^{n_{i+1}}a_{j}^{\lambda_{i}},
\end{array}
\]
for $m\in M,a\in A,n\in N$.
Now we introduce the Poisson transform of $\mathcal{B}(G/P_{o};\lambda_{\Theta})$. 
\begin{df}
The Poisson transform is a $G$-homomorphism
\[
\begin{array}{cccc}
\mathcal{P}^{\lambda}\colon &\mathcal{B}(G/P_{o};\lambda_{\Theta})&\longrightarrow &\mathcal{B}(G/K)\\
&f&\longmapsto &F(x)=\displaystyle\int_{K}f(xk)\,dk,\quad x\in G.
\end{array}
\]
Here $dk$ is the normalized Haar measure on $K$ so that $\int_{K}dk=1$.
\end{df}
We define a character of the center $Z(\g)$ of $U(\g)$. Let $d\lambda_{\Theta}\colon \mathrm{Lie}(P_{o})\rightarrow \C$ be the differentiation of $\lambda_{\Theta}$. By the restriction to $\ah\subset \mathrm{Lie}(P_{o})$, we can regard $d\lambda_{\Theta}\in \ah_{\C}^{*}$.
 Let $\omega$ be a projection map from $U(\g)$ to the symmetric algebra $S(\ah)$ of $\ah_{\C}=\ah\otimes_{\R}\C$ along the decomposition
\[
U(\g)=S(\ah)\oplus (\overline{\mathfrak{n}}U(\g)+U(\g)\mathfrak{n}).
\]
It is known that $\omega$ is an algebra homomorphism from $Z(\g)$ into $S(\ah)$. We can identify the symmetric algebra $S(\ah)$ with the algebra of polynomials on $\ah^{*}_{\C}$. Hence if we consider the evaluation of $\omega(\cdot)\in S(\ah)$ at $d\lambda_{\Theta}$, we obtain a character of $Z(\g)$ as follows
\[
\chi_{\lambda}\colon Z(\g)\ni X \longmapsto \omega(X)(d\lambda_{\Theta})\in \C.
\]

We define a subspace of $C^{\infty}(G/K)$ by
\[
C^{\infty}(G/K;\mathcal{M}_{\lambda})=\{f\in C^{\infty}(G/K)\ |\ R_{X}f=\chi_{\lambda}(X)f\ \text{for}\ X\in Z(\g)\}.
\]
We put
\[
e(\lambda_{\Theta})=\prod_{\alpha\in \bigtriangleup^{+}(\g,\ah)}\Gamma\left(\frac{1}{4}(3+\frac{2\langle \lambda_{\Theta},\alpha\rangle}{\langle \alpha,\alpha\rangle})\right)^{-1}\Gamma\left(\frac{1}{4}(1+\frac{2\langle \lambda_{\Theta},\alpha\rangle}{\langle \alpha,\alpha\rangle})\right)^{-1}.
\]
The following theorem is known as  Helgason's conjecture \cite{Hel-1}.
\begin{thm}[\cite{K-et.al}]\label{poisson}
The Poisson transform $\mathcal{P}^{\lambda}$ gives $G$-isomorphism 
\[
\mathcal{B}(G/P_{o};\lambda_{\Theta})\cong C^{\infty}(G/K;\mathcal{M}_{\lambda})\]
if and only if $e(\lambda_{\Theta})\neq 0$ . 
\end{thm}

We can also define the Poisson transform for the subspace $\mathcal{B}(G/P_{\Theta};\lambda)$ of $\mathcal{B}(G/P_{o};\lambda_{\Theta})$. We discuss the characterization of the image of $\mathcal{B}(G/P_{\Theta};\lambda)$. We consider the subspace
\[
C^{\infty}(G/K;I_{\Theta}(\lambda))
=\{f\in C^{\infty}(G/K)\ |\ R_{X}f=0\ \text{for}\ X\in I_{\Theta}(\lambda)\}
\]
of $C^{\infty}(G/K;\mathcal{M}_{\lambda})$.
\begin{rem}
We can easily show that 
\[
I_{\Theta}(\lambda)\supset \sum_{D\in Z(\g)}U(\g)(D-\omega(D)(\lambda_{\Theta}))\]
(cf. Remark 4.3 in \cite{O-3}). Hence actually $C^{\infty}(G/K;I_{\Theta}(\lambda))$ is a subspace of $C^{\infty}(G/K;\mathcal{M}_{\lambda})$.
\end{rem}

We assume

\begin{center}
$\lambda_{\Theta}+\rho\in \ah_{\C}^{*}$ is regular and dominant.
\end{center}
Here $\rho=\frac{1}{2}\mathrm{tr}(\mathrm{ad}|_{\mathfrak{n}})\in \ah_{\C}^{*}$, i.e.,
\[
\rho=\frac{1}{2}\sum_{1\le i<j\le n}(e_j-e_i)=\sum_{i=1}^{n}(i-\frac{n+1}{2})e_{i}.
\] 
This assumption is equivalent to 
\[
\frac{2\langle \lambda_{\Theta},\alpha\rangle}{\langle\alpha,\alpha\rangle}\notin\{0,-1,-2,\cdots\}\ \text{for}\ \alpha\in \bigtriangleup^{+}(\g,\ah),
\]
i.e.,
\[
(\lambda_{j}+\nu_{j})-(\lambda_{i}+\nu_{i})\notin \{0,-1,-2,\cdots\} 
\]
for $i<j$ and $\nu_{k}$ are integers which satisfy $n_{k-1}+1\le \nu_{k}\le n_{k}\ (k=i\ \text{or}\ j)$. We keep this assumption all through the remaining of this paper.

\begin{thm}[Oshima. Theorem 5.1 in \cite{O-3}]\label{Poisson}
Under the above assumption, the Poisson tranform 
\[
\begin{array}{ccc}
\mathcal{P}_{\Theta}^{\lambda}\colon \mathcal{B}(G/P_{\Theta};\lambda)&\longrightarrow &C^{\infty}(G/K,I_{\Theta}(\lambda))\\
f&\longmapsto &F(x)=\displaystyle\int_{K}f(xk)\,dk,\quad x\in G.
\end{array}
\]
is a $G$-isomorphism.
\end{thm}

\subsection{The explicit generators of $I_{\Theta}(\lambda)$.}
In the previous section, we see that a degenerate principal series representation has a realization on the subspace of $C^{\infty}(G/K)$ which is the kernel of the annihilator ideal $I_{\Theta}(\lambda)$.
In \cite{O-1},\cite{O-2} and \cite{O-3}, T.Oshima obtained several good generator systems of $I_{\Theta}(\lambda)$. We introduce one of his generators here.

We denote the space of $n\times n$ matrices with entries in $U(\g)$ by $M(n;U(\g))$. For $\mathbb{E}=(E_{ij})_{ij}\in M(n;U(\g))$, we define elements in $Z(\g)$ by
\[
\Delta_{k}=\mathrm{tr}(\mathbb{E}^{k}),\quad \text{for}\ k=1,\ldots,n.
\]
Then it is known that $Z(\g)\cong \C[\bigtriangleup_1,\ldots,\bigtriangleup_n]$ as $\C$-algebras. 
\begin{thm}[Oshima. Corollary 4.6 in \cite{O-3}]\label{generator1}
Asuume $\lambda_{\Theta}+\rho\in \ah_{\C}^{*}$ is regular and dominant. Then we have
\[
I_{\Theta}(\lambda)=\sum_{i=1}^{n}\sum_{j=1}^{n}U(\g)\prod_{k=1}^{L}(\mathbb{E}-\lambda_{k}-n_{k-1})_{ij}+\sum_{k=1}^{L-1}U(\g)(\bigtriangleup_k-\chi_{\lambda}(\bigtriangleup_k)).
\]
\end{thm} 

\section{Generalized Whittaker models}\label{sec2}
The generalized Whittaker model is the main theme of this paper. We will give a characterization of the space of the generalized Whittaker models of a degenerate principal series $\pi_{\Theta,\lambda}$ as the kernel of $I_{\Theta}(\lambda)$. This is an analogy of Yamashita's method in the case of irreducible highest weight modules \cite{Y-2}. The substantial part of his method is that the maximal globalization (in the sense of W.Schmid \cite{Sch}) of highest weight modules is given by the kernel of a certain differential operator. The corresponding theorem for the degenerate principal series is obtained in Theorem \ref{Poisson} in Section \ref{Poisson sec}. Moreover thanks to Theorem \ref{generator1}, we know explicit structures of these differential operators. Hence we can carry out the explicit calculations about the space of the generalized Whittaker models.

Let $V_{K}$ be the space of $K$-finite vectors for a continuous representation of $G$ on a complete Hausdorff locally convex space $V$. Let $X_{\Theta,\lambda}$ be $C^{\infty}(G/P_{\Theta};\lambda)_{K}$. This becomes a $(\g_{\C},K)$-module, i.e., the $\g_{\C}$-action is the differentiation of $\pi_{\Theta,\lambda}$ and the $K$-action is the restriction of $\pi_{\Theta,\lambda}$, furthermore the actions of $\g_{\C}$ and $K$ are compatible. Also $X_{\Theta,\lambda}$ is a Harish-Chandra module, i.e., finitely generated as a $U(\g)$-module and with finite $K$-multiplicities.
\subsection{Maximal globalization}
For the Harish-Chandara module $X_{\Theta,\lambda}$, let us consider its dual Harish-Chandra module $X_{\Theta,\lambda^{*}}$. Here the character $\lambda^{*}$ of $P_{\Theta}$ is defined by 
\[
\lambda^{*}=-\bar{\lambda}-2\rho_{\Theta}=(n-n_{0}-n_{1}-\bar{\lambda},\ldots,n-n_{L-1}-n_{L}-\bar{\lambda}),
\]
where $\rho_{\Theta}=\frac{1}{2}\mathrm{tr}(\mathrm{ad}|_{\mathfrak{n}_{\Theta}})\in \ah_{\C}^{*}$, i.e.,
\[
\rho_{\Theta}=\sum_{i=1}^{L}\frac{n_{i-1}+n_{i}-n}{2}\sum_{j=n_{i-1}}^{n_{i}}e_{j}.
\]
Actually, if we consider the pairing $\langle\ ,\ \rangle\colon C^{\infty}(G/P_{\Theta};\lambda)\times C^{\infty}(G/P_{\Theta};\lambda^{*})\rightarrow \C$ defined by
\[
\langle f,g\rangle=\int_{K}f(k)\overline{g(k)}\,dk
\] 
for $(f,g)\in C^{\infty}(G/P_{\Theta};\lambda)\times C^{\infty}(G/P_{\Theta};\lambda^{*})$, this is a $G$-equivariant non-degenerate sesquilinear pairing. By this pairing, the Harish-Chandra module $X_{\Theta,\lambda^{*}}=C^{\infty}(G/P_{\Theta,\lambda^{*}})_{K}$ can be identified with the dual Harish-Chandra module $(X_{\Theta,\lambda})^{*}$,i.e., all $K$-finite vectors in $\Hom_{\C}(X_{\Theta,\lambda},\C)_{K}$. Here $K$ acts on $\Hom_{\C}(X_{\Theta,\lambda},\C)$ by $k\cdot I(v)=I(\pi_{\Theta,\lambda}(k^{-1})v)$ for $I\in \Hom_{\C}(X_{\Theta,\lambda},\C)_{K}$ and $v\in X_{\Theta,\lambda}$. 

We can consider the natural $(\g_{\C}\times\g_{\C},K\times K)$-bimodule structures on $X_{\Theta,\lambda}\otimes X_{\Theta,\lambda^{*}}$ and $C^{\infty}(G)$. For $X_{1},X_{2}\in\g_{\C}$ and $k_{1},k_{2}\in K$, we put
\begin{align*}
(X_{1},X_{2})(f\otimes f^{*})=&\pi_{\Theta,\lambda^{*}}(X_{1})f\otimes f^{*}+f\otimes\pi_{\Theta,\lambda^{*}}(X_{2})f^{*},\\
(k_{1},k_{2})(f\otimes f^{*})=&\pi_{\Theta,\lambda}(k_{1})f\otimes\pi_{\Theta,\lambda^{*}}(k_{2})f^{*}
\end{align*}
for $f\in X_{\Theta,\lambda}$ and $f^{*}\in X_{\Theta,\lambda^{*}}$. Also we define
\begin{align*}
(X_{1},X_{2})g=&L_{X_{1}}g+R_{X_{2}}g,\\
(k_{1},k_{2})g=&L_{k_{1}}R_{k_{2}}g
\end{align*}
for $g\in C^{\infty}(G)$. Then we introduce the matrix coefficient map (cf. \cite{C-M}) from $(\g_{\C}\times \g_{\C},K\times K)$-bimodule $X_{\Theta,\lambda}\otimes X_{\Theta,\lambda^{*}}$ to $C^{\infty}(G)$ so that,
\begin{enumerate}
\item the map $c\colon X_{\Theta,\lambda}\otimes X_{\Theta,\lambda^{*}}\rightarrow C^{\infty}(G)$ is a $(\g_{\C}\times \g_{\C},K\times K)$-bimodule homomorphism,
\item for any $f\in X_{\Theta,\lambda}$ and $f^{*}\in X_{\Theta,\lambda^{*}}$, the evaluation at the origin $e\in G$ becomes 
\[
c(f\otimes f^{*})(e)=\langle f,f^{*}\rangle.
\]
\end{enumerate}
It is known that this matrix coefficient map is uniquely determined (cf. Theorem 8.7 in \cite{C-M}). 

If we consider the restriction of Poisson transform $\mathcal{P}_{\Theta}^{\lambda}$ on $X_{\Theta,\lambda}$, Theorem \ref{Poisson} gives us the $(\g_{\C},K)$-isomorphism
\[
\mathcal{P}_{\Theta}^{\lambda}\colon X_{\Theta,\lambda}\xrightarrow[]{\sim} C^{\infty}(G/K;I_{\Theta}(\lambda))_{K}.
\]
\begin{lem}\label{mtcf}
Take the $K$-fixed vector $f_{0}\in X_{\Theta,\lambda^{*}}$ such that $f_{0}|_{K}\equiv 1$. Then the restriction of the Poisson transform on $X_{\Theta,\lambda}$ is a matrix coefficient of an element of $X_{\Theta,\lambda}$ with $f_{0}\in X_{\Theta,\lambda^{*}}$, i.e., 
\[\mathcal{P}_{\Theta}^{\lambda}(f)=c(f\otimes f_{0}).
\]

\end{lem}
\begin{proof}
By the pairing of $C^{\infty}(G/P_{\Theta};\lambda)\times C^{\infty}(G/P_{\Theta};\lambda^{*})$ defined above, we can define a map $X_{\Theta,\lambda}\otimes X_{\Theta,\lambda^{*}}\rightarrow C^{\infty}(G)$ as follows,
\[
f\otimes f^{*}\longmapsto \langle \pi_{\Theta,\lambda}(g^{-1})f,f^{*}\rangle=\int_{K}f(gk)\overline{f^{*}(k)}\,dk,
\]
for $f\in X_{\Theta,\lambda}$ and $f^{*}\in X_{\Theta,\lambda^{*}}$. This map satisfies the conditions of the matrix coefficient map. Hence for $f\in X_{\Theta,\lambda}$, we have
\begin{align*}
\mathcal{P}_{\Theta}^{\lambda}(f)(g)&=\int_{K}f(gk)\,dk\\
&=\int_{K}f(gk)\overline{f_{0}(k)}\,dk\\
&=\int_{K}(\pi_{\Theta,\lambda}(g^{-1})f)(k)\overline{f_{0}(k)}\,dk\\
&=\langle \pi_{\Theta,\lambda}(g^{-1})f,f_{0}\rangle\\
&=c(f\otimes f_{0})(g),
\end{align*}
by the uniqueness of the matrix coefficient map.
\end{proof}

Let us consider the space of $(\g_{\C},K)$-homomorphisms of $X_{\Theta,\lambda^{*}}$ into $C^{\infty}(G)$,
\[
\Hom_{\g_{\C},K}(X_{\Theta,\lambda^{*}},C^{\infty}(G)).
\]
Here we regard $C^{\infty}(G)$ as a $(\g_{\C},K)$-module by the right translation. Moreover, this space of $(\g_{\C},K)$-homomorphisms inherits a Fr\'echet topology and a continuous $G$-action from $C^{\infty}(G)$. More precisely, we define a semi-norm on this space as follows. The space $C^{\infty}(G)$ is a Fr\'echet space of uniformly convergence on compact sets for functions on $G$ and their derivatives. Let $\{|\cdot|_{\alpha}\}_{\alpha\in \Lambda}$ be a family of countable many semi-norms on $C^{\infty}(G)$ which defines the Fr\'echet topology on $C^{\infty}(G)$ where $\Lambda$ is the index set. Take a semi-norm $|\cdot |_{\alpha}\in \{|\cdot|_{\alpha}\}_{\alpha\in \Lambda}$ and $v\in X_{\Theta,\lambda^{*}}$. Then we define a real-valued function $|\cdot|_{\alpha,v}\colon \Hom_{\g_{\C},K}(X_{\Theta,\lambda^{*}},C^{\infty}(G))\rightarrow \R_{\ge 0}$ by 
\[
|I|_{\alpha,v}=|I(v)|_{\alpha}
\]
for $I\in \Hom_{\g_{\C},K}(X_{\Theta,\lambda_{*}},C^{\infty}(G))$. We can see that the function $|\cdot|_{\alpha,v}$ defines a semi-norm on $\Hom_{\g_{\C},K}(X_{\Theta,\lambda_{*}},C^{\infty}(G))$ for $\alpha\in \Lambda$ and $v\in X_{\Theta,\lambda^{*}}$. 
\begin{lem}
Let $\{v_{n}\}$ be a countable vector space basis of the Harish-Chandra module $X_{\Theta,\lambda^{*}}$. Then the family of semi-norms $\{|\cdot|_{\alpha,v_{m}}\}_{\alpha\in \Lambda,v_{m}\in\{v_{n}\}}$ defines a Fr\'echet topology on $\Hom_{\g_{\C},K}(X_{\Theta,\lambda_{*}},C^{\infty}(G))$. 
\end{lem}
\begin{proof}
Take  $I\in \Hom_{\g_{\C},K}(X_{\Theta,\lambda^{*}},C^{\infty}(G))$, and we assume $|I|_{\alpha,v_{m}}=0$ for any $|\cdot|_{\alpha,v_{m}}\in \{|\cdot|_{\alpha,v_{m}}\}_{\alpha\in \Lambda,v_{m}\in\{v_{n}\}}$. Since $\{v_{n}\}$ is a basis of $X_{\Theta,\lambda^{*}}$, it follows that $|I(v)|_{\alpha}=0$ for any $v\in X_{\Theta,\lambda^{*}}$ and $\alpha \in \Lambda$. This means $I(v)=0$ for any $v\in X_{\Theta,\lambda^{*}}$ because $C^{\infty}(G)$ is the Hausdorff space. Thus we have $I=0$. This implies that $\Hom_{\g_{\C},K}(X_{\Theta,\lambda_{*}},C^{\infty}(G))$ is a Hausdorff space as well. Since $\{|\cdot|_{\alpha,v_{m}}\}_{\alpha\in \Lambda,v_{m}\in\{v_{n}\}}$ consists of countable many semi-norms, the space $\Hom_{\g_{\C},K}(X_{\Theta,\lambda_{*}},C^{\infty}(G))$ is metrizable by this family of semi-norms. Finally we need to check the completeness. Suppose that there exists a Cauchy sequence $\{I_{k}\}$ of the elements of $\Hom_{\g_{\C},K}(X_{\Theta,\lambda_{*}},C^{\infty}(G))$, i.e.,
\[
|I_{k}-I_{l}|_{\alpha,v_{m}}\rightarrow 0\quad\text{for}\ k,l\rightarrow \infty,
\]
for any $|\cdot|_{\alpha,v_{m}}\in \{|\cdot|_{\alpha,v_{m}}\}_{\alpha\in \Lambda,v_{m}\in\{v_{n}\}}$. Then this implies that for any $v\in X_{\Theta,\lambda^{*}}$, the sequence $\{I_{k}(v)\}\subset C^{\infty}(G)$ is a Cauchy sequence. Hence there exists $\lim_{k\rightarrow \infty}I_{k}(v)\in C^{\infty}(G)$. We define the map $\tilde{I}\colon X_{\Theta,\lambda^{*}}\rightarrow C^{\infty}(G)$ by $\tilde{I}(v)=\lim_{k\rightarrow \infty}I_{k}(v)$ for $v\in X_{\Theta,\lambda^{*}}$. Then we will show that $\tilde{I}$ is the element in $\Hom_{\g_{\C},K}(X_{\Theta,\lambda_{*}},C^{\infty}(G)).$ For any $Z\in \g_{\C}$, $v\in X_{\Theta,\lambda^{*}}$ and $\alpha\in \Lambda$, we have 
\begin{align*}
|\tilde{I}&(\pi_{\Theta,\lambda^{*}}(Z)v)-R_{Z}\tilde{I}(v)|_{\alpha}\\
&=|\tilde{I}(\pi_{\Theta,\lambda^{*}}(Z)v)-I_{k}(\pi_{\Theta,\lambda^{*}}(Z)v)+R_{Z}I_{k}(v)-R_{Z}\tilde{I}(v)|_{\alpha}\\
&\le |\tilde{I}(\pi_{\Theta,\lambda^{*}}(Z)v)-I_{k}(\pi_{\Theta,\lambda^{*}}(Z)v)|_{\alpha}+|R_{Z}I_{k}(v)-R_{Z}\tilde{I}(v)|_{\alpha}\\
&\rightarrow 0\quad \text{for}\ k\rightarrow\infty.
\end{align*}
Thus $\tilde{I}(\pi_{\Theta,\lambda^{*}}(Z)v)=R_{Z}\tilde{I}(v)$ for any $Z\in \g_{\C}$ and $v\in X_{\Theta,\lambda^{*}}$. Hence $\tilde{I}$ is a $\g_{\C}$-homomorphism. Similarly we can show that $\tilde{I}$ is a $K$-homomorphism and a linear map. Hence we could show that $\tilde{I}\in \Hom_{\g_{\C},K}(X_{\Theta,\lambda},C^{\infty}(G))$. By the construction of $\tilde{I}$, we can see that $I_{k}\rightarrow \tilde{I}$ $(k\rightarrow\infty)$ in $\Hom_{\g_{\C},K}(X_{\Theta,\lambda^{*}},C^{\infty}(G)).$ This proves the lemma.
\end{proof}
We could define a Fr\'echet topology on $\Hom_{\g_{\C},K}(X_{\Theta,\lambda_{*}},C^{\infty}(G)).$ Then a continuous $G$-action on this space is defined by left translation on $C^{\infty}(G).$ Hence this space defines a continuous Fr\'echet representation of $G$. This is called the maximal globalization of the Harish-Chandra module $X_{\Theta,\lambda}$ (cf. \cite{Sch} and \cite{K-S}).
\begin{lem}\label{contiphi}
We assume that $X_{\Theta,\lambda^{*}}$ is irreducible. Take the $K$-fixed vector $f_{0}\in X_{\Theta,\lambda^{*}}$ such that $f_{0}|_{K}\equiv 1$. We consider a mapping
\[
\begin{array}{ccc}
\Phi\colon \Hom_{\g_{\C},K}(X_{\Theta,\lambda^{*}},C^{\infty}(G))&\longrightarrow &C^{\infty}(G)\\
I&\longmapsto &I(f_{0})(g)\quad (g\in G).
\end{array}
\]
Then $\Phi$ is a continuous mapping. Moreover for any semi-norm $|\cdot|_{\alpha,v_{m}}\in \{|\cdot|_{\alpha,v_{m}}\}_{\alpha\in \Lambda,v_{m}\in\{v_{n}\}}$ on $\Hom_{\g_{\C},K}(X_{\Theta,\lambda^{*}},C^{\infty}(G))$, there exists a continuous semi-norm $\mu_{\alpha,v_{m}}$ on $C^{\infty}(G)$ such that
\[
\mu_{\alpha,v_{m}}(\Phi(I))=|I|_{\alpha,v_{m}},
\]
for $I\in\Hom_{\g_{\C},K}(X_{\Theta,\lambda^{*}},C^{\infty}(G))$.
\end{lem}
\begin{proof}
For any semi-norm $|\cdot|_{\alpha}$ on $C^{\infty}(G)$, there exists a continuous semi-norm $|\cdot|_{\alpha,f_{0}}$ on $\Hom_{\g_{\C},K}(X_{\Theta,\lambda^{*}},C^{\infty}(G))$ such that
\[
|I|_{\alpha,f_{0}}=|I(f_{0})|_{\alpha}=|\Phi(I)|_{\alpha}
\]
for $I\in\Hom_{\g_{\C},K}(X_{\Theta,\lambda^{*}},C^{\infty}(G))$. Hence $\Phi$ is the continuous. Conversely, we take a semi-norm $|\cdot|_{\alpha,v_{m}}$ on $\Hom_{\g_{\C},K}(X_{\Theta,\lambda^{*}},C^{\infty}(G))$ for $\alpha \in \Lambda$ and $v_{m}\in \{v_{n}\}$. Since $X_{\Theta,\lambda^{*}}$ is the irreducible Harish-Chandra module, there exists an element $X\in U(\g)$ such that $\pi_{\Theta,\lambda^{*}}(X)f_{0}=v_{m}$. Then we have 
\[
|I|_{\alpha,v_{m}}=|I(v_{m})|_{\alpha}=|I(\pi_{\Theta,\lambda^{*}}(X)f_{0})|_{\alpha}=|R_{X}I(f_{0})|_{\alpha}.
\]
If we recall that $U(\g)$ can be identified with the ring of left invariant differential operators on $C^{\infty}(G)$, then $\mu_{\alpha,v_{m}}(f)=|R_{X}f|_{\alpha}$ defines a continuous semi-norm on $C^{\infty}(G)$. This proves the lemma.
\end{proof}
The following proposition is essentially obtained in \cite{Sch} and \cite{K-S} (more direct proof is also given by H. Yamashita \cite{Y-2}). We give a proof for the completeness of the paper.
\begin{prop} \label{maxglob}
We assume that $X_{\Theta,\lambda^{*}}$ is irreducible. Take the $K$-fixed vector $f_{0}\in X_{\Theta,\lambda^{*}}$ such that $f_{0}|_{K}\equiv 1$. Then we have a following topological $G$-isomorphism,\[
\begin{array}{ccc}
\Phi\colon \Hom_{\g_{\C},K}(X_{\Theta,\lambda^{*}},C^{\infty}(G))&\xrightarrow[]{\sim} &C^{\infty}(G/K;I_{\Theta}(\lambda))\\
I&\longmapsto &I(f_{0})(g)\quad (g\in G).
\end{array}
\]
Here $C^{\infty}(G/K;I_{\Theta}(\lambda))$ has the Fr\'echet topology as the closed subspace of $C^{\infty}(G)$.
\end{prop}
\begin{proof}
We can immediately see that $\Phi$ preserves the action of $G$ by definition. First we show that $\Phi$ is well-defined. Take a $K$-finite element $I\in \Hom_{\g_{\C},K}(X_{\Theta,\lambda^{*}},C^{\infty}(G))_{K}$. Then by the evaluation at the origin $e\in G$, we can regard $I(\cdot)(e)$ as the element of $X_{\Theta,\lambda}\cong (X_{\Theta,\lambda^{*}})^{*}$. Since $I(f_{0})(g)\in C^{\infty}(G)$ is $K$-finite and $Z(\g)$-finite, it is a real analytic function on $G$. Let $W$ be an sufficiently small open neighbourhood of $0$ in $\g$. Then we have the Taylor expansion at the origin $e\in G$,
\begin{align*}
I(f_{0})(\exp X)&=\sum_{n=0}^{\infty}\frac{1}{n!}R_{X^{n}}(I(f_{0}))(e)\\
&=\sum_{n=0}^{\infty}\frac{1}{n!}\langle I(\cdot)(e),\pi_{\Theta,\lambda^{*}}(X^{n})f_{0}\rangle\\
&=\sum_{n=0}^{\infty}\frac{1}{n!}c(I(\cdot)(e)\otimes \pi_{\Theta,\lambda^{*}}(X^{n})f_{0})(e)\\
&=\sum_{n=0}^{\infty}\frac{1}{n!}R_{X^{n}}c(I(\cdot)(e)\otimes f_{0})(e)\\
&=c(I(\cdot)(e)\otimes f_{0})(\exp X)
\end{align*}
for $X\in W$. We can extend this equality for the identity component $G^{o}$ of $G$ because both functions are real analytic. And the fact $G=G^{o}\cdot K$ implies $I(f_{0})(g)=c(I(\cdot)(e)\otimes f_{0})(g)$ for all $g\in G$. Hence by Theorem \ref{Poisson} and Lemma \ref{mtcf}, we have the inclusion
\[
\Phi(\Hom_{\g_{\C},K}(X_{\Theta,\lambda^{*}},C^{\infty}(G))_{K})\subset C^{\infty}(G/K;I_{\Theta}(\lambda)).
\]
We recall that for a continuous representation of $G$ on a locally convex complete space $V$, the space of $K$-finite vectors $V_{K}$ is dense in $V$ (for example, Lemma 1.9, Ch.IV in \cite{Hel-2}).  Since $\Phi$ is a continuous mapping by Lemma \ref{contiphi}, we have
\begin{align*}
&\Phi(\Hom_{\g_{\C},K}(X_{\Theta,\lambda^{*}},C^{\infty}(G)))=\Phi(\mathrm{Cl}(\Hom_{\g_{\C},K}(X_{\Theta,\lambda^{*}},C^{\infty}(G))_{K})\\
&\subset \mathrm{Cl}(\Phi(\Hom_{\g_{\C},K}(X_{\Theta,\lambda^{*}},C^{\infty}(G))_{K}))\subset \mathrm{Cl}(C^{\infty}(G/K;I_{\Theta}(\lambda)))\\
&=C^{\infty}(G/K;I_{\Theta}(\lambda)).
\end{align*}
Here $\mathrm{Cl}(\cdot)$ is the closure. Hence $\Phi$ is well-defined. Next we prove $\Phi$ is the bijective map. Since $X_{\Theta,\lambda^{*}}$ is irreducible, the map $\Phi$ is injective. We will prove that $\Phi$ is surjective.

For any $F(g)\in C^{\infty}(G/K;I_{\Theta}(\lambda))_{K}$, there exist $h\in X_{\Theta,\lambda}$ and $F(g)=c(h\otimes f_{0})(g)$ $(g\in G)$ by Theorem \ref{Poisson} and Lemma \ref{mtcf}. We define an element of $\Hom_{\g_{\C},K}(X_{\Theta,\lambda},C^{\infty}(G))$ so that
\[
I_{h}(v)(g)=c(h\otimes v)(g),
\]
for $v\in X_{\Theta,\lambda^{*}}$. Then we can see that $\Phi(I_{h})(g)=I_{h}(f_{0})(g)=c(h\otimes f_{0})(g)=F(g).$ Hence we have an inclusion $C^{\infty}(G/K;I_{\Theta}(\lambda))_{K}\subset \Phi(\Hom_{\g_{\C},K}(X_{\Theta,\lambda},C^{\infty}(G)))$. Because $C^{\infty}(G/K;I_{\Theta}(\lambda))_{K}$ is a dense subspace of $C^{\infty}(G/K;I_{\Theta}(\lambda))$, for any $f\in C^{\infty}(G/K;I_{\Theta}(\lambda))$ we can choose a convergent sequence $f_{n}\rightarrow f$  $(n\rightarrow \infty)$ where $f_{n}\in C^{\infty}(G/K;I_{\Theta}(\lambda))_{K}$ for $n\in\N$. The above inclusion shows that there exist $I_{n}\in \Hom_{\g_{\C},K}(X_{\Theta,\lambda},C^{\infty}(G))$ such that $\Phi(I_{n})=f_{n}$. From the second assertion in Lemma \ref{contiphi}, the sequence $\{I_{n}\}$ is a Cauchy sequence in $\Hom_{\g_{\C},K}(X_{\Theta,\lambda},C^{\infty}(G))$. Since $\Hom_{\g_{\C},K}(X_{\Theta,\lambda},C^{\infty}(G))$ is a Fr\'echet space, i.e., complete space, there exist $I\in \Hom_{\g_{\C},K}(X_{\Theta,\lambda},C^{\infty}(G))$ such that $I_{n}\rightarrow I$ $(n\rightarrow \infty)$. Thus we have $\Phi(I)=(f)$ by the continuity of $\Phi$. This shows that $\Phi$ is a surjective map. The open mapping theorem leads that $\Phi$ is a homeomorphism.
\end{proof}

\subsection{Generalized Whittaker models}
We define a generalized Whittaker model for $X_{\Theta,\lambda}$. Let us fix a closed subgroup $U$ of $N$. We take an irreducible unitary representation $\eta$ of $U$ on a Hilbert space $V_{\eta}$. Let $V_{\eta}^{\infty}$ be the space of $C^{\infty}$-vectors in $V_{\eta}$. Let us consider the space $C^{\infty}_{\eta}(U\backslash G)=\{f\colon G\rightarrow V_{\eta}^{\infty}\ \text{smooth}\,|\, f(ng)=\eta(n)f(g), g\in G, n\in U\}$. This becomes a $G$-module by the right translation.
\begin{df}
We consider the following intertwining space
\[
\Hom_{\g_{\C},K}(X_{\Theta,\lambda^{*}},C^{\infty}_{\eta}(U\backslash G)).
\]
We call images of $X_{\Theta,\lambda^{*}}$ by these $(\g_{\C},K)$-homomorphisms generalized Whittaker models of $X_{\Theta,\lambda^{*}}$.
\end{df} 

\begin{thm}\label{Yamashita}
Assume that $X_{\Theta,\lambda^{*}}$ is irreducible. We take the $K$-fixed vector in $X_{\Theta,\lambda^{*}}$ such that $f_{0}|_{K}\equiv 1$. Then the following mapping  
\[
\begin{array}{ccc}
\tilde{\Phi}\colon\Hom_{\g_{\C},K}(X_{\Theta,\lambda^{*}},C^{\infty}_{\eta}(U\backslash G))&\xrightarrow[]{\sim}&C^{\infty}_{\eta}(U\backslash G/K; I_{\Theta}(\lambda))\\
W&\longmapsto&W(f_{0})(g)
\end{array}
\] 
is a linear isomorphism.
Here 
\begin{multline*}
C^{\infty}_{\eta}(U\backslash G/K; I_{\Theta}(\lambda))\\
=\{f\colon G\rightarrow V_{\eta}^{\infty}\ \text{smooth}\,|\, f(ngk)=\eta(n)f(g), g\in G, n\in U,\ k\in K\\
\text{and}\ R_{X}f(g)=0,\ X\in I_{\Theta}(\lambda)\}.
\end{multline*}
\end{thm}
\begin{proof}
Fix a nonzero element $\xi\in V_{\eta}$. Then we consider the linear mapping
\[
T\colon C^{\infty}_{\eta}(U\backslash G)\in f\longmapsto \langle \xi,f(g)\rangle_{\eta} \in C^{\infty}(G),
\]
which commutes with $G$ and $\g_{\C}$ actions from the right where $\langle\ ,\ \rangle_{\eta}$ is an inner product on $V_{\eta}$. Since $(\eta,V_{\eta})$ is an irreducible unitary representation of $U$, this mapping $T$ is injective. In fact, if $T(f)\equiv0$ for $f\in C^{\infty}_{\eta}(U\backslash G)$, then we have
\[0=T(f)(ng)=\langle\xi,f(ng)\rangle=\langle\xi,\eta(n)f(g)\rangle=\langle\eta(n^{-1})\xi,f(g)\rangle,
\]
for any $n\in U$ and $g\in G$. Since $V_{\eta}$ is irreducible, this implies $f\equiv 0$. By this map $T$, we also have an injective map
\[
\begin{array}{cccc}
\tilde{T}&\colon\Hom_{(\g_{\C},K)}(X_{\Theta,\lambda^{*}},C^{\infty}_{\eta}(U\backslash G))&\longrightarrow &\Hom_{(\g_{\C},K)}(X_{\Theta,\lambda^{*}},C^{\infty}(G))\\
&W&\longmapsto&T\circ W.
\end{array}
\]
For any $W\in \Hom_{(\g_{\C},K)}(X_{\Theta,\lambda^{*}},C^{\infty}(U\backslash G,\eta))$, we have $T(\tilde{\Phi}(W))=T(W(f_{0}))=T\circ W(f_{0})=\tilde{T}(W)(f_{0})=\Phi(\tilde{T}(W))$.
Hence we have the following commutative diagram,
\[
\begin{CD}
\Hom_{(\g_{\C},K)}(X_{\Theta,\lambda^{*}},C^{\infty}_{\eta}(U\backslash G)) @> {\tilde{\Phi}} >> C^{\infty}_{\eta}(U\backslash G/K)\\
@V{\tilde{T}}VV @VV{T}V\\
\Hom_{(\g_{\C},K)}(X_{\Theta,\lambda^{*}},C^{\infty}(G))@>>{\Phi}>C^{\infty}(G/K)
\end{CD}.
\]
Since $\Phi$, $T$ and $\tilde{T}$ are injective, we see $\tilde{\Phi}$ is injective. Next we show that $\im\tilde{\Phi}\subset C^{\infty}_{\eta}(U\backslash G/K; I_{\Theta}(\lambda))$. Take $W\in \Hom_{(\g_{\C},K)}(X_{\Theta,\lambda^{*}},C^{\infty}_{\eta}(U\backslash G))$, then $T(\tilde{\Phi}(W))(g)=\langle\xi,W(f_{0})(g)\rangle\in C^{\infty}(G/K;I_{\Theta}(\lambda))$ where $g\in G$. Hence we have $0=R_{X}T(\tilde{\Phi}(W))(g)=T(R_{X}\tilde{\Phi}(W))(g)$ for $X\in I_{\Theta}(\lambda)$ and $g\in G$. Since $T$ is injective, we have $R_{X}W(f_{0})(g)=0$ for $X\in I_{\Theta}(\lambda)$ and $g\in G$, i.e, $\im\tilde{\Phi}\subset C^{\infty}_{\eta}(U\backslash G/K; I_{\Theta}(\lambda))$.

Finally, we show that $\tilde{\Phi}$ is surjective. Let $f\in C^{\infty}_{\eta}(U\backslash G/K; I_{\Theta}(\lambda))$. For $v\in X_{\Theta,\lambda^{*}}$ there exist $X_{v}\in U(\g)$ such that $v=\pi_{\Theta,\lambda^{*}}(X_{v})f_{0}$ since $X_{\Theta,\lambda^{*}}$ is irreducible. Then we define a mapping $W_{f}\colon X_{\Theta,\lambda^{*}}\ni v=\pi_{\Theta,\lambda^{*}}(X_{v})f_{0}\mapsto R_{X_{v}}f(g)\in C^{\infty}_{\eta}(U\backslash G)$. We need to check that it is a well-defined mapping. If for $X_{v},X'_{v}\in \g$ we have $v=\pi_{\Theta,\lambda^{*}}(X_{v})f_{0}=\pi_{\Theta,\lambda^{*}}(X'_{v})f_{0}$, then we have $\pi_{\Theta,\lambda^{*}}(X_{v}-X'_{v})f_{0}=0$. On the other hand we have $T(f)\in C^{\infty}(G/K;I_{\Theta}(\lambda))$. Thus there exists $I_{f}\in \Hom_{\g_{\C},K}(X_{\Theta,\lambda^{*}},C^{\infty}(G))$ such that $T(f)=\Phi(I_{f})$ by Proposition \ref{maxglob}. We put $Z=X_{v}-X'_{v}$. Then we have $T(R_{Z}f)=R_{Z}T(f)=\Phi(R_{Z}I_{f})=R_{Z}I_{f}(f_{0})(g)=I_{f}(\pi_{\Theta,\lambda^{*}}(Z)f_{0})(g)=0$. Hence by the injectivity of $T$, we have $R_{Z}f(g)=0$, i.e., $R_{X_{v}}f=R_{X'_{v}}f$. This implies that $W_{f}$ is well-defined. Also we can
check that $W_{f}$ is compatible with $\g_{\C}$ and $K$ actions. Hence $W_{f}\in \Hom_{(\g_{\C},K)}(X_{\Theta,\lambda^{*}},C^{\infty}_{\eta}(U\backslash G))$ and $\tilde{\Phi}(W_{f})=W_{f}(f_{0})=f$. Hence $\tilde{\Phi}$ is surjective.
\end{proof}
\begin{rem}
This theorem is an analogue of Yamashita's result for the generalized Whittaker models of discrete series representations (Theroem 2.4. in \cite{Y-1}) and more general settings (Corollary 1.8. in \cite{Y-2}).
\end{rem}

\section{Calculus on the case of $GL(4,\R)$}\label{GL4}
In previous sections, we gave a characterization of the space of generalized Whittaker model as the kernel of some explicit differential operators. We will calculate some examples on $GL(4,\R)$ by using these theories. In particular we take  the spherical degenerate principal series representations induced from the maximal parabolic subgroups $P_{1,4}$, $P_{2,4}$ and compute dimensions of the spaces of generalized Whittaker models and find the basis for them.

Let us explain the detailed settings. We consider the case $n=4$. Hence $G=GL(4,\R)$, $K=O(4)$ , $A$ is the group of the $4\times 4$ diagonal matrices with positive real entries and $N$ is the group of $4\times 4$ strict lower triangular matrices with 1s in the diagonal entries. We put $P_{k}=P_{k,4}, k=1,2$. For $(\lambda_{1},\lambda_{2})\in \C^{2}$, we define the character $\lambda\colon P_{k}\rightarrow \C^{\times}$ and define degenerate principal series representation induced from $P_{k}$ and $\lambda$ as before. Let $X_{k,\lambda}$ be their Harish-Chandra modules  which consist of $K$-finite vectors of these degenerate principal series representations. Then by Theorem \ref{generator1}, the annihilator ideal in $U(\g)$ of the degenerate principal series representation are written by 
\begin{multline}\label{generator4}
I_{k}(\lambda)=I_{\{k,4\}}(\lambda)\\
\sum_{i=1}^{4}\sum_{j=1}^{4}U(\g)(\mathbb{E}-\lambda_{1})(\mathbb{E}-\lambda_{2}-k)+U(\g)(\sum_{i=1}^{4}E_{ii}-\lambda_{1}-(4-k)\lambda_{2}).
\end{multline}
for $k=1,2.$ We put a stronger condition for $\lambda$, 
\[
\lambda_{1}-\lambda_{2}\notin \Z.
\]

\subsection{Equivalent classes of $C^{\infty}_{\eta}(U\backslash G)$.}\label{classify-G}
A generalized Whittaker model is an image of an embeddings of $X_{\Theta,\lambda^{*}}$ into $C^{\infty}_{\eta}(U\backslash G)$ where  $U$ is a closed subgroup of $N$ and $\eta$ is its irreducible unitary representation. In this paper, we consider the space $C^{\infty}_{\eta}(U\backslash G)$ defined as follows. 
\begin{enumerate}
\item the group $U$ is a closed subgroup of $N$ and $\eta$ is its unitary character,
\item the unitary induced representation $L^{2}\Indd _{U}^{N}\eta$ is an irreducible unitary representation of $N$.
\end{enumerate} 

We will classify the $G$-equivalent classes of these $C^{\infty}_{\eta}(U\backslash G)$.
\subsubsection{The classification of $\hat{N}$}\label{nhat}

Firstly we will give the classification of the unitary dual of maximal unipotent subgroup $N$ of $G$ by using the Killirov's method for coadjoint orbits. This section contains no new results. The details of the contents of this section can be found in \cite{C-G} for example. 

Let $\n=\text{Lie}(N)$, i.e.,
\[
\n=\left\{
n(z,y_1,y_2,x_1,x_2,x_3)=
\begin{pmatrix}
0&&&\\
x_1&0&&\\
y_1&x_2&0&\\
z&y_2&x_3&0
\end{pmatrix}\ ;\ x_1,\ldots,z\in\R
\right\}.
\] 
We denote its dual $\R$-vector space by $\n^{*}=\Hom_{\R}(\n,\R)$. We identify this space with a subspace of $M(4,\R)$,
\[
\left\{l(\alpha,\beta_1,\beta_2,\gamma_1,\gamma_2,\gamma_3)=
\begin{pmatrix}
0&\gamma_1&\beta_1&\alpha\\
&0&\gamma_2&\beta_2\\
&&0&\gamma_3\\
&&&0
\end{pmatrix}\ ;\ \alpha,\ldots,\gamma_3\in\R
\right\},
\]
so that
\begin{multline*}
l(\alpha,\beta_1,\beta_2,\gamma_1,\gamma_2,\gamma_3)\cdot n(z,y_1,y_2,x_1,x_2,x_3)\\
=\text{tr}(l(\alpha,\beta_1,\beta_2,\gamma_1,\gamma_2,\gamma_3)n(z,y_1,y_2,x_1,x_2,x_3))\\
=\alpha z+\beta_{1}y_{1}+\beta_{2}y_{2}+\gamma_{1}x_{1}+\gamma_{2}x_{2}+\gamma_{3}x_{3}.
\end{multline*}
We define the coadjoint action of $N$ on $\n$ by $(\text{Ad}^{*}(n)l)(X)=l(\text{Ad}(n^{-1})X)$ for $n\in N, \l\in \n^{*}, X\in\n$. Take a basis $X_1,X_2,X_3,Y_1,Y_2,Z$ of $\n$ and its dual basis $X_1^{*},X_2^{*},X_3^{*},Y_1^{*},Y_2^{*},Z^{*}$ of $\n^{*}$ so that 
\[
n(z,y_1,y_2,x_1,x_2,x_3)=z Z+\sum_{i=1}^{2}y_i Y_i+\sum_{j=1}^{3}x_j X_j
\]
and 
\[
l(\alpha,\beta_1,\beta_2,\gamma_1,\gamma_2,\gamma_3)=\alpha Z^{*}+\sum_{i=1}^{2}\beta_i Y_i^{*}+\sum_{j=1}^{3}\gamma_j X_j^{*}. 
\]
Under our coordinate system, the coadjoint action is written as follows, 
\begin{equation}\label{coad}
\begin{aligned}
&(\text{Ad}^{*}\exp (n(z,\cdots,x_3)))(l(\alpha,\cdots,\gamma_3))\\
&=\alpha Z^{*}+(\beta_1+\alpha x_3)Y_{1}^{*}+(\beta_2 -\alpha x_1)Y_{2}^{*}\\
&+(\gamma_1 +\beta_1 x_2 +\alpha(y_2+\frac{x_2 x_3}{2}))X_{1}^{*}+(\gamma_2+x_3\beta_2 -x_1\beta_1-x_1 x_3 \alpha)X_{2}^{*}\\
&+(\gamma_3 -x_2\beta_2+\alpha(\frac{x_1 x_2}{2}-y_1))X_3^{*}.
\end{aligned}
\end{equation}

We consider the classification of coadjoint orbits of $\mathfrak{n}^{*}$.
First, we assume that $\alpha\neq 0$. Then by the equation (\ref{coad}), if we choose appropriate $x_3,x_1,y_2,y_1$, we can find in the $\text{Ad}^{*}N$-orbit a point with $\beta_1=\beta_2=\gamma_1=\gamma_3=0$. Hence if we write $l_{\alpha,\gamma_2}=l(\alpha,0,0,0,\gamma_2,0)$, the coadjoint orbit is written as follows,
\begin{multline*}
(\text{Ad}^{*}N)l_{\alpha,\gamma_2}=\{
\alpha Z^{*}+t_1 Y_{1}^{*}+t_2 Y_{2}^{*} +s_1 X_1^{*}\\
+(\gamma_2 +\frac{t_1 t_2}{\alpha})X_{2}^{*}+s_2 X_{3}^{*}\ ;\ t_1,t_2,s_1,s_2\in\R
\}.
\end{multline*}
Next, we consider the case $\alpha=0$, i.e., $l(0,\beta_1,\beta_2,\gamma_1,\gamma_2,\gamma_3)$. We assume $\beta_1\neq 0$ or $\beta_2\neq 0$. Then from the equation (\ref{coad}), we can see that there is an element $l(0,\beta'_1,\beta'_2,\gamma'_1,0,\gamma'_3)$ in $(\text{Ad}^{*}N)l(0,\beta_1,\beta_2,\gamma_1,\gamma_2,\gamma_3)$. Hence in this case, it is enough to consider the orbit
\begin{multline*}
(\text{Ad}^{*}N)l(0,\beta_1,\beta_2,\gamma_1,0,\gamma_3)\\
=\{
\beta_1 Y_1^{*}+\beta_2 Y_2^{*}+(\beta_1 t_1+\gamma_1)X_{1}^{*}\\
+t_{2}X_{2}^{*}+(\gamma_3 -\beta_{2}t_1)X_{3}^{*}\ ;\ t_1,t_2\in\R
\}.
\end{multline*}

Finally we consider the case $\beta_1=\beta_2=0$. Then we have
\[
(\text{Ad}^{*}n^{*}(0,0,0,\gamma_1,\gamma_2,\gamma_3))=\{
\gamma_1 X_{1}^{*}+\gamma_2 X_{2}^{*}+\gamma_3 X_{3}^{*}
\}.
\]
We summarize these as a proposition.
\begin{prop}\label{classcoad}
We can classify coadjoint orbits of $\n^{*}$ in following cases.
\begin{itemize}
\item[(I)] For $\alpha\in\R\backslash\{0\}$ and $\gamma_{2}\in\R$,
\begin{align*}
\mathcal{O}_{\alpha,\gamma_{2}}=&(\text{Ad}^{*}N)l(\alpha,0,0,0,\gamma_2,0)\\
=&\{
\alpha Z^{*}+t_1 Y_{1}^{*}+t_2 Y_{2}^{*} +s_1 X_1^{*}+(\gamma_2+\frac{t_1 t_2}{\alpha})X_{2}^{*}\\
&+s_2 X_{3}^{*}\ ;\ t_1,t_2,s_1,s_2\in\R
\}.
\end{align*}
We have $\text{dim}\,\mathcal{O}_{\alpha,\gamma_{2}}$=4.
\item[(II)] For $\beta_{1},\beta_{2},\gamma_{1},\gamma_{3}\in\R$ such that $\beta_{1}\beta_{2}\neq 0$,
\begin{align*}
\mathcal{O}_{\beta_{1},\beta_{2},\gamma_{1},\gamma_{3}}=&(\text{Ad}^{*}N)l(0,\beta_1,\beta_2,\gamma_1,0,\gamma_3)\\
=&\{
\beta_{1}Y_{1}^{*}+\beta_{2}Y_{2}^{*}+(\beta_1t_1+\gamma_1)X_1^{*}+t_2 X_{2}+(\gamma_3-\beta_2 t_1)X_{3}^{*}\ ;\\
&t_1,t_2\in\R\}.
\end{align*}
We have $\text{dim}\,\mathcal{O}_{\beta_{1},\beta_{2},\gamma_{1},\gamma_{3}}=2$.
\item[(III)] For $\gamma_{1},\gamma_{2},\gamma_{3}\in\R$,
\begin{align*}
\mathcal{O}_{\gamma_{1},\gamma_{2},\gamma_{3}}=&(\text{Ad}^{*}N)l(0,0,0,\gamma_1,\gamma_2,\gamma_3)\\
=&\{
\gamma_1 X_1^{*}+\gamma_2 X_{2}^{*}+\gamma_3 X_{3}^{*}
\}.
\end{align*}
We have $\text{dim}\,\mathcal{O}_{\gamma_{1},\gamma_{2},\gamma_{3}}=0$.
\end{itemize}
\end{prop}

To construct the irreducible unitary representation of $N$ from the coadjoint orbit of $l\in \n^{*}$, we should determine its radical $\mathfrak{r}_{l}$ and maximal subordinate subalgebra $\mathfrak{s}_{l}$. We define the coadjoint action of the Lie algebra $\n$ on $l\in \n^{*}$ by $((\text{ad}^{*}X)l)(Y)=l([Y,X])$ for $X,Y\in\n$.
\begin{df}
For $l\in \n^{*}$, we define the subalgebra of $\n$ such that
\[
\mathfrak{r}_{l}=\{X\in\n\ ;\ (\text{ad}^{*}X)l=0\}.
\]
We call this subalgebra the radical of $l\in \mathfrak{n}^{*}$.
\end{df}
If $V$ is a $\R$-vector space with an alternating bilinear form $B$, its isotropic subspace $W$ is the subspace such that $B(w,w')=0$ for all $w,w'\in W$. We define the radical of $B$ by $\text{rad}B=\{v\in V\ ;\ B(v,w)=0, \text{all}\ w\in V\}$. It is known that any maximal isotropic subspaces of $V$ have codimention $\frac{1}{2}\text{dim}_{\R}(V/\text{rad}B)$.
\begin{df}
For $l\in\n^{*}$, we can regard $l([X,Y])$ as a bilinear form  for $(X,Y)\in \n\times\n$. By the antisymmetry of Lie bracket $[X,Y]=-[Y,X]$ $(X,Y\in\n)$, this is an alternating form on $\n\times\n$. The subalgebra $\mathfrak{s}_{l}\subset \n$ which is isotropic for $l$ and has codimension $\frac{1}{2}\text{dim}_{\R}(\n/\mathfrak{r}_{l})$ is called maximal subordinate subalgebra of $\n$ for $l$.
\end{df}
\begin{rem}\label{maximal subordinate}
For a nilpotent Lie algebra $\n$, there exists at least one maximal subordinate subalgebra for any $l\in\n^{*}$. Although the radical for $l$ is uniquely determined, maximal subordinate subalgebras are not unique for $l$.
\end{rem}
Let us construct radicals and maximal subordinate subalgebras for coadjoint orbits (I), (II), (III) which are classified in Proposition \ref{classcoad}.

\textit{The case (I)}.  By the equation (\ref{coad}), the coadjoint action of $N$ on $l_{\alpha,\gamma_{2}}=l(\alpha,0,0,0,\gamma_2,0)$ is written as follows, 
\begin{multline*}
(\text{Ad}^{*}\exp (n(z,\cdots,x_3)))l(\alpha,0,0,0,\gamma_2,0)\\
=\{
\alpha Z^{*}+y_1 Y_{1}^{*}+y_2 Y_{2}^{*} +x_1 X_1^{*}+(\gamma_2+\frac{y_1 y_2}{\alpha})X_{2}^{*}+x_2 X_{3}^{*}\}.
\end{multline*}
Hence we can see that 
\[
\mathfrak{r}_{l_{\alpha,\gamma_{2}}}=\{\R Z+\R X_2\}.
\] 
As we noted before, there are some choices of maximal subordinate subalgebra even if it contains the radical $\mathfrak{r}_{l_{\alpha,\gamma_{2}}}$. Among these choices, we pick up a maximal subordinate subalgebra
\[
\mathfrak{s}_{l_{\alpha,\gamma_{2}}}=\{\R X_2+\R Y_1+\R Y_2+\R Z \}.
\]
It is easy to check that this subspace is isotropic and its codimension is equal to $\frac{1}{2}\text{dim}_{\R}(\n/\mathfrak{r}_{l_{\alpha,\gamma_{2}}})=2$. Also this becomes a subalgebra of $\mathfrak{n}$. We can see that $\mathfrak{s}_{l_{\alpha,\gamma_{2}}}$ does not depend on the choice of $\alpha\in\R\backslash\{0\}$ and $\gamma_{2}\in\R$. Hence we simply write $\mathfrak{s}_{(I)}=\mathfrak{s}_{l_{\alpha,\gamma_{2}}}$.

\textit{The case (II)}. 
As well as the case (I), we can see that the radical for $l_{\beta_{1},\beta_{2},\gamma_{1},\gamma_{3}}=l(0,\beta_1,\beta_2,\gamma_1,0,\gamma_3)$ is given by
\[
\mathfrak{r}_{l_{\beta_{1},\beta_{2},\gamma_{1},\gamma_{3}}}=\{\R (\beta_1 X_3+\beta_2 X_1)+ \R Y_1+\R Y_2 +\R Z\}.
\]
The codimension of maximal subordinate subalgebras are $\frac{1}{2}\text{dim}_{\R}(\n/\mathfrak{r}_{l_{\beta_{1},\beta_{2},\gamma_{1},\gamma_{3}}})=1$. We recall a fact for the codimension $1$ subalgebra of $\n$.
\begin{prop}[cf. Proposition 1.3.4 in \cite{C-G}]
Let $\g$ be a nilpotent Lie algebra and $\g_{0}$ a codimension $1$ subalgebra of $\g$. For $l\in \g^{*}$, let $l_{0}=l|_{\g_{0}}$ be the restriction to $\g_{0}$. If the radical of $l$ in $\g$ is contained in $\g_{0}$, any maximal subordinate subalgebra of $\g_{0}$ for $l_{0}$ is also maximal subordinate subalgebra of $\g$ for $l$.
\end{prop}
For any codimension $1$ subalgebra $\n_{0}$ of $\n$ containing $\mathfrak{r}_{l_{\beta_{1},\beta_{2},\gamma_{1},\gamma_{3}}}$, there exist a maximal subordinate subalgebra of $\n_{0}$ for $l_{\beta_{1},\beta_{2},\gamma_{1},\gamma_{3}}|_{\n_{0}}$ from Remark \ref{maximal subordinate}. By this proposition, this is also a maximal subordinate subalgebra of $\n$ for $l_{(\text{II})}$. By the calculation done above, this maximal subordinate subalgebra should have  codimension $1$. Hence this is nothing but $\n_{0}$. This implies that any codimension $1$ subalgebra $\n_{0}$ of $\n$ containing $\mathfrak{r}_{l_{\beta_{1},\beta_{2},\gamma_{1},\gamma_{3}}}$ is maximal subordinate subalgebra for $l_{\beta_{1},\beta_{2},\gamma_{1},\gamma_{3}}$. Among these, we pick up a maximal subordinate subalgebra
\[
\mathfrak{s}_{l_{\beta_{1},\beta_{2},\gamma_{1},\gamma_{3}}}=\{\R X_1+\R X_3 +\R Y_1+\R Y_2+\R Z \}.
\]
As in the case (I), the subalgebra $\mathfrak{s}_{l_{\beta_{1},\beta_{2},\gamma_{1},\gamma_{3}}}$ does not depends on the choice of $\beta_{1},\beta_{2},\gamma_{1},\gamma_{3}\in \R$. Hence we simply write $\mathfrak{s}_{(\text{II})}=\mathfrak{s}_{l_{\beta_{1},\beta_{2},\gamma_{1},\gamma_{3}}}$. 

\textit{The case (III)}. The coadjoint action on $l_{\gamma_{1},\gamma_{2},\gamma_{3}}=l(0,0,0,\gamma_1,\gamma_2,\gamma_3)$ is given by
\[
(\text{Ad}^{*}\exp (n(z,\cdots,x_3)))l(0,0,0,\gamma_1,\gamma_2,\gamma_3)=\{\gamma_1 X_1^{*}+\gamma_2 X_{2}^{*}+\gamma_3 X_{3}^{*}\}.
\]
It is obvious that the radical of $l_{\gamma_{1},\gamma_{2},\gamma_{3}}$ is 
\[
\mathfrak{r}_{l_{\gamma_{1},\gamma_{2},\gamma_{3}}}=\n.
\]
Also it is obvious that a maximal subordinate subalgebra of $\gamma_{1},\gamma_{2},\gamma_{3}$ is
\[
\mathfrak{s}_{\gamma_{1},\gamma_{2},\gamma_{3}}=\n.
\]

Let us recall Kirillov's theory for irreducible unitary representations of nilpotent Lie group $N$. For $l\in \n^{*}$, let $\mathfrak{s}_{l}$ be a maximal subordinate subalgebra for $l$ and let $S_{l}=\exp \mathfrak{s}_{l}$ . We can extend $l|_{\mathfrak{s}_{l}}\colon \mathfrak{s}_{l}\rightarrow \R$ to the map $\chi_{l}\colon S_{l}\rightarrow \C^{1}$ by
\[
\chi_{l}(\exp X)=e^{2\pi i l(X)},\ X\in \mathfrak{s}_{l}.
\]
This is a group homomorphism, i.e., an unitary character of $S_{l}$ because $\mathfrak{s}_{l}$ is an isotropic subspace for $l$, i.e., $l([X,Y])=0$ for $X,Y\in \mathfrak{s}_{l}$.
We consider a Hilbert space induced from $\chi_{l}$,
\begin{align*}
\mathcal{H}_{\chi_{l}}=&\{f\colon N\rightarrow \C\ \text{measurable}\ ;\ f(sx)=\chi_{l}(s)f(x)\ \text{for}\ s\in S_{l},x\in N,\\
&\text{and}\ \int_{S_{l}\backslash N}|f(x)|^{2}\,d\dot{x}<+\infty \},
\end{align*}
where $d\dot{x}$ is the right-invariant measure on $S_{l}\backslash N$. The inner product is defined by
\[
\langle f,f'\rangle=\int_{S_{l}\backslash N}f(x)\overline{f'(x)}\,d\dot{x}.
\]
It can be shown that $\mathcal{H}_{\chi_{l}}$ is complete by this inner product. The action of $N$ on $\mathcal{H}_{\chi_{l}}$ is given by the right translation. This is the unitary representation by the right-invariance of $d\dot{x}$. This representation is called the representation of $N$ induced from $\chi_{l}$, and denoted by $L^{2}\Indd_{S_{l}}^{N}(\chi_{l})$. 

\begin{thm}[Kirillov \cite{K}]\label{Kirillov}
Take $l\in \n^{*}$ and let $\mathfrak{s}_{l}$ be a maximal subordinate subalgebra of $\n$ for $l^{*}$.
\begin{enumerate}
\item The induced representation $L^{2}\Indd_{S_{l}}^{N}(\chi_{l})$ is an irreducible representation of $N$.
\item Let $\mathfrak{s}'_{l}$ be a maximal subordinate subalgebra of $\n$ for $l$ and $S'_{l}=\exp \mathfrak{s}'_{l}$. Then $L^{2}\Indd_{S'_{l}}^{N}(\chi_{l})$ is unitarily equivalent to $L^{2}\Indd_{S_{l}}^{N}(\chi_{l})$. Hence we may write $\pi_{l}$ for $L^{2}\Indd_{S_{l}}^{N}(\chi_{l})$. 
\item Let $l'\in \n^{*}$. Then $\pi_{l'}$ is unitarily equivalent to $\pi_{l}$ if and only if $l'\in(\mathrm{Ad}^{*}N)l$.
\item Let $\pi$ be an irreducible unitary representation of $N$. Then there exists an $l\in \n^{*}$ such that $\pi$ is unitarily equivalent to $\pi_{l}$.
\end{enumerate}
\end{thm}
This Kirillov's theorem implies that irreducible unitary representations of nilpotent subgroup $N$ of $G$ are equivalent to induced representations $\text{Ind}_{S_{l}}^{N}(\chi_{l})$ and their equivalent classes only depend on coadjoint orbits of $l\in\n^{*}$. We have already classified coadjoint orbits of $l\in\n^{*}$ and determined their maximal subordinate subalgebras $\mathfrak{s}_{l}$. Hence we can obtain equivalent classes of irreducible unitary representations of $N$.

\begin{prop}We retain the notation as above.
The every irreducible unitary representation of $N$ is unitarily equivalent to one of the following representations.
\begin{itemize}
\item[(I)] For $l_{\alpha,\gamma_{2}}=l(\alpha,0,0,0,\gamma_{2},0)\in \mathfrak{n}^{*}$ and its maximal subordinate subalgebra $\mathfrak{s}_{\mathrm{(I)}}=\{\R X_{2}+\R {Y}_{1}+\R Y_{2}+\R Z\}$, we define the representation 
\[
L^{2}\Indd _{S_{{\text(I)}}}^{N}\chi_{l_{\alpha,\gamma_{2}}}.
\]
Here $S_{\text{(I)}}=\exp \mathfrak{s}_{\text{(I)}}$ and $\alpha\in \R\backslash\{0\},\gamma_{2}\in\R$.

\item[(II)] For $l_{\beta_{1},\beta_{2},\gamma_{1},\gamma_{3}}=l(0,\beta_{1},\beta_{2},\gamma_{1},0,\gamma_{3})\in\mathfrak{n}^{*}$ and its maximal subordinate subalgebra $\mathfrak{s}_{(\mathrm{II})}=\{\R X_{1}+\R X_{3}+\R Y_{1}+\R Y_{2}+\R Z\}$, we define the representation
\[
L^{2}\Indd _{S_{\text(II)}}^{N}\chi_{l_{\beta_{1},\beta_{2},\gamma_{1},\gamma_{3}}}
\]
Here $S_{l_{\text{(II)}}}=\exp \mathfrak{s}_{l_\text{(II)}}$ and $\beta_{1},\beta_{2},\gamma_{1},\gamma_{3}\in\R$, $\beta_{1}\beta_{2}\neq 0$.

\item[(III)] For $l_{\gamma_{1},\gamma_{2},\gamma_{3}}=l(0,0,0,\gamma_{1},\gamma_{2},\gamma_{3})\in\mathfrak{n}^{*}$, we define the unitary character of $N$, 
\[
\chi_{l_{\gamma_{1},\gamma_{2},\gamma_{3}}}.
\]
\end{itemize}
\end{prop}
\subsubsection{Conjugacy classes of $C^{\infty}_{\eta}(U\backslash G)$}
In previous section, we classify the unitary dual of $N$. Thus the next thing to do is the classification of $G$-equivalent classes of the following spaces,

\begin{align}
&C^{\infty}_{\chi_{l_{\alpha,\gamma_{2}}}}(S_{(I)}\backslash G),\quad \alpha\in\R\backslash\{0\},\,\gamma_{2}\in\R\tag{I},\\
&C^{\infty}_{\chi_{l_{\beta_{1},\beta_{2},\gamma_{1},\gamma_{3}}}}(S_{(II)}\backslash G),\quad \beta_{1},\beta_{2}\in\R,\, \beta_{1}\beta_{2}\neq 0,\tag{II}\\
&C^{\infty}_{\chi_{l_{\gamma_{1},\gamma_{2},\gamma_{3}}}}(N\backslash G),\quad \gamma_{1},\gamma_{2},\gamma_{3}\in\R\tag{III}.
\end{align}

For $x\in G$, we write the conjugation of $g\in G$ by $x$ as $g^{x}=xgx^{-1}$. Let $H$ be a closed subgroup of $G$ and $\pi$ a continuous representation of $H$ on a complete locally convex space $E$. Then for $x\in N_{G}(H)=\{g\in G\,|\, h^{g}\in H,\ \text{for any}\ h\in H\}$, we can define conjugation of $\pi$ as $\pi^{x}(h)=\pi(h^{x})$. Then we have the following fact about the induced representation $C^{\infty}_{\pi}(H\backslash G)=\{f\colon G\rightarrow E\,\text{smooth}\mid f(hg)=\pi(h)f(g),\, g\in G,h\in H\}$ on which $G$ acts by the right translation. 
\begin{lem}\label{conj}
We retain the notations as above. The map
\[
\begin{array}{ccc}
C^{\infty}_{\pi}(H\backslash G)&\xrightarrow{\sim}&C^{\infty}_{\pi^{x}}(H\backslash G)\\
f(g)&\longmapsto&F(g)=f(xg)
\end{array}
\]
gives isomorphism as $G$-modules.
\end{lem}
\begin{proof}
We see that this map is well-defined. Take $f\in C^{\infty}_{\pi}(H\backslash G)$ and define $F(g)=f(xg)$. Then we have
\begin{align*}
F(hg)&=f(xhg)\\
&=f(xhx^{-1}xg)\\
&=\pi(xhx^{-1})f(xg)\\
&=\pi^{x}(h)F(g),
\end{align*}
for $h\in H$. Hence $F\in C^{\infty}_{\pi^{x}}(H\backslash G)$. Obviously this map is bijective and preserves the action of $G$. 
\end{proof}
\begin{lem}\label{equivconj}
Fix a maximal subordinate subalgebra $\mathfrak{s}_{l}\subset \mathfrak{n}$ for $l\in\mathfrak{n}^{*}.$ We put $S_{l}=\exp\mathfrak{s}_{l}$. Let us define a character $\chi_{l}\colon S_{l}\rightarrow \C^{1}$ so that $\chi_{l}(\exp X)=e^{2\pi\sqrt{-1}l(X)}$ for $X\in\mathfrak{s}_{l}$. Then the character $\chi_{l}$ is invariant by the conjugation by $S_{l}$, i.e.,
\[
\chi_{l}^{x}(s)=\chi_{l}(s)
\]
for $s,x\in S_{l}$.
\end{lem} 
\begin{proof}
Take an element of $x\in S_{l}$. Then there exists $Z_{x}\in \mathfrak{s}_{l}$ such that $x=\exp Z_{x}$. By the Campbell-Hausdorff formula, for any $X\in \mathfrak{s}_{l}$ we have
\begin{align*}
x(\exp X) x^{-1}&=\exp Z_{x}\exp X\exp(-Z_{x})\\
&\equiv \exp(Z_{x}+X-Z_{x}+Y)\\
&=\exp (X+Y)
\end{align*}
for some element $Y\in [\mathfrak{s}_{l},\mathfrak{s}_{l}]=\{[V,W]\mid V,W\in\mathfrak{s}_{l}\}$. If we recall that $\mathfrak{s}_{l}$ is a maximal subordinate subalgebra for $l$, we have $l(Y)=0$ for $Y\in [\mathfrak{s}_{l},\mathfrak{s}_{l}]$. Thus we have 
\begin{align*}
\chi_{l}^{x}(\exp X)&=\chi_{l}(x(\exp X) x^{-1})\\
&=e^{2\pi\sqrt{-1}l(X)}\\
&=\chi_{l}(\exp X).
\end{align*}
for any $X\in \mathfrak{s}_{l}$.
\end{proof}
By these lemmas, we have the following classifications.
\begin{prop}\label{classify}
Case(I). For $\alpha\in\R\backslash \{0\}$ and $\gamma_{2}\in \R$, we have
\[
C^{\infty}_{\chi_{l_{\alpha,\gamma_{2}}}}(S_{\text(I)}\backslash G)\cong
\begin{cases}
C^{\infty}_{\chi_{l(0,1,1,0,0,0)}}(S_{\text(I)}\backslash G)&\text{if}\ \gamma_{2}\neq 0,\qquad(\text{I}_{1})\\
C^{\infty}_{\chi_{l(0,0,1,0,0,0)}}(S_{\text(I)}\backslash G)&\text{if}\ \gamma_{2}=0\qquad(\text{I}_{2})
\end{cases}
\]
Case(II). We take $\beta_{1},\beta_{2},\gamma_{1},\gamma_{3}\in\R$ and assume $\beta_{1}\beta_{2}\neq 0$. Then we have
\begin{multline*}
C^{\infty}_{\chi_{l_{\beta_{1},\beta_{2},\gamma_{1},\gamma_{3}}}}(S_{\text(II)}\backslash G)\cong\\
\begin{cases}
C^{\infty}_{\chi_{l(0,0,1,1,0,0)}}(S_{\text(II)}\backslash G)&\text{if}\ (\beta_{1},\gamma_{1})\cdot (\gamma_{3},\beta_{2})\neq0,\qquad(\text{II}_{1})\\ \\
C^{\infty}_{\chi_{l(0,0,0,1,0,1)}}(S_{\text(II)}\backslash G)&\begin{array}{c}\text{if}\ (\beta_{1},\gamma_{1})\cdot (\gamma_{3},\beta_{2})=0\ \\ \text{and}\ \beta_{1}\neq 0,\beta_{2}\neq 0\end{array},\qquad(\text{II}_{2})\\ \\
C^{\infty}_{\chi_{l(0,0,0,1,0,0)}}(S_{\text(II)}\backslash G)&\begin{array}{c}\text{if}\ (\beta_{1},\gamma_{1})\cdot (\gamma_{3},\beta_{2})=0\ \\ \text{and}\ \beta_{1}\neq 0,\beta_{2}=0\end{array},\qquad(\text{II}_{3})\\ \\
C^{\infty}_{\chi_{l(0,0,0,0,0,1)}}(S_{\text(II)}\backslash G)&\begin{array}{c}\text{if}\ (\beta_{1},\gamma_{1})\cdot (\gamma_{3},\beta_{2})=0\ \\ \text{and}\ \beta_{1}=0,\beta_{2}\neq 0\end{array},\qquad(\text{II}_{4})
\end{cases}
\end{multline*}
where $(a,b)\cdot (c,d)=ac+bd$ for $a,b,c,d\in \R$ is a natural inner product in $\R^{2}$.
Case(III). For $\gamma_{1},\gamma_{2},\gamma_{3}\in\R$, we have
\begin{multline*}
C^{\infty}_{\chi_{l_{\gamma_{1},\gamma_{2},\gamma_{3}}}}(N\backslash G)\cong\\
\begin{cases}
C^{\infty}_{\chi_{l(0,0,0,1,1,1)}}(N\backslash G)&\text{if}\ \gamma_{1}\neq 0,\gamma_{2}\neq 0,\gamma_{3}\neq 0,\qquad(\text{III}_{1})\\
C^{\infty}_{\chi_{l(0,0,0,1,1,0)}}(N\backslash G)&\text{if}\ \gamma_{1}\neq 0,\gamma_{2}\neq 0,\gamma_{3}= 0,\qquad(\text{III}_{2})\\
C^{\infty}_{\chi_{l(0,0,0,1,0,1)}}(N\backslash G)&\text{if}\ \gamma_{1}\neq 0,\gamma_{2}=0,\gamma_{3}\neq 0,\qquad(\text{III}_{3})\\
C^{\infty}_{\chi_{l(0,0,0,0,1,1)}}(N\backslash G)&\text{if}\ \gamma_{1}=0,\gamma_{2}\neq 0,\gamma_{3}\neq 0,\qquad(\text{III}_{4})\\
C^{\infty}_{\chi_{l(0,0,0,1,0,0)}}(N\backslash G)&\text{if}\ \gamma_{1}\neq 0,\gamma_{2}=0,\gamma_{3}=0,\qquad(\text{III}_{5})\\
C^{\infty}_{\chi_{l(0,0,0,0,1,0)}}(N\backslash G)&\text{if}\ \gamma_{1}=0,\gamma_{2}\neq 0,\gamma_{3}=0,\qquad(\text{III}_{6})\\
C^{\infty}_{\chi_{l(0,0,0,0,0,1)}}(N\backslash G)&\text{if}\ \gamma_{1}=0,\gamma_{2}=0,\gamma_{3}\neq 0,\qquad(\text{III}_{7})\\
C^{\infty}_{\chi_{l(0,0,0,0,0,0)}}(N\backslash G)&\text{if}\ \gamma_{1}=0,\gamma_{2}=0,\gamma_{3}=0,\qquad(\text{III}_{8})
\end{cases}
\end{multline*}
\end{prop}
\begin{proof}
The case (I). The normalizer $N_{G}(S_{(\text{I})})$ of $S_{(\text{I})}$ in $G$ is written as the semi-direct product $L_{I}\ltimes S_{I}$ where
\[
L_{I}=\left\{
\left(\begin{array}{c|c}
A&0_{2}\\
\hline
0_{2}&B
\end{array}
\right)\ |\ A,B\in GL(2,\R)
\right\}.
\]
Here $0_{2}=\begin{pmatrix}0&0\\
0&0
\end{pmatrix}\in M(2,\R)$.  We define the action of $N_{G}(S_{(\text{I})})$ on $\chi_{(\alpha,\gamma_{2})}$ as $(x\cdot\chi_{(\alpha,\gamma_{2})})(s)=\chi_{(\alpha,\gamma_{2})}(s^{x^{-1}})$ for $x\in N_{G}(S_{(\text{I})})$ and $s\in S_{(\text{I})}$. Then by lemma \ref{conj}, if $\chi_{l_{\alpha,\gamma_{2}}}$ and $\chi_{l_{\alpha ',\gamma_{2}'}}$ are in the same $N_{G}(S_{(\text{I})})$-orbit, the spaces $C^{\infty}_{\chi_{l_{\alpha,\gamma_{2}}}}(S_{I}\backslash G)$ and $C^{\infty}_{\chi_{l_{\alpha',\gamma_{2}'}}}(S_{I}\backslash G)$ are $G$-equivalent. Also by Lemma \ref{equivconj}, we only need to classify the $L_{(I)}$-orbits of $\chi_{l_{\alpha,\gamma_{2}}}$ for $\alpha\in\R\backslash \{0\}$ and $\gamma_{2}\in\R$. Then it is easy to see that
$
l_{(\alpha,0,0,0,\gamma_{2},0)}\in \text{Ad}^{*}(N_{G}(S_{(\text{I})}))(l_{(0,1,0,0,0,1)})$
 if $\gamma_{2}\neq 0$ and $
l_{(\alpha,0,0,0,\gamma_{2},0)}\in \text{Ad}^{*}(N_{G}(S_{(\text{I})}))(l_{(0,1,0,0,0,0)})$
 if $\gamma_{2}=0$.

The case(II). The normalizer of $S_{\text{(II)}}$ in $G$ is written as the semi-direct product $L_{(II)}\ltimes S_{(II)}$ where 
\[
L_{(II)}=\left\{n_{(a,b,A)}=
\begin{pmatrix}
a&\mathbf{0}_{2}&0\\
^{t}\mathbf{0}_{2}&A&^{t}\mathbf{0}_{2}\\
0&\mathbf{0}_{2}&b
\end{pmatrix}\in G\ |\ a,b\in \R^{\times}, A\in GL(2,\R)\right\}.
\] 
Here $\mathbf{0}_{2}=(0,0)$ and $^{t}\mathbf{0}_{2}=\begin{pmatrix}0\\0\end{pmatrix}$.
Then there are following $L_{(II)}$-orbits of $\chi_{l_{\beta_{1},\beta_{2},\gamma_{1},\gamma_{2}}}$ for $\beta_{1},\beta_{2}\in \R\ (\beta_{1}\beta_{2}\neq 0)$ and $\gamma_{1},\gamma_{2}\in\R$, 
\begin{gather*}
\mathcal{O}_{1}=\{\chi_{l(0,0,v_{1},v_{2},w_{1},0,w_{2})} \mid v_{1}w_{1}+v_{2}w_{2}\neq 0\}\\
\mathcal{O}_{2}=\{\chi_{l(0,0,v_{1},v_{2},w_{1},0,w_{2})}\mid (v_{1},v_{2})\neq (0,0),\ (w_{1},w_{2})\neq 0\ \text{and}\ v_{1}w_{1}+v_{2}w_{2}=0\},\\
\mathcal{O}_{3}=\{\chi_{l(0,0,v_{1},v_{2},w_{1},0,w_{2})}\mid (v_{1},v_{2})\neq (0,0),\ (w_{1},w_{2})= 0\ \text{and}\ v_{1}w_{1}+v_{2}w_{2}=0\},\\
\mathcal{O}_{4}=\{\chi_{l(0,0,v_{1},v_{2},w_{1},0,w_{2})}\mid (v_{1},v_{2})=(0,0),\ (w_{1},w_{2})\neq 0\ \text{and}\ v_{1}w_{1}+v_{2}w_{2}=0\}.
\end{gather*}

The case(III). The normalizer of $N$ in $G$ is written as the semi-direct product $L_{(III)}\ltimes S_{(III)}$ where
\[
L_{(III)}=\left\{
\begin{pmatrix}
a_{1}&&&\\
&a_{2}&&\\
&&a_{3}&\\
&&&a_{4}
\end{pmatrix}
\ |\ a_{1},\ldots,a_{4}\in \R^{\times}
\right\}.
\]
Then the lemma is easily follows.
\end{proof}
\subsection{Differential operators on the generalized Whittaker models}\label{diff}
Let $U$ be a closed subgroup of $N$ and $\chi$ a character of $U$. By Theorem \ref{Yamashita}, the space of the generalized Whittaker model is isomorphic to the subspace of 
\[
C^{\infty}_{\chi}(U\backslash G/K)=\{f\in C^{\infty}(G)\ |\ f(ugk)=\chi(u)f(g)\ \text{for}\ (u,g,k)\in U\times G\times K\}.
\]
\begin{lem}\label{cross}
We retain the above notations. There exists a linear bijection
\[
\Xi\colon C^{\infty}_{\chi}(U\backslash G/K)\xrightarrow{\sim} C^{\infty}(U\backslash N\times A).
\]
\end{lem}
\begin{proof}
Because $N$ is a nilpotent group and $U$ is its closed subgroup, there is a smooth cross section $\theta\colon U\backslash N\rightarrow N$ with the smooth splitting of $n\in N$ so that $n=u(n)s(n)$ for $u(n)\in U$ and $s(n)\in \theta(U\backslash N)$ (cf. Theorem 1.2.12 in \cite{C-G}). This smooth cross section gives us a linear mapping
\[
\begin{array}{ccc}
\Xi\colon C^{\infty}_{\chi}(U\backslash G/K)&\xrightarrow{\sim}&C^{\infty}(U\backslash N\times A)\\
f&\longmapsto&\Xi(f)(x,a)=f(\theta(x)a),  
\end{array}
\]
for $x\in U\backslash N$ and $a\in A$. Take an element $\phi\in C^{\infty}(U\backslash N\times A)$. If we define an element of $f_{\phi}\in C^{\infty}_{\chi}(U\backslash G/K)$ by
\[
f_{\phi}(usak)=\chi(u)\phi(sa)
\]
for $u\in U,s\in \theta(U\backslash N),a\in A$ and $k\in K$. Since $G\cong U\times U\backslash N\times A\times K$, this is well defined. We denote this map $\Pi$. Then it is easy to see that $\Pi\circ\Xi=\mathrm{id}_{C^{\infty}_{\chi}(U\backslash G/K)}$ and $\Xi\circ\Pi=\mathrm{id}_{C^{\infty}(U\backslash N\times A)}$. Hence $\Xi$ is bijective.
\end{proof}
We define the action of $U(\g)$ on $C^{\infty}(U\backslash N\times A)$ by $X\cdot\Xi(f)=\Xi(R_{X}f)$ for $X\in U(\g)$ and $f\in C^{\infty}_{\chi}(U\backslash G/K).$ In this section, we will give the explicit expression of the action of $U(\g)$ on $C^{\infty}(U\backslash N\times A)$.

According to the Iwasawa decomposition $\g=\mathfrak{n}\oplus \ah\oplus\ka$, it suffices to see the action of $\mathfrak{n}, \ah$ and $\ka$ respectively. We can see that $E_{ii}\in\ah, i=1,\ldots,4$ acts on $C^{\infty}(U\backslash N\times A)$  as $\vartheta_{a_{i}}=a_{i}\frac{\partial}{\partial a_{i}}, i=1,\ldots,4$ if we denote the elements of $A$ by $a=\text{diag}(a_{1},\ldots,a_{4})$. By the right $K$-invariance of $C^{\infty}_{\chi}(U\backslash G/K)$, the elements in $\ka$ acts trivially. Hence we have the following symmetric relation among the generators of the annihilator ideal $I_{k}(\lambda)$.
\begin{lem}\label{symmetric}
For $F\in C^{\infty}(G/K)$, we have
\[
\left((\mathbb{E}-\lambda_1)(\mathbb{E}-\lambda_2-k)\right)_{ij}F=\left((\mathbb{E}-\lambda_1)(\mathbb{E}-\lambda_2-k)\right)_{ji}F,
\]
where $1\le i,j\le 4,$ and $k=1,2$.
\end{lem}
\begin{proof}
Take elements $(E_{ij}-E_{ji}),\ 1\le i<j\le 4$ as the generators of $\ka$. Then we have $(E_{ij}-E_{ji})F=0,\ 1\le i<j\le 4$ for $F\in C^{\infty}(G/K)$, i.e., $E_{ij}F=E_{ji}F,\ 1\le i<j\le 4$. This implies that 
\begin{align*}
&\left((\mathbb{E}-\lambda_1)(\mathbb{E}-\lambda_2-k)\right)_{ij}F\\
&=(\sum_{l=1}^{4}E_{il}E_{lj}-(\lambda_1+\lambda_2+k)E_{ij}+\lambda_1(\lambda_2+k)\delta_{ij})F\\&=(\sum_{l=1}^{4}E_{il}E_{jl}-(\lambda_1+\lambda_2+k)E_{ji}+\lambda_1(\lambda_2+k)\delta_{ij})F\\
&=(\sum_{l=1}^{4}(E_{jl}E_{il}-[E_{jl},E_{il}])-(\lambda_1+\lambda_2+k)E_{ij}+\lambda_1(\lambda_2+k)\delta_{ij})F\\
&=(\sum_{l=1}^{4}(E_{jl}E_{li}-(\delta_{li}E_{jl}-\delta_{jl}E_{il}))-(\lambda_1+\lambda_2+k)E_{ji}+\lambda_1(\lambda_2+k)\delta_{ij})F\\
&=(\sum_{l=1}^{4}E_{jl}E_{li}-(E_{ji}-E_{ij})-(\lambda_1+\lambda_2+k)E_{ji}+\lambda_1(\lambda_2+k)\delta_{ij})F\\
&=(\sum_{l=1}^{4}E_{jl}E_{li}-(\lambda_1+\lambda_2+k)E_{ji}+\lambda_1(\lambda_2+k)\delta_{ij})F\\
&=\left((\mathbb{E}-\lambda_1)(\mathbb{E}-\lambda_2-k)\right)_{ji}F.
\end{align*}
This is the required equation.
\end{proof}
We give more precise expressions of $\left((\mathbb{E}-\lambda_1)(\mathbb{E}-\lambda_2-k)\right)_{ij}\ \text{mod}\ U(\g)\ka$ for $k=1,2$ and $1\le i<j\le 4$ below.
\begin{lem}\label{modk}
Representatives $\left((\mathbb{E}-\lambda_1)(\mathbb{E}-\lambda_2-k)\right)_{ij}$ modulo $U(\g)\ka$ ($k=1,2$ and $1\le i<j \le 4$) are written as follows,

\begin{gather}
\begin{split}
E_{11}^{2}+E_{21}^{2}+E_{31}^{2}+E_{41}^{2}-&(\lambda_{1}+\lambda_{2}+k-3)E_{11}\\
&-(E_{2}+E_{3}+E_{4})+\lambda_{1}(\lambda_{2}+k),
\end{split}\tag{$(i,j)=(1,1)$}\\
E_{21}(E_{11}+E_{22}-(\lambda_{1}+\lambda_{2}+k-3))+E_{32}E_{31}+E_{42}E_{41},\tag{$(i,j)=(1,2)$}\\
E_{31}(E_{11}+E_{33}-(\lambda_{1}+\lambda_{2}+k-2))+E_{32}E_{21}+E_{43}E_{41},\tag{$(i,j)=(1,3)$}\\
E_{41}(E_{11}+E_{44}-(\lambda_{1}+\lambda_{2}+k-2))+E_{42}E_{21}+E_{31}E_{43},\tag{$(i,j)=(1,4)$}\\
\begin{split}
E_{22}^2-(\lambda_{1}+\lambda_{2}+k-2)E_{22}+&E_{21}^2+E_{32}^2+E_{42}^2\\
&-(E_{33}+E_{44})+\lambda_{1}(\lambda_{2}+k),
\end{split}\tag{$(i,j)=(2,2)$}\\
E_{32}(E_{22}+E_{33}-(\lambda_{1}+\lambda_{2}+k-2))+E_{21}E_{31}+E_{43}E_{42},\tag{$(i,j)=(2,3)$}\\
E_{42}(E_{22}+E_{44}-(\lambda_{1}+\lambda_{2}+k-2))+E_{21}E_{41}+E_{32}E_{43},\tag{$(i,j)=(2,4)$}\\
\begin{split}
E_{33}^2-(\lambda_{1}+\lambda_{2}+k-1)E_{33}+E_{31}^{2}+&E_{32}^{2}+E_{43}^{2}\\&-E_{44}+\lambda_{1}(\lambda_{2}+k),
\end{split}\tag{$(i,j)=(3,3)$}\\
E_{43}(E_{33}+E_{44}-(\lambda_{1}+\lambda_{2}+k-1))+E_{31}E_{41}+E_{32}E_{42},\,\tag{$(i,j)=(3,4)$}\\
E_{44}^{2}-(\lambda_{1}+\lambda_{2}+k)E_{44}+E_{41}^{2}+E_{42}^{2}+E_{43}^{2}.\tag{$(i,j)=(4,4)$}
\end{gather}
\end{lem}
\begin{proof}
If we note that $(E_{ij}-E_{ji})\ 1\le i<j\le 4$ are the generators of $\ka$, this lemma can be obtained by the direct computations.
\end{proof}
Along the classification obtained in Proposition \ref{classify}, we express the action of $\mathfrak{n}$.

\textit{The case (I)}. We consider the space
\begin{equation*}
C^{\infty}_{\chi_{l(0,1,\varepsilon,0,0,0)}}(S_{\text{(I)}}\backslash G/K)
\end{equation*}
Here $\varepsilon=1$ (resp. $=0$) corresponds to the case $(I_{1})$ (resp. $(I_{2})$) classified in Proposition \ref{classify}.
If we notice that $\mathfrak{s}_{\text{(I)}}=\{\R X_2+\R Y_1+\R Y_2+\R Z \}$ is not only a subalgebra of $\mathfrak{n}$ but also an ideal of $\mathfrak{n}$, then $\mathfrak{s}_{\text{(I)}}\backslash \mathfrak{n}\cong\{\R X_{1}+\R X_{3}\}$ can be seen as a subalgebra of $\mathfrak{n}$. Hence $S_{\text{(I)}}\backslash N$ is isomorphic to the subgroup $\{\exp(uX_{1}+vX_{3})\ |\ u,v\in\R\}$ of $N$. This isomorphism gives a smooth cross section $\theta_{\text{(I)}}\colon S_{\text{(I)}}\backslash N\rightarrow N$. Then we have the linear isomorphism
\[
\Xi_{\text{(I)}}\colon C^{\infty}_{\chi_{l(0,1,\varepsilon,0,0,0)}}(S_{\text{(I)}}\backslash G/K)\xrightarrow{\sim} C^{\infty}(S_{\text{(I)}}\backslash N\times A)
\]
by Lemma \ref{cross}. We introduce a coordinate system on $S_{\text{(I)}}\backslash N\times A$ as follows,
\[
\begin{array}{ccc}
\R^{2}\times(\R_{>0})^{4}&\xrightarrow{\sim}&S_{\text{(I)}}\backslash N\times A\\
(u,v)\times(a_{1},a_{2},a_{3},a_{4})&\longmapsto&\exp(uX_{1}+vX_{3})\times\diag(a_{1},a_{2},a_{3},a_{4})
\end{array}
\] 

\begin{prop}\label{n-difI}
We regard the space $C^{\infty}(S_{(I)}\backslash N\times A)$ as the image of the space $C^{\infty}_{\chi_{l(0,1,\varepsilon,0,0,0)}}(S_{\text{(I)}}\backslash G/K)$ by the mapping $\Xi_{(I)}$ for each $\varepsilon =0,1$. Then $\mathfrak{n}$ acts on $C^{\infty}(S_{(I)}\backslash N\times A)$ as follows,
\begin{align*}
&E_{21}F=\frac{a_{2}}{a_{1}}\frac{\partial}{\partial u}F,&E_{31}F=\varepsilon2\pi\sqrt{-1}\frac{a_{3}}{a_{1}},F\\
&E_{41}F=0,&E_{32}F=2\pi\sqrt{-1}\frac{a_{3}}{a_{2}}(v-\varepsilon u)F,\\
&E_{42}F=2\pi\sqrt{-1}\frac{a_{4}}{a_{2}}F,&E_{43}F=\frac{a_{4}}{a_{3}}\frac{\partial}{\partial v},
\end{align*}
for $F\in C^{\infty}(S_{\text{(I)}}\backslash N\times A)$.
\end{prop}
\begin{proof}
For $F\in C^{\infty}(S_{\text{(I)}}\backslash N\times A)$, there exists a $f\in C^{\infty}_{\chi_{l(0,\varepsilon,1,0,0,0)}}(S_{\text{(I)}}\backslash G/K)
$ such that $F(u,v;a)=\Xi_{\text{(I)}}(f)=f(\exp(uX_{1}+vX_{3})a)$ for $u,v\in \R$ and $a\in A$. Hence for $E_{ij} (1\le j<i\le4)$ we have
\begin{equation}\label{2222}
\begin{aligned}
(E_{ij}F)(u,v;a)&=\Xi_{\text{(I)}}(R_{E_{ij}}f)\\
&=\frac{d}{dt}f(\exp(uX_{1}+vX_{3})a\exp(tE_{ij}))|_{t=0}\\
&=\frac{d}{dt}f(\exp(uX_{1}+vX_{3})\exp(t\text{Ad}(a)E_{ij})a)|_{t=0}\\
&=\frac{a_{i}}{a_{j}}\frac{d}{dt}f(\exp(uX_{1}+vX_{3})\exp(tE_{ij})a)|_{t=0}.
\end{aligned}
\end{equation}
By the direct computation, we have
\begin{multline*}
\exp(u X_1+v X_3)\cdot \exp n(z,\cdots,x_3)\\
=\exp n(z',y'_1,y'_2,0,x_2,0)\cdot \exp((u+x_1) X_1+(v+x_3) X_3),
\end{multline*}
where
\begin{align*}
z'&=z+vy_1-uy_2+\frac{1}{2}x_3 y_1-\frac{1}{2}x_1 y_2-uvx_2-\frac{1}{2}vx_1 x_2-\frac{1}{2}ux_2 x_3-\frac{1}{3}x_1 x_2 x_3,\\
y'_1&=y_1-ux_2-\frac{x_1 x_2}{2},\\
y'_2&=y_2+vx_2+\frac{x_2 x_3}{2}.
\end{align*} 
Hence we have
\begin{equation}\label{1111}
\begin{aligned}
&f(\exp(u X_1+v X_3)\exp n(z,\cdots,x_3))\\
&=f(\exp n(z',y'_1,y'_2,0,x_2,0)\exp((u+x_1) X_1+(v+x_3) X_3))\\
&=\chi_{l(0,\varepsilon,1,0,0,0)}(\exp n(z',y'_1,y'_2,0,x_2,0))f(\exp((u+x_1) X_1+(v+x_3) X_3))\\
&=e^{2\pi\sqrt{-1}(\varepsilon y'_{1}+y'_{2})}f(\exp((u+x_1) X_1+(v+x_3) X_3)).
\end{aligned}
\end{equation}
Combining formulas (\ref{2222}) and (\ref{1111}), we have the proposition.
\end{proof}

\textit{The case (II)}. We consider the space
\[
C^{\infty}_{\chi_{l(0,0,\varepsilon_{1},\varepsilon_{2},0,\varepsilon_{3})}}(S_{\text{(II)}}\backslash G/K).
\]
Here each $(\varepsilon_{1},\varepsilon_{2},\varepsilon_{3})$ corresponds to the case (II$_{i}$) $(i=1,\ldots,4)$ in Proposition \ref{classify} as follows,
\[
(\varepsilon_{1},\varepsilon_{2},\varepsilon_{3})=
\begin{cases}
(1,1,0)&\text{if the case (II$_{1}$)},\\
(0,1,1)&\text{if the case (II$_{2}$)},\\
(0,1,0)&\text{if the case (II$_{3}$)},\\
(0,0,1)&\text{if the case (II$_{4}$)}.
\end{cases}
\]
The subalgebra $\mathfrak{s}_{\text{(II)}}=\{\R X_{1}+\R X_{3}+\R Y_{1}+\R Y_{2}+\R Z\}$ is a codimension $1$ subalgebra of $\mathfrak{n}$. Because any codimension $1$ subalgebra of a nilpotent Lie algebra becomes an ideal (cf. Lemma 1.1.8 in \cite{C-G}), the subalgebra $\mathfrak{s}_{\text{(II)}}$ is an ideal of $\mathfrak{n}$. By the similar argument as in the case (I), the homogeneous space $S_{\text{(II)}}\backslash N$ is isomorphic to the subgroup $\{\exp(uX_{2})\ |\ u\in\R\}$ of $N$. This isomorphism gives a smooth section $\theta_{(\text{II})}\colon S_{\text{(II)}}\backslash N\rightarrow N$. Then we have a linear bijection
\[
\Xi_{\text{(II)}}\colon C^{\infty}_{\chi_{l(0,0,\varepsilon_{1},\varepsilon_{2},0,\varepsilon_{3})}}(S_{\text{(II)}}\backslash G/K)\xrightarrow{\sim} C^{\infty}(S_{\text{(II)}}\backslash N\times A).
\]
We introduce a coordinate system on $S_{\text{(II)}}\backslash N\times A$ as follows,
\[
\begin{array}{ccc}
\R\times(\R_{>0})^{4}&\xrightarrow{\sim}&S_{\text{(II)}}\backslash N\times A\\
u\times(a_{1},a_{2},a_{3},a_{4})&\longmapsto&\exp(uX_{2})\times\diag(a_{1},a_{2},a_{3},a_{4})
\end{array}
\] 
Then we can write down the action of $\mathfrak{n}$ on $C^{\infty}(S_{\text{(II)}}\backslash N\times A)$.
\begin{prop}\label{n-difII}
We regard the space $C^{\infty}(S_{(II)}\backslash N\times A)$ as the image of the space $C^{\infty}_{\chi_{l(0,0,\varepsilon_{1},\varepsilon_{2},0,\varepsilon_{3})}}(S_{\text{(II)}}\backslash G/K)$ by the mapping $\Xi_{(II)}$. Then $\mathfrak{n}$ acts on $C^{\infty}(S_{(II)}\backslash N\times A)$ as follows,
\begin{align*}
E_{21}F&=(2\pi\sqrt{-1}(\frac{a_{2}}{a_{1}})\varepsilon_{2})F, &E_{31}F&=0,\\
E_{41}F&=0,& E_{32}F&=\frac{a_{3}}{a_{2}}\frac{\partial}{\partial u}F,\\
E_{42}F&=(2\pi \sqrt{-1}(\frac{a_{4}}{a_{2}})\varepsilon_{1})F,&E_{43}F&=(2\pi\sqrt{-1}(\frac{a_{4}}{a_{3}})(\varepsilon_{3}-\varepsilon_{1} u))F.
\end{align*}
Here $F\in C^{\infty}(S_{\text{(II)}}\backslash N\times A)$ and 
\[
(\varepsilon_{1},\varepsilon_{2},\varepsilon_{3})=
\begin{cases}
(1,1,0)&\text{if the case (II$_{1}$)},\\
(0,1,1)&\text{if the case (II$_{2}$)},\\
(0,1,0)&\text{if the case (II$_{3}$)},\\
(0,0,1)&\text{if the case (II$_{4}$)}.
\end{cases}
\]
\end{prop}
\begin{proof}
The proposition can be obtained in the same way as the case (I) by the formula,
\[
\exp(uX_2)\cdot \exp n(z,\cdots,x_3)=\exp(n(z',y'_1,y'_2,x_1,0,x_3))\cdot \exp(u+x_2),
\]
where
\begin{align*}
z'&=z+\frac{1}{6}x_1 x_2 x_3\\
y'_1&=y_1+x_1 u+\frac{1}{2}x_1 x_2\\
y'_2&=y_2-x_3 u-\frac{1}{2}x_2 x_3.
\end{align*}

\end{proof}

The case (III). We consider the space
\[
C^{\infty}_{\chi_{l(0,0,0,\varepsilon_{1},\varepsilon_{2},\varepsilon_{3})}}(N\backslash G/K)
\]
where 
\[
(\varepsilon_{1},\varepsilon_{2},\varepsilon_{3})=
\begin{cases}
(1,1,1)&\text{if the case (III$_{1}$)},\\
(1,1,0)&\text{if the case (III$_{2}$)},\\
(1,0,1)&\text{if the case (III$_{3}$)},\\
(0,1,1)&\text{if the case (III$_{4}$)},\\
(1,0,0)&\text{if the case (III$_{5}$)},\\
(0,1,0)&\text{if the case (III$_{6}$)},\\
(0,0,1)&\text{if the case (III$_{7}$)},\\
(0,0,0)&\text{if the case (III$_{8}$)}.
\end{cases}
\]
By the Iwasawa decomposition, we have the linear bijection
\[
\Xi_{\text{(III)}}\colon C^{\infty}_{\chi_{l(0,0,0,\varepsilon_{1},\varepsilon_{2},\varepsilon_{3})}}(N\backslash G/K)\ni f\mapsto f|_{A}\in C^{\infty}(A).
\] 
by the restriction to $A$. Then we have the following proposition.
\begin{prop}\label{n-dimIII}
Let us consider the space $C^{\infty}(A)$ as image of the space $C^{\infty}_{\chi_{l(0,0,0,\varepsilon_{1},\varepsilon_{2},\varepsilon_{3})}}(N\backslash G/K)$ by the mapping $\Xi_{(III)}$. Then $\mathfrak{n}$ acts on $C^{\infty}(A)$ as follows,
\begin{align*}
E_{21}F&=(2\pi \sqrt{-1}(\frac{a_{2}}{a_{1}})\varepsilon_1)F, &E_{31}F&=0,\\
E_{41}F&=0,& E_{32}F&=(2\pi\sqrt{-1}(\frac{a_{3}}{a_{2}})\varepsilon_2)F,\\
E_{42}F&=0,&E_{43}F&=(2\pi\sqrt{-1}(\frac{a_{4}}{a_{3}})\varepsilon_3)F.
\end{align*}
Here $F\in C^{\infty}(A)$ and 
\[
(\varepsilon_{1},\varepsilon_{2},\varepsilon_{3})=
\begin{cases}
(1,1,1)&\text{if the case (III$_{1}$)},\\
(1,1,0)&\text{if the case (III$_{2}$)},\\
(1,0,1)&\text{if the case (III$_{3}$)},\\
(0,1,1)&\text{if the case (III$_{4}$)},\\
(1,0,0)&\text{if the case (III$_{5}$)},\\
(0,1,0)&\text{if the case (III$_{6}$)},\\
(0,0,1)&\text{if the case (III$_{7}$)},\\
(0,0,0)&\text{if the case (III$_{8}$)}.
\end{cases}
\]
\end{prop}
\begin{proof}
It is obvious from the formula, 
\[
(E_{ij}F)(a)=\frac{d}{dt}F(a\exp(tE_{ij}))|_{t=0}=\frac{d}{dt}F(\exp(t\text{Ad}(a)E_{ij})a)|_{t=0}.
\]
for $1\le i\neq j\le4$.
\end{proof}

\subsection{Generalized Whittaker models of $GL(4,\R)$}\label{explicit}
After these preparations given in the previous sections, we can investigate explicit structures of the generalized Whittaker models. More precisely to say, we will give the dimensions of these spaces and their basis in terms of the hypergeometric functions one and two variables.

\subsubsection{The embeddings into the spaces (I)}
We will give the explicit structures of the embeddings of the Harish-Chandra modules $X_{k,\lambda}$ ($k=1,2$) into the spaces $(\text{I}_{1})$ and $(\text{I}_{2})$ which are classified in Proposition \ref{classify}. These spaces are isomorphic to $C^{\infty}_{\chi_{l(0,1,\varepsilon,0,0,0)}}(S_{\text{(I)}}\backslash G/K\,;\,I_{k}(\lambda))$ by Theorem \ref{Yamashita}. We will consider the image of $C^{\infty}_{\chi_{l(0,1,\varepsilon,0,0,0)}}(S_{\text{(I)}}\backslash G/K\,;\,I_{k}(\lambda))$ by the mapping $\Xi_{(\text{I})}$ defined in Section \ref{diff}. Here $\varepsilon=1$ (resp. $=0$) for (I$_{1}$) (resp. (I$_{2}$)). Hence our purpose is to investigate these spaces  
\[
C^{\infty}_{\chi_{l(0,1,\varepsilon,0,0,0)}}(S_{\text{(I)}}\backslash N\times A;\,I_{k}(\lambda))=\Xi_{(I)}(C^{\infty}_{\chi_{l(0,1,\varepsilon,0,0,0)}}(S_{\text{(I)}}\backslash G/K\,;\,I_{k}(\lambda))).
\]

\begin{prop}\label{diff-op-2-1}
For $\varepsilon=0,1$, we consider the spaces $C^{\infty}_{\chi_{l(0,1,\varepsilon,0,0,0)}}(S_{\text{(I)}}\backslash N\times A;\,I_{k}(\lambda))$. Then these are the solution spaces of the following systems of the differential equations on $C^{\infty}(S_{\text{(I)}}\backslash N\times A)$.
\begin{gather}
\begin{split}
[\vartheta_{a_{1}}^{2}-&(\lambda_{1}+\lambda_{2}+k-3)\vartheta_{a_{1}}+\lambda_{1}(\lambda_{2}+k)+(\frac{a_{2}}{a_{1}})^{2}\frac{\partial^{2}}{\partial u^{2}}\\ &+\varepsilon(\frac{a_{3}}{a_{1}})^{2}(2\pi\sqrt{-1})^{2}-(\vartheta_{a_{2}}+\vartheta_{a_{3}}+\vartheta_{a_{4}})+\lambda_{1}(\lambda_{2}+k)]\phi=0,
\end{split}\label{1-1}\\
[\frac{\partial}{\partial u}(\vartheta_{a_{1}}+\vartheta_{a_{2}}-(\lambda_{1}+\lambda_{2}+k-3))+\varepsilon(\frac{a_{3}}{a_{2}})^{2}(2\pi\sqrt{-1})^{2}(v-\varepsilon u)]\phi=0,\label{1-2}\\
[\varepsilon(\vartheta_{a_{1}}+\vartheta_{a_{3}}-(\lambda_{1}+\lambda_{2}+k-2))+(v-\varepsilon u)\frac{\partial}{\partial u}]\phi=0,\label{1-3}\\
[\frac{\partial}{\partial u}+\varepsilon\frac{\partial}{\partial v}]\phi=0,\label{1-4}\\
\begin{split}
[\vartheta_{a_{2}}^{2}-(\lambda_{1}+\lambda_{2}+k-2)\vartheta_{a_{2}}+(\frac{a_{2}}{a_{1}})^{2}\frac{\partial ^{2}}{\partial u^{2}}+(\frac{a_{3}}{a_{2}})^{2}(2\pi\sqrt{-1})^{2}(v-\varepsilon u)^{2}\\+(\frac{a_{4}}{a_{2}})^{2}(2\pi\sqrt{-1})^{2}-(\vartheta_{a_{3}}+\vartheta_{a_{4}})+\lambda_{1}(\lambda_{2}+k)]\phi=0,
\end{split}\label{2-2}\\
[(v-\varepsilon u)(\vartheta_{a_{2}}+\vartheta_{a_{3}}-(\lambda_{1}+\lambda_{2}+k-2))+\varepsilon(\frac{a_{2}}{a_{1}})^{2}\frac{\partial}{\partial u}+(\frac{a_{4}}{a_{3}})^{2}\frac{\partial}{\partial v}]\phi=0,\label{2-3}\\
[(\vartheta_{a_{2}}+\vartheta_{a_{4}}-(\lambda_{1}+\lambda_{2}+k-2))+(v-\varepsilon u)\frac{\partial}{\partial v}]\phi=0,\label{2-4}\\
\begin{split}
[\vartheta_{a_{3}}^{2}-&(\lambda_{1}+\lambda_{2}+k-1)\vartheta_{a_{3}}+\varepsilon(\frac{a_{3}}{a_{1}})^{2}(2\pi\sqrt{-1})^{2}\\&+(\frac{a_{3}}{a_{2}})^{2}(2\pi\sqrt{-1})^{2}(v-\varepsilon u)^{2}+(\frac{a_{4}}{a_{3}})^{2}\frac{\partial^{2}}{\partial v^{2}}-\vartheta_{a_{4}}+\lambda_{1}(\lambda_{2}+k)]\phi=0,
\end{split}\label{3-3}\\
[\frac{\partial}{\partial v}(\vartheta_{a_{3}}+\vartheta_{a_{4}}-(\lambda_{1}+\lambda_{2}+k-1))+(\frac{a_{3}}{a_{2}})^{2}(2\pi\sqrt{-1})^{2}(v-\varepsilon u)]\phi=0,\label{3-4}\\
[\vartheta_{a_{4}}^{2}-(\lambda_{1}+\lambda_{2}+k)\vartheta_{a_{4}}+(\frac{a_{4}}{a_{2}})^{2}(2\pi\sqrt{-1})^{2}+(\frac{a_{4}}{a_{3}})^{2}\frac{\partial^{2}}{\partial v^{2}}+\lambda_{1}(\lambda_{2}+k)]\phi=0,\label{4-4}\\
[(\vartheta_{a_{1}}+\vartheta_{a_{2}}+\vartheta_{a_{3}}+\vartheta_{a_{4}}-k\lambda_{1}-(4-k)\lambda_{2})]\phi=0.\label{central}
\end{gather}
Here $\phi\in C^{\infty}(S_{\text{(I)}}\backslash N\times A)$.
\end{prop}
\begin{proof}
Recall that $I_{k}(\lambda)$ is written as (\ref{generator4}). Then these differential equations immediately follows from Lemma \ref{symmetric}, Proposition \ref{modk} and Proposition \ref{n-difI}.
\end{proof}
We solve these systems of the differential equations.

(i) \textit{The case $\varepsilon=1$.}

\noindent
We investigate the space $C^{\infty}_{\chi_{l(0,1,1,0,0,0)}}(S_{\text{(I)}}\backslash N\times A;\,I_{k}(\lambda))$, i.e, the case $\varepsilon=1$.

We consider the case $k=1$, i.e., the embedding of the Harish-Chandra module $X_{1,\lambda}$.
\begin{prop}
For the case $\varepsilon=1$ and $k=1$, the solution space of the system of the differential equations defined in Proposition \ref{diff-op-2-1} is $\{0\}$.
\end{prop}
\begin{proof}
The equations (\ref{1-3}) and (\ref{2-4}) give us the equation,
\[
[\vartheta_{a_{1}}+\vartheta_{a_{2}}+\vartheta_{a_{3}}+\vartheta_{a_{4}}-2\lambda_{1}-2\lambda_{2}+2+2(v-u)(\frac{\partial}{\partial u}+\frac{\partial}{\partial v})]\phi=0.
\] 
By the equation (\ref{1-4}), the term $\frac{\partial}{\partial u}+\frac{\partial}{\partial v}$ can be eliminated. Hence the only remaining term is 
\[
[\vartheta_{a_{1}}+\vartheta_{a_{2}}+\vartheta_{a_{3}}+\vartheta_{a_{4}}-2\lambda_{1}-2\lambda_{2}+2]\phi=0.
\]
However if we compare this with the equation (\ref{central}),
\[
[\vartheta_{a_{1}}+\vartheta_{a_{2}}+\vartheta_{a_{3}}+\vartheta_{a_{4}}-\lambda_{1}-3\lambda_{2}]\phi=0,
\]
 we can conclude the solution space of these equations must be $\{0\}$ because we assume $\lambda_{1}-\lambda_{2}\notin \Z$.
\end{proof}

Next, we consider the case $k=2$, i.e., the embedding of the Harish-Chandra module $X_{2,\lambda}$.
We introduce a new coordinate system below,
\begin{equation}\label{change1}
\begin{gathered}
x_{1}=a_{1}a_{2}a_{3}a_{4},\\
x_{2}=(\pi\sqrt{-1})^{2}\left((\frac{a_{3}}{a_{2}})^{2}(v-u)^{2}+(\frac{a_{4}}{a_{2}})^{2}+(\frac{a_{3}}{a_{1}})^{2}\right),\\
x_{3}=\left(\frac{a_{1}a_{3}}{a_{2}a_{4}}(v-u)^{2}+\frac{a_{2}a_{3}}{a_{1}a_{4}}+\frac{a_{1}a_{4}}{a_{2}a_{3}}\right)^{-2},\\
x_{4}=\frac{a_{1}a_{3}}{a_{2}a_{4}},\\
x_{5}=\frac{a_{1}a_{4}}{a_{2}a_{3}},\\
x_{6}=u.
\end{gathered}
\end{equation}

\begin{prop}\label{nanikana}
Let $\varepsilon=1$ and $k=2$. We consider the system of the differential equations in Proposition \ref{diff-op-2-1}. By adding and substituting each differential equations and multiplying some rational functions, the system of the differential equations under the new coordinate system $x_{1},\ldots,x_{6}$ is written as followings ,
\begin{gather}
(\vartheta_{x_{1}}-\frac{\lambda_{1}+\lambda_{2}}{2})\phi=0,\\
[x_{2}-(\vartheta_{x_{2}}-\frac{1}{2})(2\vartheta_{x_{3}}-\vartheta_{x_{2}})]\phi=0,\label{horn1}\\
\begin{split}
[x_{3}(\vartheta_{x_{2}}-2\vartheta_{x_{3}})&(\vartheta_{x_{2}}-2\vartheta_{x_{3}}-1)\\
&-(\vartheta_{x_{3}}-\frac{1}{4}(\lambda_{2}-\lambda_{1})-1)(\vartheta_{x_{3}}+\frac{1}{4}(\lambda_{2}-\lambda_{1}))]\phi=0,
\end{split}\label{horn2}\\
\frac{\partial}{\partial x_{4}}\phi=0,\\
\frac{\partial}{\partial x_{5}}\phi=0,\\
\frac{\partial}{\partial x_{6}}\phi=0.
\end{gather}
\end{prop}
\begin{proof}
First, we put 
\begin{align*}
\alpha_{1}&=a_{1}a_{2},& \alpha_{2}&=a_{1}a_{2}^{-1}, &\alpha_{3}&=a_{3}a_{4},\\
\alpha_{4}&=a_{3}a_{4}^{-1},& u'&=u, &v'&=(v-u).
\end{align*}
Then the differential equation (\ref{1-4}) becomes
\begin{equation}\label{uprim}
\frac{\partial}{\partial u'}\phi=0.
\end{equation}
Furthermore we exchange the variables $\alpha_{2},\alpha_{4},v'$ to
\begin{align*}
w&=\alpha_{2}\alpha_{4}v'^{2}+\alpha_{2}\alpha_{4}^{-1}+\alpha_{2}^{-1}\alpha_{4}, &\beta_{2}&=\alpha_{2}\alpha_{4}, &\beta_{4}=\alpha_{2}\alpha_{4}^{-1}.
\end{align*}
Then equations (\ref{2-3}) and (\ref{2-4}) become
\begin{gather}
\vartheta_{\beta_{4}}\phi=0\label{2-32},\\
\vartheta_{\beta_{2}}\phi=0\label{2-42},
\end{gather}
respectively. Setting 
\begin{align*}
\beta_{1}&=\alpha_{1}\alpha_{3},&\beta_{3}=\alpha_{1}\alpha_{3}^{-1},
\end{align*}
the equation (\ref{central}) becomes
\begin{equation}\label{central2}
(2\vartheta_{\beta_{1}}-(\lambda_1+\lambda_{2}))\phi=0.
\end{equation}
Also we can see that the equation (\ref{1-2}) is written as 
\begin{equation}
[2w\frac{\partial}{\partial w}(2(\vartheta_{\beta_{1}}+\vartheta_{\beta_{3}})-(\lambda_{1}+\lambda_{2}-1))-(2\pi\sqrt{-1})^{2}\beta_{3}^{-1}w]\phi=0.\label{1-22}
\end{equation}
If we eliminate $\vartheta_{\beta_{1}}$ from (\ref{1-22}) by using the equation (\ref{central2}), it can be written as
\begin{equation}\label{1-23}
[2w\frac{\partial}{\partial w}(2\vartheta_{\beta_{3}}+1)-(2\pi\sqrt{-1})^{2}\beta_{3}^{-1}w]\phi=0.
\end{equation}
We note that the equation (\ref{3-4}) can be reduced to the same equation.
Taking into account the equations (\ref{central2}) and (\ref{1-23}), the equation (\ref{1-1}) can be reduced to the equation
\begin{equation}\label{1-12}
[(\vartheta_{\beta_{3}}+\vartheta_{w}-\frac{1}{2}(\lambda_{1}-\lambda_{2}-4))(\vartheta_{\beta_{3}}+\vartheta_{w}+\frac{1}{2}(\lambda_{1}-\lambda_{2}))-4\frac{\partial^{2}}{\partial w^{2}}]\phi=0.
\end{equation}
We can also see that the equations (\ref{2-2}), (\ref{3-3}) and (\ref{4-4}) can be written as the same equation (\ref{1-12}).
Finally, we put
\begin{align*}
\gamma_{1}&=(\pi\sqrt{-1})^{2}\beta_{3}^{-1}w,& \gamma_{2}=w^{-2}.
\end{align*}
Then the equation (\ref{1-23}) is equivalent to 
\begin{equation}\label{hornn1}
[(\vartheta_{\gamma_{1}}-2\vartheta_{\gamma_{2}})(\frac{1}{2}-\vartheta_{\gamma_{1}})-\gamma_{1}]\phi=0.
\end{equation}
Also the equation (\ref{1-12}) is written as
\begin{equation}\label{hornn2}
[(\vartheta_{\gamma_{2}}-\frac{1}{4}(\lambda_{2}-\lambda_{1})-1)(\vartheta_{\gamma_{2}}+\frac{1}{4}(\lambda_{2}-\lambda_{1}))-\gamma_{2}(\vartheta_{\gamma_{1}}-2\vartheta_{\gamma_{2}})(\vartheta_{\gamma_{1}}-2\vartheta_{\gamma_{2}}-1)]\phi=0.
\end{equation}
If we put 
\begin{align*}
x_{1}&=\beta_{1},&x_{2}&=\gamma_{1},&x_{3}&=\gamma_{2},\\
x_{4}&=\beta_{2},&x_{5}&=\beta_{4}, &x_{6}&=u',
\end{align*}
then the theorem follows from the equations (\ref{uprim}), (\ref{2-32}), (\ref{2-42}), (\ref{central2}), (\ref{hornn1}) and (\ref{hornn2}).

\end{proof}
\begin{cor}
The change of variables (\ref{change1}) gives a diffeomorphism from $\{(a_{1},\ldots,a_{4},u,v)\in\R^{6}\mid a_{i}\in\R_{>0}\ (i=1,\ldots,4)\}\cong S_{(I)}\backslash N\times A$ to the domain $D_{1}=\{(x_{1},\ldots,x_{6})\mid x_{i}\in\R_{>0}\,(i=1,3,4,5),x_{2}\in\R_{<0},x_{6}\in\R\}$.
\end{cor}
\begin{proof}
We should show that this gives bijection and the Jacobi determinant is not zero, i.e.,
\[
\left|\frac{\partial(x_{1},\ldots,x_{6})}{\partial(a_{1},\ldots,a_{4},u,v)}\right|(p)=
\left|
\begin{array}{ccccc}
\frac{\partial x_{1}}{\partial a_{1}}(p)&\cdots&\frac{\partial x_{1}}{\partial a_{4}}(p)&\frac{\partial x_{1}}{\partial u}(p)&\frac{\partial x_{1}}{\partial v}(p)\\
\vdots&&\vdots&\vdots&\vdots\\
\frac{\partial x_{6}}{\partial a_{1}}(p)&\cdots&\frac{\partial x_{6}}{\partial a_{4}}(p)&\frac{\partial x_{6}}{\partial u}(p)&\frac{\partial x_{6}}{\partial v}(p)
\end{array}
\right|\neq 0,
\]
for any $p\in\{(a_{1},\ldots,a_{4},u,v)\in\R^{6}\mid a_{i}\in\R_{>0}\ (i=1,\ldots,4)\}.$ Here $|X|$ means the determinant of $X\in M(6,\R)$. As we see in the proof of the previous proposition, this change of variables is the composition of the following change of variables.

Step1.
\begin{align*}
\alpha_{1}&=a_{1}a_{2},& \alpha_{2}&=a_{1}a_{2}^{-1}, &\alpha_{3}&=a_{3}a_{4},\\
\alpha_{4}&=a_{3}a_{4}^{-1},& u'&=u, &v'&=(v-u).
\end{align*}
Here we can see this gives a bijection from $\{(a_{1},\ldots,a_{4},u,v)\in\R^{6}\mid a_{i}\in\R_{>0}\ (i=1,\ldots,4)\}$ to $\{(\alpha_{1},\ldots,\alpha_{4},u',v')\in \R^{6}\mid \alpha_{i}\in \R_{>0}\ (i=1,\ldots,4)\}$.

Step2.
\begin{align*}
\beta_{1}&=\alpha_{1}\alpha_{3},&\beta_{2}&=\alpha_{2}\alpha_{4},&\beta_{3}&=\alpha_{1}\alpha_{3}^{-1},\\
\beta_{4}&=\alpha_{2}\alpha_{4}^{-1},&w&=\alpha_{2}\alpha_{4}v'^{2}+\alpha_{2}\alpha_{4}^{-1}+\alpha_{2}^{-1}\alpha_{4} ,&u'&=u'.
\end{align*}
Here we can see this gives a bijection from $\{(\alpha_{1},\ldots,\alpha_{4},u',v')\in \R^{6}\mid \alpha_{i}\in \R_{>0}\ (i=1,\ldots,4)\}$ to $\{(\beta_{1},\ldots,\beta_{4},w,u')\in \R^{6}\mid \beta_{i},w\in \R_{>0}\ (i=1,\ldots,4)\}$.

Step3.
\begin{align*}
x_{1}&=\beta_{1},&x_{2}&=(\pi\sqrt{-1})^{2}\beta_{3}^{-1}w,&x_{3}&=w^{-2},\\
x_{4}&=\beta_{2},&x_{5}&=\beta_{4},&x_{6}&=u'.
\end{align*}
Here we can see this gives a bijection from $\{(\beta_{1},\ldots,\beta_{4},w,u')\in \R^{6}\mid \beta_{i},w\in \R_{>0}\ (i=1,\ldots,4)\}$ to $\{(x_{1},\ldots,x_{6}')\in \R^{6}\mid x_{i}\in\R_{>0}\,(i=1,3,4,5),x_{2}\in\R_{<0},x_{6}\in\R\}$.

Also it is not hard to see that
\begin{multline*}
\left|\frac{\partial(x_{1},\ldots,x_{6})}{\partial(a_{1},\ldots,a_{4},u,v)}\right|(p)\\
=\left|\frac{\partial(x_{1},\ldots,x_{6})}{\partial(\beta_{1},\ldots,\beta_{4},w,u')}\frac{\partial(\beta_{1},\ldots,\beta_{4},w,u')}{\partial(\alpha_{1},\ldots,\alpha_{4},u',v')}\frac{\partial(\alpha_{1},\ldots,\alpha_{4},u',v')}{\partial(a_{1},\ldots,a_{4},u,v)}\right|(p)\neq 0,
\end{multline*}
for any $p\in\{(a_{1},\ldots,a_{4},u,v)\in\R^{6}\mid a_{i}\in\R_{>0}\ (i=1,\ldots,4)\}.$ Thus we have the proposition.
\end{proof}
Pick up the differential equations (\ref{horn1}) and (\ref{horn2}). We take $f(x_{2},x_{3})$ as a solution of them. We take a function $F(x_{2},x_{3})$ such that 
\[
f=x_{2}^{\frac{1}{2}}x_{3}^{\frac{1}{4}(\lambda_{1}-\lambda_{2})}F.
\]
Then this $F(x_{2},x_{3})$ satisfies
\begin{gather}
[x_{2}-\vartheta_{x_{2}}(2\vartheta_{x_{3}}-\vartheta_{x_{2}}+\frac{1}{2}(\lambda_{1}-\lambda_{2}-1))]F(x_{2},x_{3})=0,\\
\begin{split}
[x_{3}(2\vartheta_{x_{3}}-\vartheta_{x_{2}}+\frac{1}{2}(\lambda_{1}-\lambda_{2}-1))(2\vartheta_{x_{3}}-\vartheta_{x_{2}}+\frac{1}{2}(\lambda_{1}-\lambda_{2}-1)+1)\\
-\vartheta_{x_{3}}(\vartheta_{x_{3}}+\frac{1}{2}(\lambda_{1}-\lambda_{2})-1)]F(x_{2},x_{3})=0.
\end{split}
\end{gather} 
These are the differential equations for Horn's hypergeometric function $\mathrm{H}_{10}(\frac{1}{2}(\lambda_{1}-\lambda_{2}-1),\frac{1}{2}(\lambda_{1}-\lambda_{2});x_{2},x_{3})$ (cf. \cite{Horn}). We denote it by $\mathfrak{H}_{10}(a,d;x,y)$ the solution space of the system of the partial differential equations for Horn's hypergeometric function $\mathrm{H}_{10}(a,d;x,y)$, i.e., 
\begin{gather*}
[x(2\vartheta_{x}-\vartheta_{y}+a)(2\vartheta_{x}-\vartheta_{y}+a+1)-\vartheta_{x}(\vartheta_{x}+d-1)]f(x,y)=0,\\
[y-\vartheta_{y}(2\vartheta_{x}-\vartheta_{y}+a)]f(x,y)=0.
\end{gather*}
We can see more detailed properties of $\mathfrak{H}_{10}(a,d;x,y)$ in Appendix.
\begin{df}
Let $U \subset \R^{n}$ be a domain. A function $f(x)$ on $U$ is called rapidly decreasing on $U$ if it satisfies,
\[
\sup_{x\in U}|x^{\alpha}f(x)|<\infty
\]
for any $\alpha=(\alpha_{1},\ldots,\alpha_{n})\in \N^{n}$, where $x^{\alpha}=x_{1}^{\alpha_{1}}\cdots x_{n}^{\alpha_{n}}$.

Also a function $f(x)$ on $U$ is called slowly increasing on $U$ if there exists $N\in \N$ such that
\[
\sup_{x\in U}(1+|x|)^{-N}|f(x)|<\infty
\]
where $|x|=\sqrt{x_{1}^{2}+\cdots +x_{n}^{2}}$ for $x\in U$.
\end{df}

\begin{thm}\label{thmI-1}
The space $C^{\infty}_{\chi_{l(0,1,1,0,0,0)}}(S_{\text{(I)}}\backslash N\times A;\,I_{k}(\lambda))$ consists of the elements, 
\[
x_{1}^{\frac{\lambda_{1}+\lambda_{2}}{2}}x_{2}^{\frac{1}{2}}x_{3}^{\frac{1}{4}(\lambda_{1}-\lambda_{2})}f(x_{2},x_{3}),
\] 
for $f(x,y)\in\mathfrak{H}_{10}(\frac{1}{2}(\lambda_{1}-\lambda_{2}-1),\frac{1}{2}(\lambda_{1}-\lambda_{2});x,y)$.
Here 
\begin{gather*}
x_{1}=a_{1}a_{2}a_{3}a_{4},\\
x_{2}=(\pi\sqrt{-1})^{2}\left((\frac{a_{3}}{a_{2}})^{2}(v-u)^{2}+(\frac{a_{4}}{a_{2}})^{2}+(\frac{a_{3}}{a_{1}})^{2}\right),\\
x_{3}=\left(\frac{a_{1}a_{3}}{a_{2}a_{4}}(v-u)^{2}+\frac{a_{2}a_{3}}{a_{1}a_{4}}+\frac{a_{1}a_{4}}{a_{2}a_{3}}\right)^{-2},\\
x_{4}=\frac{a_{1}a_{3}}{a_{2}a_{4}},\\
x_{5}=\frac{a_{1}a_{4}}{a_{2}a_{3}},\\
x_{6}=u.
\end{gather*}
Thus we have
\[
\mathrm{dim}_{\C}C^{\infty}_{\chi_{l(0,1,1,0,0,0)}}(S_{\text{(I)}}\backslash N\times A;\,I_{k}(\lambda))=4.
\]
In $C^{\infty}_{\chi_{l(0,1,1,0,0,0)}}(S_{\text{(I)}}\backslash N\times A;\,I_{k}(\lambda))$, there is a rapidly decreasing function on $\{(x_{2},\ldots,x_{6})\mid x_{2}\in\R_{<0}, x_{3},x_{4},x_{5}\in\R_{> 0},x_{6}\in\R\}$ and it is unique up to constant multiple.

\end{thm}
\begin{proof}
This follows immediately from Proposition \ref{nanikana} and the argument above.The second assertion follows from Appendix B and Thorem \ref{multhorn}.
\end{proof}

(ii) \textit{the case $\varepsilon=0$.}

\noindent
We investigate the space  $C^{\infty}_{\chi_{l(0,1,0,0,0,0)}}(S_{\text{(I)}}\backslash N\times A;\,I_{k}(\lambda))$, i.e, the case $\varepsilon=0.$

We introduce a new cooridinate system as below,
\begin{align*}
x_{1}&=a_{1},& x_{2}&=a_{2},\\
x_{3}&=a_{4}^{2}+a_{3}^{2}v^{2}, &x_{4}&=(a_{3}^{-2}+a_{4}^{-2}v^{2})^{-1},\\
x_{5}&=a_{3}a_{4}^{-1}, &x_{6}&=u.
\end{align*}
\begin{lem}
The change of variables given above is the diffeomorphism from $\{(a_{1},\ldots,a_{4},u,v)\in\R^{6}\mid a_{i}\in\R_{>0}\ (i=1,\ldots,4)\}\cong S_{(I)}\backslash N\times A$ to the domain $D_{2}=\{(x_{1},\ldots,x_{6})\mid x_{i}\in\R_{>0}\,(i=1,\ldots,5),\,x_{6}\in\R\}$.
\end{lem}
\begin{proof}
We can see that this is a bijection. Also it is not hard to see that
\[
\left|\frac{\partial(x_{1},\ldots,x_{6})}{\partial(a_{1},\ldots,a_{4},u,v)}\right|(p)\neq 0,
\]
for $p\in\{(a_{1},\ldots,a_{4},u,v)\in\R^{6}\mid a_{i}\in\R_{>0}\ (i=1,\ldots,4)\}$ by direct computation.
\end{proof}
Then the systems of the differential equations in Proposition \ref{diff-op-2-1} are reduced to the followings 
\begin{prop}\label{diffop0-k}
For $k=1,2$, we consider the systems of the differential equations in Proposition \ref{diff-op-2-1}. Then by adding and substituting each differential equations and multiplying some rational functions, these systems are written on $\{(a_{1},\ldots,a_{4},u,v)\in\R^{6}\mid a_{i}\in\R_{>0}\ (i=1,\ldots,4)\}$ as follows.
\begin{gather}
(\vartheta_{x_{1}}-(\lambda_{1}-(4-k)))(\vartheta_{x_{1}}-\lambda_{2})\phi=0,\label{111}\\
\begin{split}
[\vartheta_{x_{2}}^{2}-(\lambda_{1}+\lambda_{2}&-3+k)\vartheta_{x_{2}}+\vartheta_{x_{1}}\\
&+(2\pi\sqrt{-1})^{2}x_{2}^{-2}x_{3}+\lambda_{2}(\lambda_{1}-(4-k))]\phi=0,
\end{split}\label{112}\\
\begin{split}
[4(\vartheta_{x_{3}}^{2}+\vartheta_{x_{4}}^{2})-2(\lambda_{1}+\lambda_{2}-&1+k)\vartheta_{x_{3}}+2\vartheta_{x_{4}}\\
&+(2\pi\sqrt{-1})^{2}x_{2}^{-2}x_{3}+2\lambda_{1}(\lambda_{2}+k)]\phi=0,
\end{split}\label{113}\\
(2\vartheta_{x_{4}}-\lambda_{1})(2\vartheta_{x_{4}}-(\lambda_{2}+k))\phi=0,\label{114}\\
\vartheta_{x_{5}}\phi=0,\\
\vartheta_{x_{6}}\phi=0.
\end{gather}
\end{prop}
\begin{proof}
By the equations (\ref{1-1}) and (\ref{central}), we have the new equation
\[
[\vartheta_{a_{1}}^{2}-(\lambda_{1}+\lambda_{2}-4+k)\vartheta_{a_{1}}+\left(\frac{a_{2}}{a_{1}}\right)^{2}\frac{\partial^{2}}{\partial u^{2}} +\lambda_{2}(\lambda_{1}-(4-k))]\phi=0.
\]
By the equation (\ref{1-4}), we can eliminate the term $\frac{\partial}{\partial u}$ from the above equation, then it can be written as
\begin{equation}\label{1banme}
(\vartheta_{a_{1}}-(\lambda_{1}-(4-k)))(\vartheta_{a_{1}}-\lambda_{2})\phi=0.
\end{equation}
Next, from the equations (\ref{2-2}) and (\ref{central}), we have a new equation
\begin{equation}\label{2banme}
[\vartheta_{a_{2}}-(\lambda_{1}+\lambda_{2}-3+k)\vartheta_{a_{2}}+\vartheta_{a_{1}}+(2\pi\sqrt{-1})^{2}(a_{4}^{2}+a_{3}^{2}v^{2})a_{2}^{-2}+\lambda_{2}(\lambda_{1}-(4-k))]\phi=0,
\end{equation}
as well.
If we put $\alpha_{3}=a_{3}a_{4}$ and $\alpha_{4}=a_{3}a_{4}^{-1}$, we have a new equation 
\[
[2\vartheta_{\alpha_{4}}+(\alpha_{4}^{-2}v^{-1}-v)\frac{\partial}{\partial v}]\phi=0
\]
from the equation (\ref{2-3}) and (\ref{2-4}). Moreover if we put $w=\alpha_{4}^{-1}+\alpha_{4}v^{2}$, $\beta_{3}=\alpha_{3}$ and $\beta_{4}=\alpha_{4}$, the above equation is reduced to 
\begin{equation}\label{3banme}
\vartheta_{\beta_{4}}\phi=0.
\end{equation}
And the equation (\ref{3-4}) becomes
\begin{equation}\label{koreha}
[\vartheta_{w}(2\vartheta_{\beta_{3}}-(\lambda_{1}+\lambda_{2}-1+k))+\frac{1}{2}(2\pi\sqrt{-1})^{2}\beta_{3}wa_{2}^{-2}]\phi=0.
\end{equation}
Also if we consider the sum of the equations (\ref{3-3}) and (\ref{4-4}), we have a new equation
\begin{equation}\label{nanda}
[2\vartheta_{\beta_{3}}^{2}-2(\lambda_{1}+\lambda_{2}+k)\vartheta_{\beta_{3}}+2\vartheta_{w}^{2}+2\vartheta_{w}+(2\pi\sqrt{-1})^{2}w\beta_{3}a_{2}^{-2}+2\lambda_{1}(\lambda_{2}+k)]\phi=0.
\end{equation}
By the equation (\ref{koreha}), we can eliminate the term $\beta_{3}wa_{2}^{-2}$ from (\ref{nanda}). Then we have
\begin{equation}\label{5banme}
(\vartheta_{\beta_{3}}-\vartheta_{w}-\lambda_{1})(\vartheta_{\beta_{3}}-\vartheta_{w}-(\lambda_{2}+k))\phi=0.
\end{equation}
Hence if we put 
\begin{align}
x_{1}&=a_{1}, &x_{2}&=a_{2}, &x_{3}&=\beta_{3}w,\\
x_{4}&=\beta_{3}w^{-1},&x_{5}&=\beta_{2}, &x_{6}&=u,
\end{align}
then we have the proposition from the equations (\ref{1banme}), (\ref{2banme}), (\ref{3banme}), (\ref{nanda}), (\ref{5banme}) and (\ref{1-4}).
\end{proof}

Let $\mathfrak{MB}(\nu\,;\,x)$ be the solution space of the differential equation
\[
[\frac{d^{2}}{dx^{2}}+\frac{1}{x}\frac{d}{dx}-(1+\frac{\nu^{2}}{x^{2}})]f(x)=0,
\]
i.e., the solution space of the modified Bessel equation. We have $\dim_{\C}\mathfrak{MB}(\nu\,;\,x)=2$. In $\mathfrak{MB}(\nu\,;\,x)$, there is a series solution
\[
I_{\nu}(x)=\sum_{m=0}^{\infty}\frac{(\frac{x}{2})^{\nu+m}}{m!\Gamma(\nu+m+1)}.
\]
Also there is a solution as a slowly increasing function
\[
K_{\nu}(x)=\frac{\pi}{2}\frac{I_{-\nu}(x)-I_{\nu}(x)}{\sin\nu\pi},
\]
and any slowly increasing function in $\mathfrak{MB}(\nu\,;\,x)$ are constant multiples of $K_{\nu}(x)$.
\begin{thm}\label{thmI-2}
For the cases $k=1,2$, $C^{\infty}_{\chi_{l(0,1,0,0,0,0)}}(S_{\text{(I)}}\backslash N\times A;\,I_{k}(\lambda))$ are written as follows under the coordinate system, 
\begin{align*}
x_{1}&=a_{1},& x_{2}&=a_{2},\\
x_{3}&=a_{4}^{2}+a_{3}^{2}v^{2}, &x_{4}&=(a_{3}^{-2}+a_{4}^{-2}v^{2})^{-1},\\
x_{5}&=a_{3}a_{4}^{-1}, &x_{6}&=u.
\end{align*}

(i) For $k=1$, $C^{\infty}_{\chi_{l(0,1,0,0,0,0)}}(S_{\text{(I)}}\backslash N\times A;\,I_{1}(\lambda))$ consists of 
\[
x_{1}^{\lambda_{2}}x_{2}^{\frac{\lambda_{1}+\lambda_{2}-2}{2}}x_{3}^{\frac{\lambda_{1}+\lambda_{2}}{4}}x_{4}^{\frac{\lambda_{2}+1}{2}}f(2\pi x_{2}^{-1}\sqrt{x_{3}})
\]
for $f(x)\in\mathfrak{MB}(\frac{\lambda_{1}-\lambda_{2}-2}{2}\,;\,x)$. Thus we have
\[
\dim_{\C}C^{\infty}_{\chi_{l(0,1,0,0,0,0)}}(S_{\text{(I)}}\backslash N\times A;\,I_{1}(\lambda))=2, 
\]
and there is a slowly increasing function on $\{(x_{1},\ldots,x_{6})\mid x_{i}\in\R_{>0}\,(i=1,\ldots,5),\,x_{6}\in\R\}$ and it is unique up to constant multiple.

(ii) For $k=2$, $C^{\infty}_{\chi_{l(0,1,0,0,0,0)}}(S_{\text{(I)}}\backslash N\times A;\,I_{2}(\lambda))$ consists of
\begin{multline*}
Cx_{1}^{\lambda_{2}}x_{2}^{\frac{\lambda_{1}+\lambda_{2}-1}{2}}x_{3}^{\frac{\lambda_{1}+\lambda_{2}+1}{4}}x_{4}^{\frac{\lambda_{1}}{2}}f(2\pi x_{2}^{-1}\sqrt{x_{3}})\\
+C'x_{1}^{\lambda_{1}-2}x_{2}^{\frac{\lambda_{1}+\lambda_{2}-1}{2}}x_{3}^{\frac{\lambda_{1}+\lambda_{2}+1}{4}}x_{4}^{\frac{\lambda_{2}+2}{2}}g(2\pi x_{2}^{-1}\sqrt{x_{3}}),
\end{multline*}
where $C,C'\in\C$, $f(x)\in \mathfrak{MB}(\frac{\lambda_{1}-\lambda_{2}-1}{2}\,;\,x)$ and $g(x)\in \mathfrak{MB}(\frac{\lambda_{1}-\lambda_{2}-3}{2}\,;\,x)$. Thus we have
\[
\dim_{\C}C^{\infty}_{\chi_{l(0,1,0,0,0,0)}}(S_{\text{(I)}}\backslash N\times A;\,I_{2}(\lambda))=4,
\]
and there is a 2-dimensional subspace which consists of slowly increasing functions on $\{(x_{1},\ldots,x_{6})\mid x_{i}\in\R_{>0}\,(i=1,\ldots,5),\,x_{6}\in\R\}$.

\end{thm}
\begin{proof}
The solution space of the equation (\ref{111}) consists of the elements
\[
C_{1}\phi_{1}(x_{2},x_{3},x_{4})x_{1}^{\lambda_{1}-(4-k)}+C_{2}\phi_{2}(x_{2},x_{3},x_{4})x_{1}^{\lambda_{2}},
\]
where $C_{i}$ are constants and $\phi_{i}$ are arbitrary functions for $i=1,2$. We determine functions $\phi_{i}$ by the equation (\ref{112}). Then these functions satisfy following equations,
\begin{gather*}
[\vartheta_{x_{2}}^{2}-(\lambda_{1}+\lambda_{2}-3+k)\vartheta_{x_{2}}+(2\pi\sqrt{-1})^{2}x_{2}^{-2}x_{3}+(\lambda_{1}-4+k)(\lambda_{2}+1)]\phi_{1}=0,\\
[\vartheta_{x_{2}}^{2}-(\lambda_{1}+\lambda_{2}-3+k)\vartheta_{x_{2}}+(2\pi\sqrt{-1})^{2}x_{2}^{-2}x_{3}+(\lambda_{1}-3+k)\lambda_{2}]\phi_{2}=0.
\end{gather*}
We define the functions $\phi'_{i}$ so that $\phi_{i}=x_{2}^{\frac{\lambda_{1}+\lambda_{2}-3+k}{2}}\phi'_{i}$ for $i=1,2$, then $\phi'_{i}$ satisfy following equations,
\begin{gather*}
[\vartheta_{x_{2}}^{2}-((2\pi x_{2}^{-1}\sqrt{x_{3}})^{2}+(\frac{\lambda_{1}-\lambda_{2}-5+k}{2})^{2})]\phi'_{1}=0,\\
[\vartheta_{x_{2}}^{2}-((2\pi x_{2}^{-1}\sqrt{x_{3}})^{2}+(\frac{\lambda_{1}-\lambda_{2}-3+k}{2})^{2})]\phi'_{2}=0.
\end{gather*}
For some fixed $x_{3}$, if we put $z=(2\pi x_{2}^{-1}\sqrt{x_{3}})$, these equations are nothing but modified Bessel equations
\[
[\frac{d^{2}}{dz^{2}}+\frac{1}{z}\frac{d}{dz}-(1+\frac{\nu_{i}^{2}}{z^{2}})]\phi'_{i}=0,
\]
where $\nu_{1}=\frac{\lambda_{1}-\lambda_{2}-5+k}{2}$ and $\nu_{2}=\frac{\lambda_{1}-\lambda_{2}-3+k}{2}$.
Hence the intersection of the solution space of the equations (\ref{111}) and (\ref{112}) consist of functions written as follows,
\begin{multline*}
C_{1}\zeta_{1}(x_{3},x_{4})f_{1}(x_{2},x_{3})x_{1}^{\lambda_{1}-(4-k)}x_{2}^{\frac{\lambda_{1}+\lambda_{2}-3+k}{2}}\\
+C_{2}\zeta_{2}(x_{3},x_{4})f_{2}(x_{2},x_{3})x_{1}^{\lambda_{2}}x_{2}^{\frac{\lambda_{1}+\lambda_{2}-3+k}{2}},
\end{multline*}
where $C_{i}$ are constants, $\zeta_{i}$ are arbitrary functions and $f_{i}\in \mathfrak{MB}(\nu_{i}\,;\,2\pi x_{2}^{-1}\sqrt{x_{3}})$ for $i=1,2$.

By the same argument, we can also see that the intersection of the solution spaces of the equations (\ref{113}) and (\ref{114}) are written as follows,
\begin{multline*}
D_{1}\xi_{1}(x_{1},x_{2})g_{1}(x_{2},x_{3})x_{3}^{\frac{\lambda_{1}+\lambda_{2}-1+k}{4}}x_{4}^{\frac{\lambda_{1}}{2}}\\+D_{2}\xi_{2}(x_{1},x_{2})g_{2}(x_{2},x_{3})x_{3}^{\frac{\lambda_{1}+\lambda_{2}-1+k}{4}}x_{4}^{\frac{\lambda_{2}+k}{2}},
\end{multline*}
where $D_{i}$ are constants, $\xi_{i}$ are arbitrary functions and $g_{i}\in \mathfrak{MB}(\mu_{i}\,;\,2\pi x_{2}^{-1}\sqrt{x_{3}})$ for $i=1,2$. Here we put $\mu_{1}=\frac{\lambda_{1}-\lambda_{2}+1-k}{2}$ and $\mu_{2}=\frac{\lambda_{1}-\lambda_{2}-1-k}{2}$.

Note that $\mathfrak{MB}(\nu\,;\,x)\cap \mathfrak{MB}(\mu\,;\,x)=\{0\}$ if $\nu\neq \mu$. Then intersections of the solution spaces of the equations (\ref{111}), (\ref{112}), (\ref{113}) and (\ref{114}) are written as follows. For $k=1$, 
\[
x_{1}^{\lambda_{2}}x_{2}^{\frac{\lambda_{1}+\lambda_{2}-2}{2}}x_{3}^{\frac{\lambda_{1}+\lambda_{2}}{4}}x_{4}^{\frac{\lambda_{2}+1}{2}}f(2\pi x_{2}^{-1}\sqrt{x_{3}})
\]
where $f(x)\in\mathfrak{MB}(\frac{\lambda_{1}-\lambda_{2}-2}{2}\,;\,x)$.
And for $k=2$, 
\begin{multline*}
Cx_{1}^{\lambda_{2}}x_{2}^{\frac{\lambda_{1}+\lambda_{2}-1}{2}}x_{3}^{\frac{\lambda_{1}+\lambda_{2}+1}{4}}x_{4}^{\frac{\lambda_{1}}{2}}f(2\pi x_{2}^{-1}\sqrt{x_{3}})\\
+C'x_{1}^{\lambda_{1}-2}x_{2}^{\frac{\lambda_{1}+\lambda_{2}-1}{2}}x_{3}^{\frac{\lambda_{1}+\lambda_{2}+1}{4}}x_{4}^{\lambda_{2}+2}g(2\pi x_{2}^{-1}\sqrt{x_{3}}),
\end{multline*}
where $C,C'\in\C$, $f(x)\in \mathfrak{MB}(\frac{\lambda_{1}-\lambda_{2}-1}{2}\,;\,x)$ and $g(x)\in \mathfrak{MB}(\frac{\lambda_{1}-\lambda_{2}-3}{2}\,;\,x)$.
\end{proof}

\subsubsection{The embeddings into the space (II)}
We consider the embeddings of $X_{k,\lambda}$ ($k=1,2$) into the spaces (II$_{i}$) $(i=1,\ldots,4)$ which are classified in Proposition \ref{classify}. These spaces are isomorphic to $C^{\infty}_{\chi((0,0,\varepsilon_{1},\varepsilon_{2},0,\varepsilon_{3}))}(S_{\text{(II)}}\backslash G/K\,;\,I_{k}(\lambda))$. We consider the image of this space by the map $\Xi_{(\text{II})}$ defined in Section \ref{diff}. Here  
\[
(\varepsilon_{1},\varepsilon_{2},\varepsilon_{3})=
\begin{cases}
(1,1,0)&\text{if the case (II$_{1}$)},\\
(0,1,1)&\text{if the case (II$_{2}$)},\\
(0,1,0)&\text{if the case (II$_{3}$)},\\
(0,0,1)&\text{if the case (II$_{4}$)}.
\end{cases}
\]
We denote these space by 
\[
C^{\infty}_{\chi_{l(0,0,\varepsilon_{1},\varepsilon_{2},0,\varepsilon_{3})}}(S_{\text{(II)}}\backslash N\times A;\,I_{k}(\lambda))=\Xi_{(II)}(C^{\infty}_{\chi_{l(0,0,\varepsilon_{1},\varepsilon_{2},0,\varepsilon_{3})}}(S_{\text{(II)}}\backslash G/K;\,I_{k}(\lambda)).
\]
\begin{prop}\label{diffopII}
The function space $C^{\infty}_{\chi_{l(0,0,\varepsilon_{1},\varepsilon_{2},0,\varepsilon_{3})}}(S_{\text{(II)}}\backslash N\times A;\,I_{k}(\lambda))$ for $k=1,2$ are equal to the solution spaces of the following systems of the differential equations in $C^{\infty}(S_{\text{(II)}}\backslash N\times A)$.
\begin{gather}
\begin{split}
[\vartheta_{a_{1}}^{2}-(\lambda_{1}+\lambda_{2}-3+k)\vartheta_{a_{1}}&+(2\pi\sqrt{-1})^{2}(\frac{a_{2}}{a_{1}})^{2}\varepsilon_{2}\\
&+\lambda_{1}(\lambda_{2}+k)-(\vartheta_{a_{2}}+\vartheta_{a_{3}}+\vartheta_{a_{4}})]\phi=0,
\end{split}\label{2-1-1}\\
\varepsilon_{2}(\vartheta_{a_{1}}+\vartheta_{a_{2}}-(\lambda_{1}+\lambda_{2}-3+k))\phi=0,\label{2-1-2}\\
\frac{\partial}{\partial u}\phi=0,\label{2-1-3}\\
\varepsilon_{1}\varepsilon_{2}\phi=0,\label{2-1-4}\\
\begin{split}
[\vartheta_{a_{2}}^{2}-(\lambda_{1}+\lambda_{2}-2+k)\vartheta_{a_{2}}+(2\pi\sqrt{-1})^{2}(\frac{a_{2}}{a_{1}})^{2}\varepsilon_{2}+(\frac{a_{3}}{a_{2}})^{2}\frac{\partial^{2}}{\partial u^{2}}\\+(2\pi\sqrt{-1})^{2}(\frac{a_{4}}{a_{2}})^{2}\varepsilon_{1}+\lambda_{1}(\lambda_{2}+k)-(\vartheta_{a_{3}}+\vartheta_{a_{4}})]\phi=0,
\end{split}\label{2-2-2}\\
[\frac{a_{3}}{a_{2}}\frac{\partial}{\partial u}(\vartheta_{a_{2}}+\vartheta_{a_{3}}-(\lambda_{1}+\lambda_{2}-2+k))+(2\pi\sqrt{-1})^{2}\frac{a_{4}^{2}}{a_{2}a_{3}}\varepsilon_{1}(\varepsilon_{3}-\varepsilon_{1}u)]\phi=0,\label{2-2-3}\\
[\varepsilon_{1}(\vartheta_{a_{2}}+\vartheta_{a_{4}}-(\lambda_{1}+\lambda_{2}-1+k))+(\vartheta_{3}-\vartheta_{1}u)\frac{\partial}{\partial u}]\phi=0,\label{2-2-4}\\
\begin{split}
[\vartheta_{a_{3}}^{2}-(\lambda_{1}&+\lambda_{2}-1+k)\vartheta_{a_{3}}+(\frac{a_{3}}{a_{2}})^{2}\frac{\partial^{2}}{\partial u^{2}}\\&+(2\pi\sqrt{-1})^{2}(\frac{a_{4}}{a_{3}})^{2}(\varepsilon_{3}-\varepsilon_{1}u)^{2}-\vartheta_{a_{4}}+\lambda_{1}(\lambda_{2}+k)]\phi=0,
\end{split}\label{2-3-3}\\
[(\varepsilon_{3}-\varepsilon_{1}u)(\vartheta_{a_{3}}+\vartheta_{a_{4}}-(\lambda_{1}+\lambda_{2}-1+k))+\varepsilon_{1}\frac{\partial}{\partial u}]\phi=0,\label{2-3-4}\\
\begin{split}
[\vartheta_{a_{4}}^{2}-(\lambda_{1}+&\lambda_{2}+k)\vartheta_{a_{4}}+(2\pi\sqrt{-1})^{2}(\frac{a_{4}}{a_{2}})^{2}\varepsilon_{1}\\&+(2\pi\sqrt{-1})^{2}(\frac{a_{4}}{a_{3}})^{2}(\varepsilon_{3}-\varepsilon_{1}u)^{2}+\lambda_{1}(\lambda_{2}+k)]\phi=0,
\end{split}\label{2-4-4}\\
[\vartheta_{a_{1}}+\vartheta_{a_{2}}+\vartheta_{a_{3}}+\vartheta_{a_{4}}-k\lambda_{1}-(4-k)\lambda_{2}]\phi=0.\label{2-central}
\end{gather}
Here $\phi\in C^{\infty}(S_{\text{(II)}}\backslash N\times A)$.
\end{prop}
\begin{proof}
As well as Proposition \ref{diff-op-2-1}, these are obtained by the direct computation from Lemma \ref{symmetric}, Proposition \ref{modk} and Proposition \ref{n-difII}.
\end{proof}

(i) \textit{The case $(\varepsilon_{1},\varepsilon_{2},\varepsilon_{3})=(1,1,0)$.}

\begin{thm}\label{thmII-1}
When $(\varepsilon_{1},\varepsilon_{2},\varepsilon_{3})=(1,1,0)$, we have
\[
C^{\infty}_{\chi_{l(0,0,1,1,0,0)}}(S_{\text{(II)}}\backslash N\times A;\,I_{k}(\lambda))=\{0\},
\]
for $k=1,2$.
\end{thm}
\begin{proof}
It is immediate from the equation (\ref{2-1-4}).
\end{proof}

(ii) \textit{The case $(\varepsilon_{1},\varepsilon_{2},\varepsilon_{3})=(0,1,1)$.}
 
\begin{thm}\label{thmII-2}
When $(\varepsilon_{1},\varepsilon_{2},\varepsilon_{3})=(0,1,1)$, the space $C^{\infty}_{\chi_{l(0,0,0,1,0,1)}}(S_{\text{(II)}}\backslash N\times A;\,I_{k}(\lambda))$ $(k=1,2)$ are written as follows.

(i) If $k=1$, we have
\[
C^{\infty}_{\chi_{l(0,0,0,1,0,1)}}(S_{\text{(II)}}\backslash N\times A;\,I_{1}(\lambda))=\{0\}.
\]

(ii) If $k=2$, the space $C^{\infty}_{\chi_{l(0,0,0,1,0,1)}}(S_{\text{(II)}}\backslash N\times A;\,I_{2}(\lambda))$ consists of 
\[
x_{1}^{\frac{\lambda_{1}+\lambda_{1}-1}{2}}x_{2}^{\frac{1}{2}}x_{3}^{\frac{\lambda_{1}+\lambda_{2}+1}{2}}x_{4}^{\frac{1}{2}}f(2\pi x_{2})g(2\pi x_{3})
\]
for $f(x),g(x)\in \mathcal{MB}(\frac{\lambda_{1}-\lambda_{2}-2}{2};x)$. Here we put $$x_{1}=a_{1}a_{2},x_{2}=a_{1}^{-1}a_{2},x_{3}=a_{3}a_{4},x_{4}=a_{3}^{-1}a_{4}.$$

Thus we have $\dim_{\C}C^{\infty}_{\chi_{l(0,0,0,1,0,1)}}(S_{\text{(II)}}\backslash N\times A;\,I_{2}(\lambda))=4$. There exists a slowly increasing function on the domain $\{(x_{1},\ldots,x_{4},u)\mid x_{i}\in\R_{>0},u\in\R\}$ in $C^{\infty}_{\chi_{l(0,0,0,1,0,1)}}(S_{\text{(II)}}\backslash N\times A;\,I_{2}(\lambda))$ and it is unique up to constant.
\end{thm}
\begin{proof}
We show the case $k=1$ first. By the equations (\ref{2-1-2}) and (\ref{2-3-4}), we have the new one,
\[
[\vartheta_{a_{1}}+\vartheta_{a_{2}}+\vartheta_{a_{3}}+\vartheta_{a_{4}}-2\lambda_{1}-2\lambda_{2}+2]\phi=0.
\]
Comparing this equation with the equation (\ref{2-central}), we can conclude that the space $C^{\infty}_{\chi_{l(0,0,0,1,0,1)}}(S_{\text{(II)}}\backslash N\times A;\,I_{1}(\lambda))$ is equal to $\{0\}.$
Next, we consider the case $k=2$. By the equation (\ref{2-1-3}), we can eliminate the terms $\frac{\partial}{\partial u}$ from the other differential equations. Then the equations in Proposition \ref{diffopII} are reduced to the followings,
\begin{gather*}
[\vartheta_{a_{1}}+\vartheta_{a_{2}}-(\lambda_{1}+\lambda_{2}-1)]\phi=0,\\
\begin{split}
[\vartheta_{a_{1}}^{2}-(\lambda_{1}+&\lambda_{2}-1)\vartheta_{a_{1}}+\lambda_{1}(\lambda_{2}+2)\\&+(2\pi\sqrt{-1})a_{1}^{-2}a_{2}^{2}-(\vartheta_{a_{2}}+\vartheta_{a_{3}}+\vartheta_{a_{4}})]\phi=0,
\end{split}\\
[\vartheta_{a_{2}}^{2}-(\lambda_{1}+\lambda_{2})\vartheta_{a_{2}}+(2\pi\sqrt{-1})^{2}+\lambda_{1}(\lambda_{2}+2)-(\vartheta_{a_{3}}+\vartheta_{a_{4}})]\phi=0,\\
[\vartheta_{a_{3}}+\vartheta_{a_{4}}-(\lambda_{1}+\lambda_{2}+1)]\phi=0,\\
[\vartheta_{a_{3}}^{2}-(\lambda_{1}+\lambda_{2}+1)\vartheta_{a_{3}}+(2\pi\sqrt{-1})^{2}a_{3}^{-2}a_{4}^{2}-\vartheta_{a_{4}}+\lambda_{1}(\lambda_{2}+2)]\phi=0,\\
[\vartheta_{a_{4}}^{2}-(\lambda_{1}+\lambda_{2}+2)\vartheta_{a_{4}}+(2\pi\sqrt{-1})^{2}a_{3}^{-2}a_{4}^{2}+\lambda_{1}(\lambda_{2}+2)]\phi=0.
\end{gather*}
We put 
\begin{align*}
x_{1}&=a_{1}a_{2}, &x_{2}&=a_{1}^{-1}a_{2},\\
x_{3}&=a_{3}a_{4}, &x_{4}&=a_{3}^{-1}a_{4},\\
\end{align*}
Then we can rewrite above differential equations as follows,
\begin{gather*}
[2\vartheta_{x_{1}}-(\lambda_{1}+\lambda_{2}-1)]\phi=0,\\
[2\vartheta_{x_{3}}-(\lambda_{1}+\lambda_{2}+1)]\phi=0,\\
[\vartheta_{x_{2}}^{2}-\vartheta_{x_{2}}+(2\pi\sqrt{-1}x_{2})^{2}-(\frac{\lambda_{1}-\lambda_{2}-2}{2})^{2}-\frac{1}{4}]\phi=0,\\
[\vartheta_{x_{4}}^{2}-\vartheta_{x_{4}}+(2\pi\sqrt{-1}x_{4})^{2}-(\frac{\lambda_{1}-\lambda_{2}-2}{2})^{2}-\frac{1}{4}]\phi=0.
\end{gather*}
We take $\phi'$ as $\phi=x_{2}^{\frac{1}{2}}x_{4}^{\frac{1}{2}}\phi'$. Then we can see that $\phi'$ satisfies following equations,
\begin{gather*}
[2\vartheta_{x_{1}}-(\lambda_{1}+\lambda_{2}-1)]\phi'=0,\\
[2\vartheta_{x_{3}}-(\lambda_{1}+\lambda_{2}+1)]\phi'=0,\\
[\vartheta_{x_{2}}^{2}-((2\pi\sqrt{-1}x_{2})^{2}+(\frac{\lambda_{1}-\lambda_{2}-2}{2})^{2})]\phi'=0,\\
[\vartheta_{x_{4}}^{2}-((2\pi\sqrt{-1}x_{4})^{2}+(\frac{\lambda_{1}-\lambda_{2}-2}{2})^{2})]\phi=0.
\end{gather*}
Then we can conclude that
\[
\phi(x_{1},x_{2},x_{3},x_{4})=x_{1}^{\frac{\lambda_{1}+\lambda_{1}-1}{2}}x_{2}^{\frac{1}{2}}x_{3}^{\frac{\lambda_{1}+\lambda_{2}+1}{2}}x_{4}^{\frac{1}{2}}f(2\pi x_{2})g(2\pi x_{3}),
\]
where $f,g\in\mathfrak{MB}(\frac{\lambda_{1}-\lambda_{2}-2}{2};x)$. Hence we have the proposition.
\end{proof}

(iii) \textit{The case $(\varepsilon_{1},\varepsilon_{2},\varepsilon_{3})=(0,1,0)$.}
\begin{thm}\label{thmII-3}
When $(\varepsilon_{1},\varepsilon_{2},\varepsilon_{3})=(0,1,0)$, the space $C^{\infty}_{\chi_{l(0,0,0,1,0,0)}}(S_{\text{(II)}}\backslash N\times A;\,I_{k}(\lambda))$ $(k=1,2)$ are written as follows.

(i) If $k=1$, the space $C^{\infty}_{\chi_{l(0,0,0,1,0,0)}}(S_{\text{(II)}}\backslash N\times A;\,I_{1}(\lambda))$ consists of 
\[
x_{1}^{\frac{\lambda_{1}+\lambda_{2}-2}{2}}x_{2}^{\frac{1}{2}}(x_{3}x_{4})^{\lambda_{2}+1}f(2\pi x_{2})
\]
for $f(x)\in \mathfrak{MB}(\frac{\lambda_{1}-\lambda_{2}-3}{2};x)$. Here we put $$x_{1}=a_{1}a_{2},x_{2}=a_{1}^{-1}a_{2},x_{3}=a_{3},x_{4}=a_{4}.$$

Thus we have $\dim_{\C}C^{\infty}_{\chi_{l(0,0,0,1,0,0)}}(S_{\text{(II)}}\backslash N\times A;\,I_{1}(\lambda))=2$. There exists a slowly increasing function on the domain $\{(x_{1},\ldots,x_{4},u)\mid x_{i}\in\R_{>0},u\in\R\}$ in $C^{\infty}_{\chi_{l(0,0,0,1,0,0)}}(S_{\text{(II)}}\backslash N\times A;\,I_{2}(\lambda))$ and it is unique up to constant.

(ii) If $k=2$, the space $C^{\infty}_{\chi_{l(0,0,0,1,0,0)}}(S_{\text{(II)}}\backslash N\times A;\,I_{2}(\lambda))$ consists of 
\[
(C_{1}x_{3}^{\lambda_{2}}x_{4}^{\lambda_{1}}+C_{2}x_{3}^{\lambda_{1}-1}x_{4}^{\lambda_{2}+2})\times x_{1}^{\frac{\lambda_{1}+\lambda_{2}-1}{2}}x_{2}^{\frac{1}{2}}f(2\pi x_{2})
\]
for $f(x)\in \mathfrak{MB}(\frac{\lambda_{1}-\lambda_{2}-2}{2};x)$ and $C_{1},C_{2}\in\C$. Here $x_{i}$ $(i=1\ldots,4)$ are same as (i).

Thus we have $\dim_{\C}C^{\infty}_{\chi_{l(0,0,0,1,0,0)}}(S_{\text{(II)}}\backslash N\times A;\,I_{2}(\lambda))=4$. There exists a 2-dimentional supspace of $C^{\infty}_{\chi_{l(0,0,0,1,0,0)}}(S_{\text{(II)}}\backslash N\times A;\,I_{2}(\lambda))$ which consits of slowly increasing functions on $\{(x_{1},\ldots,x_{4},u)\mid x_{i}\in\R_{>0},u\in\R\}$.
\end{thm}
\begin{proof}
By the equation (\ref{2-1-3}), we can eliminate the term $\frac{\partial}{\partial u}$ from the other equations. Then the equations in Proposition \ref{diffopII} are reduced to the followings,
\begin{gather*}
[\vartheta_{a_{1}}^{2}-(\lambda_{1}+\lambda_{2}-4+k)\vartheta_{a_{1}}+(2\pi\sqrt{-1})^{2}a_{1}^{-2}a_{2}^{2}+\lambda_{2}(\lambda_{1}-(4-k))]\phi=0,\\
[\vartheta_{a_{1}}+\vartheta_{a_{2}}-(\lambda_{1}+\lambda_{2}-3+k)]\phi=0,\\
[\vartheta_{a_{2}}^{2}-(\lambda_{1}+\lambda_{2}-3+k)\vartheta_{a_{2}}+(2\pi\sqrt{-1})^{2}a_{1}^{-2}a_{2}^{2}+\vartheta_{a_{1}}+\lambda_{2}(\lambda_{1}-(4-k))]\phi=0,\\
[\vartheta_{a_{3}}^{2}-(\lambda_{1}+\lambda_{2}-1+k)\vartheta_{a_{3}}-\vartheta_{a_{4}}+\lambda_{1}(\lambda_{2}+k)]\phi=0,\\
[(\vartheta_{a_{4}}-\lambda_{1})(\vartheta_{a_{4}}-(\lambda_{2}+k))]\phi=0.
\end{gather*}
We put
\begin{align*}
x_{1}&=a_{1}a_{2},&x_{2}&=a_{1}^{-1}a_{2}.
\end{align*}
We take $\phi'$ as $\phi=x_{2}^{\frac{1}{2}}\phi'$.
Then we have
\begin{gather}
[\vartheta_{x_{1}}-\frac{\lambda_{1}+\lambda_{2}-3+k}{2}]\phi'=0,\label{nn1}\\
[\vartheta_{x_{2}}^{2}-((2\pi x_{2})^{2}+(\frac{\lambda_{1}-\lambda_{2}-(4-k)}{2})^{2})]\phi'=0,\label{nn2}\\
[\vartheta_{a_{3}}^{2}-(\lambda_{1}+\lambda_{2}-1+k)\vartheta_{a_{3}}-\vartheta_{a_{4}}+\lambda_{1}(\lambda_{2}+k)]\phi'=0,\label{nn3}\\
(\vartheta_{a_{4}}-\lambda_{1})(\vartheta_{a_{4}}-(\lambda_{2}+k))\phi'=0.\label{nn4}
\end{gather} 
The solution of (\ref{nn1}) and (\ref{nn2}) is $$\phi'(x_1,x_2,a_3,a_4)=c(a_3,a_4)x_{1}^{\frac{\lambda_{1}+\lambda_{2}-3+k}{2}}f(2\pi x_{2}),$$ for an arbitrary function $c(a_{3},a_{4})$ and $f(x)\in\mathfrak{MB}(\frac{\lambda_{1}-\lambda_{2}-(4-k)}{2};x)$. We solve the equations (\ref{nn3}) and (\ref{nn4}) to  determine the function $c(a_{3},c_{4})$. Then we can see that
\[
c(a_{3},a_{4})=
\begin{cases}
(a_{3}a_{4})^{\lambda_{2}+1}&\text{for}\ $k=1$,\\
C_{1}a_{3}^{\lambda_{2}+1}a_{4}^{\lambda_{1}}+C_{2}a_{3}^{\lambda_{1}-1}a_{4}^{\lambda_{2}+2}&\text{for}\ k=2,
\end{cases}
\]
for some constants $C_{1},C_{2}\in\C$. This concludes the proposition.
\end{proof}
\subsubsection{The embeddings into the space (III)}
We consider the embeddings of $X_{k,\lambda}$ ($k=1,2$) into the spaces type (III$_{i}$) $(i=1,\ldots,8)$ which are classified in Proposition \ref{classify}. These spaces are isomorphic to $C^{\infty}_{\chi((0,0,0,\varepsilon_{1},\varepsilon_{2},\varepsilon_{3}))}(N\backslash G/K;\,I_{k}(\lambda))$. We consider the image of this space by the map $\Xi_{(\text{III})}$ defined in Section \ref{diff}. Here 
\[
(\varepsilon_{1},\varepsilon_{2},\varepsilon_{3})=
\begin{cases}
(1,1,1)&\text{if the case (III$_{1}$)},\\
(1,1,0)&\text{if the case (III$_{2}$)},\\
(1,0,1)&\text{if the case (III$_{3}$)},\\
(0,1,1)&\text{if the case (III$_{4}$)},\\
(1,0,0)&\text{if the case (III$_{5}$)},\\
(0,1,0)&\text{if the case (III$_{6}$)},\\
(0,0,1)&\text{if the case (III$_{7}$)},\\
(0,0,0)&\text{if the case (III$_{8}$)}.
\end{cases}
\]
We denote the image of $\Xi_{(III)}$ by
\[
C^{\infty}_{\chi_{l(0,0,0,\varepsilon_{1},\varepsilon_{2},\varepsilon_{3})}}(A;I_{k}(\lambda))=\Xi_{(III)}(C^{\infty}_{\chi((0,0,0,\varepsilon_{1},\varepsilon_{2},\varepsilon_{3}))}(N\backslash G/K\,;\,I_{k}(\lambda))).
\]
\begin{prop}
The space $C^{\infty}_{\chi_{l(0,0,0,\varepsilon_{1},\varepsilon_{2},\varepsilon_{3})}}(A;I_{k}(\lambda))$ is equal to the solution space of the following system of the differential equations on $C^{\infty}(A)$.

\begin{gather}
\begin{split}
[\vartheta_{a_{1}}-(\lambda_{1}+\lambda_{2}+k-3)\vartheta_{a_{1}}&+(2\pi\sqrt{-1}\frac{a_{2}}{a_{1}})^{2}\varepsilon_{1}\\
&-(\vartheta_{a_{2}}+\vartheta_{a_{3}}+\vartheta_{a_{4}})+\lambda_{1}(\lambda_{2}+k)]\phi=0,
\end{split}\\
\varepsilon_{1}2\pi\sqrt{-1}\frac{a_{2}}{a_{1}}(\vartheta_{a_{1}}+\vartheta_{a_{2}}-(\lambda_{1}+\lambda_{2}+k-3))\phi=0,\\
\varepsilon_{1}\varepsilon_{2}\phi=0,\label{ufufu}\\
\begin{split}
[\vartheta_{a_{2}}-(\lambda_{1}+\lambda_{2}+k-2)\vartheta_{a_{2}}+&(2\pi\sqrt{-1}\frac{a_{2}}{a_{1}})^{2}\varepsilon_{1}+(2\pi\sqrt{-1}\frac{a_{3}}{a_{2}})^{2}\varepsilon_{2}\\&
-(\vartheta_{a_{3}}+\vartheta_{a_{4}})+\lambda_{1}(\lambda_{2}+k)]\phi=0,
\end{split}\\
\varepsilon_{2}2\pi\sqrt{-1}\frac{a_{3}}{a_{2}}(\vartheta_{a_{2}}+\vartheta_{a_{3}}-(\lambda_{1}+\lambda_{2}+k-2))\phi=0,\\
\varepsilon_{2}\varepsilon_{3}\phi=0,\label{ahaha}\\
\begin{split}
[\vartheta_{a_{3}}^{2}-(\lambda_{1}+\lambda_{2}&+k-1)\vartheta_{a_{3}}+(2\pi\sqrt{-1}\frac{a_{3}}{a_{2}})^{2}\varepsilon_{2}\\
&+(2\pi\sqrt{-1}\frac{a_{4}}{a_{3}})^{2}\varepsilon_{3}-\vartheta_{a_{4}}+\lambda_{1}(\lambda_{2}+k)]\phi=0,
\end{split}\\
\varepsilon_{3}2\pi\sqrt{-1}\frac{a_{4}}{a_{3}}(\vartheta_{a_{3}}+\vartheta_{a_{4}}-(\lambda_{1}+\lambda_{2}+k-1))\phi=0,\\
[\vartheta_{a_{4}}^{2}-(\lambda_{1}+\lambda_{2}+k)\vartheta_{a_{4}}+(2\pi\sqrt{-1}\frac{a_{4}}{a_{3}})^{2}\varepsilon_{3}+\lambda_{1}(\lambda_{2}+k)]\phi=0,\\
[\vartheta_{a_{1}}+\vartheta_{a_{2}}+\vartheta_{a_{3}}+\vartheta_{a_{4}}-k\lambda_{1}-(4-k)\lambda_{2}]\phi=0.
\end{gather}
Here $\phi\in C^{\infty}(A)$.
\end{prop}
\begin{proof}
As well as Proposition \ref{diff-op-2-1} and Proposition \ref{diffopII}, this system of differential equations are obtained by the direct computation by using Lemma \ref{symmetric}, Proposition \ref{modk} and Proposition \ref{n-dimIII}.
\end{proof}
\begin{thm}\label{thmIII}
The space $C^{\infty}_{\chi_{l(0,0,0,\varepsilon_{1},\varepsilon_{2},\varepsilon_{3})}}(A;I_{k}(\lambda))$ are written as follows.
\begin{itemize}
\item[(i)] When 
\[(\varepsilon_{1},\varepsilon_{2},\varepsilon_{3})=
\begin{cases}
(1,1,1)\\
(1,1,0)\\
(0,1,1)
\end{cases},
\]
then we have
\[
C^{\infty}_{\chi_{l(0,0,0,\varepsilon_{1},\varepsilon_{2},\varepsilon_{3})}}(A;I_{k}(\lambda))=\{0\}.
\]
\item[(ii)]  When $(\varepsilon_{1},\varepsilon_{2},\varepsilon_{3})=(1,0,1)$, the space $C^{\infty}_{\chi_{l(0,0,0,1,0,1)}}(A;I_{k}(\lambda))$ consists of followings. If $k=1$, we have $C^{\infty}_{\chi_{l(0,0,0,1,0,1)}}(A;I_{1}(\lambda))=\{0\}$. If $k=2$,  the space $C^{\infty}_{\chi_{l(0,0,0,1,0,1)}}(A;I_{2}(\lambda))$ consists of
\[
x_{1}^{\frac{1}{2}}x_{2}^{\frac{\lambda_{1}+\lambda_{2}-1}{2}}x_{3}^{\frac{1}{2}}x_{4}^{\frac{\lambda_{1}+\lambda_{2}+1}{2}}f(2\pi x_{1})g(2\pi x_{3})
\]
for $f(x), g(x)\in\mathfrak{MB}(\frac{\lambda_{1}-\lambda_{2}-2}{2};x)$. Here we put $$x_{1}=\frac{a_{2}}{a_{1}},x_{2}=a_{1}a_{2},x_{3}=\frac{a_{4}}{a_{3}},x_{4}=a_{3}a_{4}.$$

Thus we have $\dim_{\C}C^{\infty}_{\chi_{l(0,0,0,1,0,1)}}(A;I_{2}(\lambda))=4$. There is a 1-dimensional subspace of $C^{\infty}_{\chi_{l(0,0,0,1,0,1)}}(A;I_{2}(\lambda))$ which consists  slowly increasing functions on $\{(x_{1},\ldots,x_{4})\mid x_{i}\in\R_{>0},i=1,\ldots,4\}$.

\item[(iii)] When $(\varepsilon_{1},\varepsilon_{2},\varepsilon_{3})=(1,0,0)$, the space $C^{\infty}_{\chi_{l(0,0,0,1,0,0)}}(A;I_{k}(\lambda))$ consists of 
\[
x_{1}^{\frac{1}{2}}x_{2}^{\frac{\lambda_{1}+\lambda_{2}+k-3}{2}}f(2\pi x_{1})\times
\begin{cases}
x_{3}^{\lambda_{2}+1}x_{4}^{\lambda_{2}+1}&\text{if}\ k=1,\\
C_{1}x_{3}^{\lambda_{2}+1}x_{4}^{\lambda_{1}}+C_{2}x_{3}^{\lambda_{1}-1}x_{4}^{\lambda_{2}+2}&\text{if}\ k=2,
\end{cases}
\]
for $f(x)\in\mathfrak{MB}(\frac{\lambda_{1}-\lambda_{2}+k-4}{2};x)$ and $C_{1},C_{2}\in \C$. Here we put $$x_{1}=\frac{a_{2}}{a_{1}},x_{2}=a_{1}a_{2},x_{3}=a_{3},x_{4}=a_{4}.$$

Thus if $k=1$, we have $\dim_{\C}C^{\infty}_{\chi_{l(0,0,0,1,0,0)}}(A;I_{1}(\lambda))=2$. There is a 1-dimensional subspace of $C^{\infty}_{\chi_{l(0,0,0,1,0,0)}}(A;I_{1}(\lambda))$ which consists  slowly increasing functions on $\{(x_{1},\ldots,x_{4})\mid x_{i}\in\R_{>0},i=1,\ldots,4\}$. 

Also if $k=2$, we have $\dim_{\C}C^{\infty}_{\chi_{l(0,0,0,1,0,0)}}(A;I_{2}(\lambda))=4$. There is a 2-dimensional subspace of $C^{\infty}_{\chi_{l(0,0,0,1,0,0)}}(A;I_{2}(\lambda))$ which consists  slowly increasing functions on $\{(x_{1},\ldots,x_{4})\mid x_{i}\in\R_{>0},i=1,\ldots,4\}$.

\item[(iv)] When $(\varepsilon_{1},\varepsilon_{2},\varepsilon_{3})=(0,1,0)$, the space $C^{\infty}_{\chi_{l(0,0,0,0,1,0)}}(A;I_{k}(\lambda))$ consists of followings. We put $$x_{1}=a_{1},x_{2}=\frac{a_{3}}{a_{2}},x_{3}=a_{2}a_{3},x_{4}=a_{4}.$$ If $k=1$, 
\[
x_{1}^{\lambda_{2}}x_{2}^{\frac{1}{2}}x_{3}^{\frac{\lambda_{1}+\lambda_{2}+k-2}{2}}x_{4}^{\lambda_{2}+1}f(2\pi x_{2})
\]
for $f(x)\in\mathfrak{MB}(\frac{\lambda_{1}-\lambda_{2}-2}{2};x)$. Thus we have $\dim_{\C}C^{\infty}_{\chi_{l(0,0,0,0,1,0)}}(A;I_{1}(\lambda))=2$. There is a 1-dimensional subspace of $C^{\infty}_{\chi_{l(0,0,0,0,1,0)}}(A;I_{k}(\lambda))$ which consists  slowly increasing functions on $\{(x_{1},\ldots,x_{4})\mid x_{i}\in\R_{>0},i=1,\ldots,4\}$.

 Also if $k=2$, 
\[
Cx_{1}^{\lambda_{1}-2}x_{2}^{\frac{1}{2}}x_{3}^{\frac{\lambda_{1}+\lambda_{2}+k-2}{2}}x_{4}^{\lambda_{2}+2}g_{1}(2\pi x_{2})+C'x_{1}^{\lambda_{2}}x_{2}^{\frac{1}{2}}x_{3}^{\frac{\lambda_{1}+\lambda_{2}+k-2}{2}}x_{4}^{\lambda_{1}}g_{2}(2\pi x_{2})
\]
for $C,C'\in\C$, $g_{1}\in\mathfrak{MB}(\frac{\lambda_{1}-\lambda_{2}-3}{2};x)$ and $g_{2}\in\mathfrak{MB}(\frac{\lambda_{1}-\lambda_{2}}{2};x)$. Thus we have $\dim_{\C}C^{\infty}_{\chi_{l(0,0,0,0,1,0)}}(A;I_{2}(\lambda))=4$. There is a 2-dimensional subspace of $C^{\infty}_{\chi_{l(0,0,0,1,0,0)}}(A;I_{k}(\lambda))$ which consists  slowly increasing functions on $\{(x_{1},\ldots,x_{4})\mid x_{i}\in\R_{>0},i=1,\ldots,4\}$.

\item[(v)] When $(\vartheta_{1},\vartheta_{2},\vartheta_{3})=(0,0,1)$, the space $C^{\infty}_{\chi_{l(0,0,0,0,0,1)}}(A;I_{k}(\lambda))$ consists of
\[
x_{3}^{\frac{1}{2}}x_{4}^{\frac{\lambda_{1}+\lambda_{2}+k-1}{2}}g(2\pi x_{3})
\times
\begin{cases}
x_{1}^{\lambda_{2}}x_{2}^{\lambda_{2}}&\text{if}\ k=1,\\
C_{1}x_{1}^{\lambda_{2}}x_{2}^{\lambda_{1}-1}+C_{2}x_{1}^{\lambda_{1}^2}x_{2}^{\lambda_{2}+1	}&\text{if}\ k=2,
\end{cases}
\]
for $g(x)\in\mathfrak{MB}(\frac{\lambda_{1}-\lambda_{2}-k}{2};x)$ and $C_{1},C_{2}\in \C$. Here we put $$x_{1}=a_{1},x_{2}=a_{2},x_{3}=\frac{a_{4}}{a_{3}},x_{4}=a_{3}a_{4}.$$

Thus if $k=1$, we have $\dim_{\C}C^{\infty}_{\chi_{l(0,0,0,0,0,1)}}(A;I_{1}(\lambda))=2$. There is a 1-dimensional subspace of $C^{\infty}_{\chi_{l(0,0,0,0,0,1)}}(A;I_{1}(\lambda))$ which consists  slowly increasing functions on $\{(x_{1},\ldots,x_{4})\mid x_{i}\in\R_{>0},i=1,\ldots,4\}$.

Also if $k=2$, we have $\dim_{\C}C^{\infty}_{\chi_{l(0,0,0,0,0,1)}}(A;I_{2}(\lambda))=4$. There is a 2-dimensional subspace of $C^{\infty}_{\chi_{l(0,0,0,0,0,1)}}(A;I_{2}(\lambda))$ which consists  slowly increasing functions on $\{(x_{1},\ldots,x_{4})\mid x_{i}\in\R_{>0},i=1,\ldots,4\}$.
\item[(vi)] When $(\varepsilon_{1},\varepsilon_{2},\varepsilon_{3})=(0,0,0)$, the space $C^{\infty}_{\chi_{l(0,0,0,0,0,0)}}(A;I_{k}(\lambda))$ consists of the followings. If $k=1$, 
\begin{multline*}
C_{1}a_{1}^{\lambda_{2}}a_{2}^{\lambda_{2}}a_{3}^{\lambda_{2}}a_{4}^{\lambda_{1}}+C_{2}a_{1}^{\lambda_{2}}a_{2}^{\lambda_{1}-2}a_{3}^{\lambda_{2}+1}a_{4}^{\lambda_{2}+1}\\
+C_{3}a_{1}^{\lambda_{1}-3}a_{2}^{\lambda_{2}+1}a_{3}^{\lambda_{2}+1}a_{4}^{\lambda_{2}+1}+C_{4}a_{1}^{\lambda_{2}}a_{2}^{\lambda_{2}}a_{3}^{\lambda_{1}-1}a_{4}^{\lambda_{2}+1},
\end{multline*}
for $C_{i}\in\C$, $i=1,\ldots,4$.

Also if $k=2$,
\begin{multline*}
C_{1}a_{1}^{\lambda_{2}}a_{2}^{\lambda_{2}}a_{3}^{\lambda_{1}}a_{4}^{\lambda_{1}}+C_{2}a_{1}^{\lambda_{1}-2}a_{2}^{\lambda_{1}-1}a_{3}^{\lambda_{2}+1}a_{4}^{\lambda_{1}}\\
+C_{3}a_{1}^{\lambda_{1}-2}a_{2}^{\lambda_{2}+1}a_{3}^{\lambda_{2}+1}a_{4}^{\lambda_{1}}+C_{4}a_{1}^{\lambda_{2}}a_{2}^{\lambda_{1}-1}a_{3}^{\lambda_{1}-1}a_{4}^{\lambda_{2}+2}\\
+C_{5}a_{1}^{\lambda_{1}-2}a_{2}^{\lambda_{2}+1}a_{3}^{\lambda_{1}-1}a_{4}^{\lambda_{2}+2}+C_{6}a_{1}^{\lambda_{1}-2}a_{2}^{\lambda_{1}-2}a_{3}^{\lambda_{2}+2}a_{4}^{\lambda_{2}+2}
\end{multline*}
for $C_{i}\in\C$, $i=1,\ldots,6$.

\end{itemize}
\end{thm}
\begin{proof}
(i) It is immediately follows from equations (\ref{ufufu}) and (\ref{ahaha}).

(ii) If we put $x_{1}=\frac{a_{2}}{a_{1}},x_{2}=a_{1}a_{2},x_{3}=\frac{a_{4}}{a_{3}},x_{4}=a_{3}a_{4}$, differential equations are written as
\begin{gather*}
\begin{split}
[\vartheta_{x_{1}}^{2}-\vartheta_{x_{1}}-&\frac{1}{4}(\lambda_{1}+\lambda_{2}+k-3)(\lambda_{1}+\lambda_{2}+k-1)\\
&+(2\pi\sqrt{-1}x_{1})^{2}+(\lambda_{1}+k-3)(\lambda_{2}+1)]\phi=0,
\end{split}\\
[2\vartheta_{x_{2}}-(\lambda_{1}+\lambda_{2}+k-3)]\phi=0,\\
[\vartheta_{x_{3}}^{2}-\vartheta_{x_{3}}-\frac{1}{4}((\lambda_{1}+\lambda_{2}+k)^{2}+1)+(2\pi\sqrt{-1}x_{3})^{2}+\lambda_{1}(\lambda_{2}+k)]\phi=0,\\
[2\vartheta_{x_{4}}-(\lambda_{1}+\lambda_{2}+k-1)]\phi=0, \\
[2\vartheta_{x_{2}}+2\vartheta_{x_{4}}-k\lambda_{1}-(4-k)\lambda_{2}]\phi=0.
\end{gather*}
Then we can solve these equations.

(iii) If we put $x_{1}=\frac{a_{2}}{a_{1}},x_{2}=a_{1}a_{2},x_{3}=a_{3},x_{4}=a_{4}$, differential equations are written as
\begin{gather*}
\begin{split}
[\vartheta_{x_{1}}^{2}-\vartheta_{x_{1}}-&\frac{1}{4}(\lambda_{1}+\lambda_{2}+k-3)(\lambda_{1}+\lambda_{2}+k-1)\\
&+(2\pi\sqrt{-1}x_{1})^{2}+(\lambda_{1}+k-3)(\lambda_{2}+1)]\phi=0,
\end{split}\\
[2\vartheta_{x_{2}}-(\lambda_{1}+\lambda_{2}+k-3)]\phi=0,\\
(\vartheta_{x_{3}}-(\lambda_{1}+k-3))(\vartheta_{x_{3}}-(\lambda_{2}+1))\phi=0,\\
(\vartheta_{x_{4}}-\lambda_{1})(\vartheta_{x_{4}}-(\lambda_{2}+k))\phi=0,\\
[\vartheta_{x_{3}}+\vartheta_{x_{4}}+(1-k)\lambda_{1}+(k-3)\lambda_{2}+(k-3)]\phi=0,\\
[2\vartheta_{x_{2}}+\vartheta_{x_{3}}+\vartheta_{x_{4}}-k\lambda_{1}-(4-k)\lambda_{2}]\phi=0.
\end{gather*}
Then we can solve these equations.

(iv) If we put $x_{1}=a_{1},x_{2}=\frac{a_{3}}{a_{2}},x_{3}=a_{2}a_{3},x_{4}=a_{4}$, differential equations are written as
\begin{gather*}
(\vartheta_{x_{1}}-(\lambda_{1}+(k-4)))(\vartheta_{x_{1}}-\lambda_{2})\phi=0,\\
\begin{split}
[\vartheta_{x_{2}}^{2}-\vartheta_{x_{2}}-\frac{1}{4}(\lambda_{1}&+\lambda_{2}+k-2)(\lambda_{1}+\lambda_{2}+k)\\
&+(2\pi\sqrt{-1}x_{2})^{2}-\vartheta_{x_{4}}+\lambda_{1}(\lambda_{2}+k)]\phi=0,
\end{split}\\
(2\vartheta_{x_{3}}-(\lambda_{1}+\lambda_{2}+k-2))\phi=0,\\
(\vartheta_{x_{4}}-\lambda_{1})(\vartheta_{x_{4}}-(\lambda_{2}+k))\phi=0,\\
[\vartheta_{x_{1}}+2\vartheta_{x_{3}}+\vartheta_{x_{4}}-k\lambda_{1}-(4-k)\lambda_{2}]\phi=0.
\end{gather*}
Then we can solve these equations.

(v) If we put $x_{1}=a_{1},x_{2}=a_{2},x_{3}=\frac{a_{4}}{a_{3}},x_{4}=a_{3}a_{4}$, differential equations are written as
\begin{gather*}
(\vartheta_{x_{1}}-(\lambda_{1}-4+k))(\vartheta_{x_{1}}-\lambda_{2})\phi=0,\\
(\vartheta_{x_{2}}-(\lambda_{1}-1))(\vartheta_{x_{2}}-(\lambda_{2}-1+k))\phi=0,\\
[\vartheta_{x_{3}}^{2}-\vartheta_{x_{3}}-\frac{1}{4}((\lambda_{1}+\lambda_{2}+k)^{2}+1)+(2\pi\sqrt{-1}x_{3})^{2}+\lambda_{1}(\lambda_{2}+k)]\phi=0,\\
[2\vartheta_{x_{4}}-(\lambda_{1}+\lambda_{2}+k-1)]\phi=0,\\
[\vartheta_{x_{1}}+\vartheta_{x_{2}}+2\vartheta_{x_{4}}-k\lambda_{1}-(4-k)\lambda_{2}]\phi=0.
\end{gather*}
Then we can solve these equations.

(vi) Differential equations are written as
\begin{gather*}
(\vartheta_{a_{1}}-(\lambda_{1}-(k-4)))(\vartheta_{a_{1}}-\lambda_{2})\phi=0,\\
[\vartheta_{a_{2}}^{2}-(\lambda_{1}+\lambda_{2}+k-2)\vartheta_{a_{2}}-(\vartheta_{a_{3}}+\vartheta_{a_{4}})+\lambda_{1}(\lambda_{2}+k)]\phi=0,\\
[\vartheta_{a_{3}}^{2}-(\lambda_{1}+\lambda_{2}+k-1)\vartheta_{a_{3}}-\vartheta_{a_{4}}+\lambda_{1}(\lambda_{2}+k)]\phi=0,\\
(\vartheta_{a_{4}}-\lambda_{1})(\vartheta_{a_{4}}-(\lambda_{2}+k))\phi=0,\\
[\vartheta_{a_{1}}+\vartheta_{a_{2}}+\vartheta_{a_{3}}+\vartheta_{a_{4}}-k\lambda_{1}-(4-k)\lambda_{2}]\phi=0.
\end{gather*}
Then we can solve these equations.

\end{proof}
\appendix

\section*{Appendix}
\section{The multiplicity one theorem for Horn's hypergeometric functions}
We will give a kind of the multiplicity one theorem for Horn's hypergeometric functions for the purpose of the application to the multiplicity theorem for the generalized Whittaker models. 

Let  $P_{i}(x)$ and $Q_{i}(x)$ are nonzero polynomials on $x=(x_{1},\ldots,x_{n})$ for $i=1,\ldots,n$. Then the Horn's hypergeometric functions are the solutions of the system of linear partial differential equations 
\begin{equation}\label{HHS}
[x_{i}P_{i}(\vartheta)-Q_{i}(\vartheta)]f(x)=0,\quad i=1,\ldots, n.
\end{equation}
Here $\vartheta_{i}=x_{i}\frac{\partial}{\partial x_{i}}$ and $\vartheta=(\vartheta_{1},\ldots,\vartheta_{n})$. We assume that $P_{i}$ and $Q_{i}$ can be decomposed by products of linear factors, i.e.,
\begin{align*}
&P_{i}(s)=\prod_{k=1}^{p}(\langle A_{k},s\rangle-c_{i}), &Q_{i}(s)=\prod_{l=1}^{q}(\langle B_{l},s\rangle-d_{l})
\end{align*}
for $s\in\R^{n}$, $A_{k},B_{l}\in \R^{n}$, $c_{k},d_{l}\in \C$ and $\langle\ ,\ \rangle$ denote the natural inner product in $\R^{n}$. We also assume $P_{i}(s)$, $Q_{i}(s+e_{i})$ are relatively prime for $i=1,\ldots,n$. Here $e_{i}=(0,\ldots,0,1,0,\ldots,0)$ ($1$ in the $i$th position).

We consider the following system of difference equations associated with this system of differential equations (\ref{HHS}),
\begin{equation}\label{difference equation}
P_{i}(-(s+e_{i}))\phi(s+e_{i})=Q_{i}(-s)\phi(s)\quad i=1,\ldots,n.
\end{equation}
We consider the general solutions for this system of difference equations (\ref{difference equation}).
\begin{rem}
Let $\phi$ be a solution of the system of difference equations (\ref{difference equation}). We consider an integral representation of Mellin-Barnes type, 
\[
f(x)=\displaystyle\int_{\mathcal{C}}\phi(s)x^{-s}\,ds.
\]

Then under the following assumptions, we can see that $f(x)$ is a solution of the system of differential equations (\ref{HHS}).

\textit{A}. For any $i=1,\ldots,n$, the translation of the contour $\mathcal{C}$ with respect to the basis $e_{i}$ is homologically equivalent to $\mathcal{C}$ in the complement of the set of the singularities of the integrand $\phi(s)$ in $\C^{n}$.

\textit{B}. The integral converges absolutely and it can be differentiated with  respect to $x$ sufficiently many times.
\end{rem}

If we put 
\[
R_{i}(s)=\frac{Q_{i}(-s)}{P_{i}(-(s+e_{i}))}\quad i=1,\ldots,n,
\]

\begin{thm}[Ore-Sato-Sadykov \cite{Ore},\cite{Sat},\cite{Sad}]\label{ore-sato}
The system of difference equations (\ref{difference equation}) is solvable if and only if 
\begin{equation}\label{solvable}
R_{i}(s+e_{j})R_{j}(s)=R_{j}(s+e_{i})R_{i}(s),\quad i,j=1,\ldots,n.
\end{equation}
If the system (\ref{difference equation}) is solvable, then its solution is unique up to an arbitrary periodic function $\psi(s)$ with respect to $e_{i}$, i.e., $\psi(s+e_{i})=\psi(s)$, $i=1,\ldots,n$. Furthermore, there exist $p',q'\in \N$, $A'_{k},B'_{l}\in \R^{n}\ (1\le k \le p', 1\le l\le q')$, $c'_{k},d'_{l}\in \C\ (1\le k \le p', 1\le l\le q')$ and $t_{i}\in \R$ $(i=1,\ldots,n)$ such that the general solution of $(\ref{difference equation})$ is written as follows,
\[
\phi(s)=t^{-s}\frac{\prod_{l=1}^{q'}\Gamma(\langle B'_{l},s\rangle-d'_{l})}{\prod_{k=1}^{p'}\Gamma(\langle A'_{k},s\rangle-c'_{i})}\psi(s),
\]
where $t^{-s}=t_{1}^{-s_{1}}\cdots t_{n}^{-s_{n}}$ and $\psi(s)$ is an arbitrary periodic function satisfying $\psi(s+e_{i})=\psi(s)$.
\end{thm}

Suppose that the system of difference equations (\ref{difference equation}) is solvable, i.e., the condition (\ref{solvable}) is satisfied and we can choose a solution, 
\[
\phi(s)=t^{-s}\frac{\prod_{l=1}^{q'}\Gamma(\langle B'_{l},s\rangle-d'_{l})}{\prod_{k=1}^{p'}\Gamma(\langle A'_{k},s\rangle-c'_{i})}
\]
which satisfies following conditions;

\textit{C}. we have the inequality,
\[
\displaystyle\sum_{l=1}^{q'}|\langle B'_{l},s\rangle|-\displaystyle\sum_{k=1}^{p'}|\langle A'_{k},s\rangle|\ge \displaystyle\sum_{i=1}^{n}|s_{i}|
\]
for $s\in\R^{n}$.

\textit{D}. the function $\phi(s)$  has no zero if each $\textrm{Re}(s_{i})$ are sufficiently large for $i=1,\ldots,n$.
\begin{rem}
We consider the integral
\[
f(x)=\displaystyle\int_{\sigma_{1}-\sqrt{-1}\infty}^{\sigma_{1}+\sqrt{-1}\infty}\cdots \int_{\sigma_{n}-\sqrt{-1}\infty}^{\sigma_{n}+\sqrt{-1}\infty} \phi(s)x^{-s}\,ds,
\]
for appropriate $\sigma_{i}\in \R\ i=1,\ldots,n$.
Under the assumption \textit{C}, it follows that the integral is absolutely convergent in the set $\{x\in\R^{n}\mid (t_{1}x_{1},\ldots,t_{n}x_{n})\in(\R_{\ge 0})^{n}\}$.
\end{rem}
The following theorem is a generalization of the theorem of Diaconu and Goldfeld (Theorem 6.1.6 in \cite{G})
\begin{thm}[Multiplicity one]\label{M1}
If $f(x)$ is a solution of the system of Horn's hypergeometric differential equations $\mathrm{(\ref{HHS})}$ which satisfies the growth condition
\[
\sup_{x\in (\R_{\ge 0})^{n}}|x^{\alpha}f(tx)| < +\infty
\]
for sufficiently large integers $\alpha_{i}\in \N,$ $i=1,\ldots,n$, then it is unique up to constant multiple. Here $x^{\alpha}=x_{1}^{\alpha_{1}}\cdots x_{n}^{\alpha_{n}}$ and $tx=(t_{1}x_{1},\ldots,t_{n}x_{n})$.
\end{thm}
\begin{proof}
We consider the Mellin transform of $f(tx)$ as the function of $x$,
\[
\mathcal{M}[f,s]=\displaystyle\int_{0}^{\infty}\cdots\int_{0}^{\infty}f(tx)x^{s-1}\,dx.
\]
This integral converges absolutely and $\mathcal{M}[f,s]$ is analytic function of $s$ if each $\mathrm{Re}(s_{i})$ are sufficiently large by the assumption of $f(x)$. Changing the variables $x$ to $tx=(t_{1}x_{1},\cdots,t_{n}x_{n})$, then we have
\[
\mathcal{M}[f,s]=t^{-s}\displaystyle\int_{0}^{t_{1}^{-1}\infty}\cdots\int_{0}^{t_{n}^{-1}\infty}f(x)x^{s-1}\,dx.
\]
By the growth condition of $f(x)$, we have
\begin{multline*}
\displaystyle\int_{0}^{t_{1}^{-1}\infty}\cdots\int_{0}^{t_{n}^{-1}\infty}\frac{\partial^{k}}{\partial x_{i}^{k}}f(x)x^{s-1}\,dx\\
=(-1)^{k}\displaystyle\int_{0}^{t_{1}^{-1}\infty}\cdots\int_{0}^{t_{n}^{-1}\infty}f(x)\frac{\partial^{k}}{\partial x_{i}^{k}}x^{s-1}\,dx, 
\end{multline*}
by integration by parts for $i=1,\ldots,n$. Recall that $f(x)$ satisfies the system of the partial differential equations (\ref{HHS}), then we have the system of the difference equations for $\mathcal{M}[f,s]$,
\[
P_{i}(-(s+e_{i}))\mathcal{M}[f,s+e_{i}]=Q_{i}(-s)\mathcal{M}[f,s]\quad i=1,\ldots,n.
\]
Hence by Theorem \ref{ore-sato}, there is a periodic function $\psi(s)$ and we have
\begin{equation}\label{analytic}
\frac{\prod_{l=1}^{q'}\Gamma(\langle B'_{l},s\rangle-d'_{l})}{\prod_{k=1}^{p'}\Gamma(\langle A'_{k},s\rangle-c'_{i})}\psi(s)=\displaystyle\int_{0}^{t_{1}^{-1}\infty}\cdots\int_{0}^{t_{n}^{-1}\infty}f(x)x^{s-1}\,dx.
\end{equation}
By Stirling's formula and the assumption \textit{C}, we obtain the estimate for $\mathrm{Re}(s_{i})>0$ $(i=1,\ldots,n)$,
\[
\frac{\prod_{l=1}^{q}\Gamma(\langle B_{l},s\rangle-d_{l})}{\prod_{k=1}^{p}\Gamma(\langle A_{k},s\rangle-c_{i})}=O\left(\exp(-\frac{1}{2}\pi\sum_{i=1}^{n}|\mathrm{Im}(s_{i})|)\right)\quad\text{as}\quad \sum_{i=1}^{n}|\mathrm{Im}(s_{i})|\rightarrow +\infty.
\]
Also by the Riemann-Lebesgue theorem, we have
\[
\mathcal{M}[f,s]\rightarrow 0\quad \text{as}\quad \sum_{i=1}^{n}|\mathrm{Im}(s_{i})|\rightarrow +\infty.
\]
Combining these estimates, we obtain the asymptotic behaviour of the periodic function 
\begin{equation}\label{estimate}
\psi(s)=O(\exp(\frac{1}{2}\pi|\mathrm{Im}(s_{i})|)),
\end{equation}
as $\mathrm{Im}(s_{i})\rightarrow\infty$ and the other $s_{j}$ $(i\neq j)$ are fixed. The right hand side of the equation (\ref{analytic}) is the analytic function of $s$ when $\mathrm{Re}(s_{i})\ (i=1,\ldots,n)$ are sufficiently large. Thus if we recall that the assumption \textit{D} and the periodicy of $\psi(s)$, we can see that $\psi(s)$ is an entire function. We put $z_{i}=\exp 2\pi\sqrt{-1}s_{i}$ for $i=1,\ldots,n$. And we consider the Laurant expansion of $\phi(s)$ with respect to $z_{1}$,
\[
\psi(s)=\sum_{k=-\infty}^{\infty}c_{k}^{(1)}(s_{2},\ldots,s_{n})z_{1}^{k}.
\]
Here $c_{k}^{(1)}(s_{2},\ldots,s_{n})$ are periodic and entire functions for $(s_{2},\ldots,s_{n})\in \C^{n-1}$.
We write $s_{i}=\sigma_{i}+\sqrt{-1}\tau_{i}$ for $\sigma_{i},\tau_{i}\in\R$, $i=1,\ldots,n.$ We consider an integration
\begin{multline*}
\int_{0}^{1}|\psi(s)|^{2}\,d\sigma_{i}=\sum_{k=-\infty}^{\infty}|c_{k}^{(1)}(s_{2},\ldots,s_{n})|^{2}\exp(-4\pi k\tau_{i})\\
\ge |c_{t}^{(1)}(s_{2},\ldots,s_{n})|^{2}\exp(-4\pi t \tau_{i})
\end{multline*}
for every $t=0,\pm 1,\pm 2,\ldots$. However the estimate (\ref{estimate}) tells us that there exist constants $M_{i}\in \R_{>0}$ and we have
\[
\exp(\pi|\tau_{i}|)>M_{i}\int_{0}^{1}|\psi(s)|^{2}\,d\sigma_{i}
\]
for sufficiently large $\tau_{i}$.
Thus we have $c_{t}^{(1)}(s_{2},\ldots,s_{n})=0$ for $t=\pm 1,\pm 2,\ldots$. The remaining coefficient $c_{0}^{(1)}(s_{2},\ldots,s_{n})$ is also the periodic and entire functions for $(s_{2},\ldots,s_{n})\in \C^{n-1}$. Hence we apply the same argument for $c_{0}^{(1)}(s_{2},\ldots,s_{n})$ with respect to $s_{2}$. And also we can succeed inductively for $i=3,\ldots,n$. Thus we can conclude $\psi(s)$ must be a constant. This completes the proof of the theorem.
\end{proof}

\section{Horn's hypergeometric function $\mathrm{H}_{10}$} 

We will give some facts about Horn's two variables hypergeometric function $\mathrm{H}_{10}$. Horn's hypergeometric function $\mathrm{H}_{10}$ is the hypergeometric seires defined as follows,
\[
\mathrm{H}_{10}(a,d\,;\,x,y)=\sum_{m=0,n=0}^{\infty}\frac{(a)_{2m-n}}{(d)_{m}m!n!}x^{m}y^{n}.
\]
Here the symbol $(a)_{m}$ means the Pochhammer symbol, i.e., $(a)_{m}=a(a+1)\cdots (a+(m-1))$ for $a\in \C$ and $m\in \N$. It is not hard to see that this power series satisfies the system of hypergeometric partial differential equations,
\begin{equation}\label{H10}
\begin{gathered}
\{x(2\vartheta_{x}-\vartheta_{y}+a)(2\vartheta_{x}-\vartheta_{y}+a+1)-\vartheta_{x}(\vartheta_{x}+d-1)\} \phi(x,y)=0,\\
\{y-\vartheta_{y}(2\vartheta_{x}-\vartheta_{y}+a)\}\phi(x,y)=0.
\end{gathered}
\end{equation}

It is known that the dimension of the solution space is $4$ (cf. \cite{A-K}). We define another convergent series
\[
\tilde{\mathbb{H}}_{10}(a,d\,;\,x,y)=\sum_{m=0,n=0}^{\infty}\frac{(-1)^{m+2n}}{(a+1)_{m+2n}(d)_{n}m!n!}x^{m}y^{n}.
\]
Then the basis of the solution space are written by the power series below,
\begin{gather*}
\mathrm{H}_{10}(a,d\,;\,x,y)\\
y^{-d+1}\mathrm{H}_{10}(a-2d+2,-d+2\,;\,x,y),\\
x^{a}\tilde{\mathbb{H}}_{10}(a,d\,;\,x,x^{2}y),\\
x^{a}y^{-d+1}\tilde{\mathbb{H}}_{10}(a-2d+3,-d+2\,;\,x,x^{2}y).
\end{gather*}

The system of hypergeometric differential equations $(\ref{H10})$ has the solution which has the Mellin-Barnes integral representation. This is written as follows,
\begin{multline*}
\phi(x,y)=\\
\int_{\sigma_{1}-\sqrt{-1}\infty}^{\sigma_{1}+\sqrt{-1}\infty}\int_{\sigma_{2}-\sqrt{-1}\infty}^{\sigma_{2}+\sqrt{-1}\infty}\Gamma(s_{1})\Gamma(s_{1}-2s_{2}-a)\Gamma(s_{2})\Gamma(s_{2}-d+1)(-x)^{-s_{1}}y^{-s_{2}}\,ds_{1}\,ds_{2}.
\end{multline*}
Here $\sigma_{1}\in\R$ and $\sigma_{2}\in\R$ satisfy the conditions, $\sigma_{1}>0$, $\sigma_{2}>\mathrm{max}\{0,\mathrm{Re}(d-1)\}$ and $\sigma_{1}-2\sigma_{2}>\mathrm{Re}(a)$.
This integral converges absolutely for $x\in\R_{\le 0}$ and $y\in\R_{\ge 0}$.
\begin{thm}\label{multhorn}
If $f(x,y)$ is a solution of the system (\ref{H10}) which satisfies that
\[
\sup_{x,y\in\R_{\ge 0}}|x^{\alpha_{1}}y^{\alpha_{2}}f(-x,y)|<+\infty
\]
for sufficiently large $\alpha_{1},\alpha_{2}\in \N$, 
then 
\begin{multline*}
f(x,y)=\\
C\int_{\sigma_{1}-\sqrt{-1}\infty}^{\sigma_{1}+\sqrt{-1}\infty}\int_{\sigma_{2}-\sqrt{-1}\infty}^{\sigma_{2}+\sqrt{-1}\infty}\Gamma(s_{1})\Gamma(s_{1}-2s_{2}-a)\Gamma(s_{2})\Gamma(s_{2}-d+1)(-x)^{-s_{1}}y^{-s_{2}}\,ds_{1}\,ds_{2},
\end{multline*}
for some constant $C$.
\end{thm}
\begin{proof} For
\begin{multline*}
\phi(x,y)=\\
\int_{\sigma_{1}-\sqrt{-1}\infty}^{\sigma_{1}+\sqrt{-1}\infty}\int_{\sigma_{2}-\sqrt{-1}\infty}^{\sigma_{2}+\sqrt{-1}\infty}\Gamma(s_{1})\Gamma(s_{1}-2s_{2}-a)\Gamma(s_{2})\Gamma(s_{2}-d+1)(-x)^{-s_{1}}y^{-s_{2}}\,ds_{1}\,ds_{2},
\end{multline*}
it is easy to see that $\phi$ satisfies the assumptions of Theorem \ref{M1}. Hence we only need to check that $\phi$ satisfies the growth condition. If we write a complex number $s=\sigma+\sqrt{-1}\tau$, we have
\[
|x|^{-s}=|x|^{-\sigma}.
\]
Thus we have the inequality,
\[|\phi(x,y)|\le M|x|^{-\sigma_{1}}|y|^{-\sigma_{2}},\]
for $x\in\R_{\le 0}$ and $y\in\R_{\ge 0}$.
Here the constant
\[
M=\left|\int_{\sigma_{1}-\sqrt{-1}\infty}^{\sigma_{1}+\sqrt{-1}\infty}\int_{\sigma_{2}-\sqrt{-1}\infty}^{\sigma_{2}+\sqrt{-1}\infty}\Gamma(s_{1})\Gamma(s_{1}-2s_{2}-a)\Gamma(s_{2})\Gamma(s_{2}-d+1)\,ds_{1}\,ds_{2}\right|. 
\]
We can choose $\sigma_{1}$ and $\sigma_{2}$ as $\sigma_{1}>0$, $\sigma_{2}>\mathrm{max}\{0,\mathrm{Re}(d-1)\}$ and $\sigma_{1}-2\sigma_{2}>\mathrm{Re}(a)$. Thus $\phi(x,y)$ satisfies the growth condition.
\end{proof}


\begin{thebibliography}{33}
\bibitem{A} Abe, N.: Generalized Jacquet modules of parabolic induction, preprint, arXiv:0710.3206.
\bibitem{A-K} Appell, P., Kamp\'e de F\'eriet, J.: Fonctions Hyperg\'eom\'etriques et Hypersph\'eriques, polyn\^ome d'Hermite. Gauthier-Villars, Paris, 1926.
\bibitem{C-G} Corwin, Lawrence J., Greenleaf, F.: Representations of nilpotent Lie groups and their applications. Part I. Basic theory and examples. Cambridge Studies in Advanced Mathematics, 18. Cambridge University Press, 1990.
\bibitem{C-M} Casselman, W., Mili\v ci\'c, D.: Asymptotic behaviour of matrix coefficients of admissible representations. Duke Math. J. 49 (1982), no. 4, 869--930.
\bibitem{G} Goldfeld, D.: Automorphic Forms and L-Functions for the Group $GL(n,R)$. Cambridge Studies in Advanced Mathematics (No. 99). Cambridge University Press, 2006.
\bibitem{Has} Hashizume, M.: Whittaker models for real reductive groups.
Japan. J. Math. (N.S.) 5 (1979), no. 2, 349--401. 
\bibitem{Has-sem} Hashizume, M.: Whittaker functions on semisimple Lie groups. Hiroshima Math. J. 12 (1982), 259--293.
\bibitem{Hel-1} Helgason, S.: A duality for symmetric spaces with applications to group representations. Adv. in Math. 5 (1976), 1--154. 
\bibitem{Hel-2} Helgason, S.: Groups and geometric analysis. Integral geometry, invariant differential operators, and spherical functions. Corrected reprint of the 1984 original. Mathematical Surveys and Monographs, 83. American Mathematical Society, Providence, RI, 2000.
\bibitem{H-I-O} Hirano, M., Ishii, T., Oda, T.: Whittaker functions for $P\sb J$-principal series representations of ${\rm Sp}(3,\bold R)$. Adv. Math. 215 (2007), no. 2, 734--765.
\bibitem{Horn} Horn, J.: \"Uber die Convergentz der hypergeometrichen Reihen zweier und dreier Ver\"anderlichen. Math. Ann. 34 (1889), 544--600.
\bibitem{I-O} Ishii, T., Oda, T.: Generalized Whittaker functions of the degenerate principal series representations of ${\rm SL}(3,R)$. Comment. Math. Univ. St. Pauli 54 (2005), no. 2, 187--209.
\bibitem{K} Kirillov, A.: Unitary representations of nilpotent Lie groups. Uspehi Mat. Nauk 17 1962 no. 4 (106), 57--110.
\bibitem{K-et.al} Kashiwara, M., Kowata, A., Minemura, K., Okamoto, K., Oshima, T., Tanaka, M.: Eigenfunctions of invariant differential operators on a symmetric space. Ann. of Math. (2) 107 (1978), no. 1, 1--39.
\bibitem{K-S} Kashiwara, M., Schmid, W.: Quasi-equivariant $\mathscr{D}$-modules, equivariant derived category, and representations of reductive Lie groups. Lie theory and geometry, 457--488, Progr. Math., 123, Birkh\"auser Boston, 1994. 
\bibitem{M} Matumoto, H.: Whittaker vectors and the Goodman-Wallach operators.
Acta Math. 161 (1988), no. 3-4, 183--241.
\bibitem{O-H} Hiroe, K., Oda, T.: Hecke-Siegel's pull-back formula for the Epstein zeta function with a harmonic polynomial. J. Number Theory 128 (2008), no. 4, 835--857.
\bibitem{Ore} Ore, O.: Sur la forme de fonctions hyperg\'eom\'etriques de plusieurs variables. J. Math. Pures et Appl. 9 (1930), 1--34.
\bibitem{O-1} Oshima, T.: Generalized Capelli identities and boundary value problems for ${\rm GL}(n)$. Structure of solutions of differential equations (Katata/Kyoto, 1995), 307--335, World Sci. Publ., 1996. 
\bibitem{O-2} Oshima, T.: A quantization of conjugacy classes of matrices. Adv. Math. 196 (2005), no. 1, 124--146. 
\bibitem{O-4} Oshima, T.: Whittaker model of degenerate principal series. Representation theory of groups and extension of harmonic analysis (Japanese). RIMS K\^oky\^uroku No. 1467 (2006), 71--78.
\bibitem{O-3} Oshima, T.: Annihilators of generalized Verma modules of the scalar type for classical Lie algebras. Harmonic analysis, group representations, automorphic forms and invariant theory, 277--319, Lect. Notes Ser. Inst. Math. Sci. Natl. Univ. Singap., 12, World Sci. Publ., 2007.
\bibitem{O-S} Oshima, T.: Heckman-Opdam hypergeometric functions and their specializations, Harmonische Analysis und Darstellungstheorie Topologischer Gruppen, Mathematisches Forschungsinstitut Oberwolfach, Report 49(2007), 38-40.
\bibitem{Sad} Sadykov, T.: On the Horn system of partial differential equations and series of hypergeometric type. Math. Scand. 91 (2002), no. 1, 127--149.
\bibitem{Sat} Sato, M.: Theory of prehomogeneous vector spaces. Nagoya Math. J. 120 (1990), 1--34. 
\bibitem{Sch} Schmid, W.: Boundary value problems for group invariant differential equations. The mathematical heritage of \'Elie Cartan. Ast\'erisque 1985, 311--321.
\bibitem{T-1} Terras, A.: Bessel series expansions of the Epstein zeta function and the functional equation. Trans. Amer. Math. Soc. 183 (1973), 477--486. 
\bibitem{T-2} Terras, A.: The Chowla-Selberg method for Fourier expansion of higher rank Eisenstein series. Canad. Math. Bull. 28 (1985), no. 3, 280--294. 
\bibitem{W-1} Wallach, N.: Lie algebra cohomology and holomorphic continuation of generalized Jacquet integrals. Representations of Lie groups, Kyoto, Hiroshima, 1986, 123--151, Adv. Stud. Pure Math., 14, Academic Press, Boston, 1988. 
\bibitem{W-2} Wallach, N.: Holomorphic continuation of generalized Jacquet integrals for degenerate principal series.  Represent. Theory  10  (2006), 380--398
\bibitem{Y-3} Yamashita, H.: On Whittaker vectors for generalized Gelfand-Graev representations of semisimple Lie groups. J. Math. Kyoto. Univ. 26-2 (1986), 263--298.
\bibitem{Y-1} Yamashita, H.: Embeddings of discrete series into induced representations of semisimple Lie groups. I. General theory and the case of ${\rm SU}(2,2)$. Japan. J. Math. (N.S.) 16 (1990), no. 1, 31--95. 
\bibitem{Y-2} Yamashita, H.: Cayley transform and generalized Whittaker models for irreducible highest weight modules. Nilpotent orbits, associated cycles and Whittaker models for highest weight representations. Ast\'erisque No. 273 (2001), 81--137. 
\end{thebibliography}
\end{document}